\documentclass[a4paper,12pt]{amsart}
\usepackage{yhmath}
\usepackage{mathtools} 
\usepackage{graphicx}
\usepackage{mathtools}
\usepackage{mathrsfs}
\usepackage{amsmath,amssymb}
\usepackage{amsthm}
\usepackage{amsfonts}
\usepackage{dsfont}
\usepackage{fancyhdr}
\usepackage{harpoon}
\usepackage{amssymb}
\usepackage{enumerate}
\usepackage[bookmarks=true,colorlinks,linkcolor=black]{hyperref}
\usepackage{geometry}
\usepackage{extarrows}
\usepackage{tikz}
\usetikzlibrary{decorations.markings}
\usetikzlibrary{arrows}
\usepackage{blkarray}
\usepackage{multirow}
\usepackage{color}
\usepackage{float}
\usepackage[utf8]{inputenc}
\usepackage{chngcntr}
\usepackage{etoolbox}
\usepackage{tikz-cd}
\usepackage[xindy]{imakeidx}
\usepackage{blkarray}
\usepackage{multirow}
\usepackage{cancel}
\usepackage{array}
\usepackage{longtable}
\usepackage{float}
\usepackage[utf8]{inputenc}
\usepackage{chngcntr}
\usepackage{etoolbox}
\usepackage{caption}
\usepackage{stackrel}

\usepackage{diagbox} 
\usepackage{hyperref}
\hypersetup{  
    colorlinks=true,
    linkcolor=red,
    filecolor=magenta,      
    urlcolor=cyan,
    pdftitle={Overleaf Example},
    pdfpagemode=FullScreen,
    }

\newtheorem{thm}{Theorem}[section]
\newtheorem{cor}[thm]{Corollary}
\newtheorem{conj}[thm]{Conjecture}
\newtheorem{cla}[thm]{Claim}
\newtheorem{lem}[thm]{Lemma}
\newtheorem{prop}[thm]{Proposition}
\theoremstyle{definition}
\newtheorem{defn}[thm]{Definition}
\newtheorem{ex}[thm]{Example}
\newtheorem{rem}[thm]{Remark}
\theoremstyle{definition}

\numberwithin{equation}{section}

\DeclareMathOperator{\Area}{Area}

\DeclareMathOperator{\vol}{vol}

\DeclareMathOperator{\Aut}{Aut}

\DeclareMathOperator{\Mod}{Mod}

\DeclareMathOperator{\supp}{supp}

\DeclareMathOperator{\Aff}{Aff}

\DeclareMathOperator{\SL}{SL}
\DeclareMathOperator{\Sp}{Sp}

\DeclareMathOperator{\GL}{GL}

\DeclareMathOperator{\inj}{inj}

\DeclareMathOperator{\hol}{hol}

\DeclareMathOperator{\Id}{Id}

\DeclareMathOperator{\Jac}{Jac}

    {\begin{center}%
     \textsc{Acknowledgements}
\end{center}}%
    {\null}
\begin{document}

%!TEX program = xelatex

\title{Effective Density of surfaces near Teichm\"{u}ller curves}
\author{Siyuan Tang}%
\address{B\MakeLowercase{eijing} I\MakeLowercase{nternational} C\MakeLowercase{enter} \MakeLowercase{for} M\MakeLowercase{athematical} R\MakeLowercase{esearch}, P\MakeLowercase{eking} U\MakeLowercase{niversity}, B\MakeLowercase{eijing}, 100871}
\email{1992.siyuan.tang@gmail.com,  siyuantang@pku.edu.cn}
\maketitle

\begin{abstract}
 We study the dynamics of $\SL_{2}(\mathbb{R})$ on the stratum of translation surfaces $\mathcal{H}(2)$.  Especially, we obtain   effective density theorems on  $\mathcal{H}(2)$ for orbits of the upper triangular subgroup $P$ of $\SL_{2}(\mathbb{R})$ with the based surfaces near a small Teichm\"{u}ller curve.
  
  The proof is based on the use of   McMullen's  classification theorem, together with the effective equidistribution theorems in homogeneous dynamics.   
  In particular, we compare the $P$-orbit of a surface, and the $P$-orbit of its absolute periods using   the  Lindenstrauss-Mohammadi-Wang's effective equidistribution theorem.
\end{abstract}

\tableofcontents

\section{Introduction}
\subsection{Main results}
Let $\mathcal{H}(2)$  be the moduli space of  translation surfaces of type $2$. Thus, for $x\in  \mathcal{H}(2)$, $x$ is a translation surface of genus $2$, and $x$ has a double zero. In practice, we may identify $x$ as a pair $(M,\omega)$, consisting of a compact Riemann surface $M$ of genus $2$, and a holomorphic $1$-form $\omega$ on $M$. The space $\mathcal{H}(2)$  can be identified with the quotient of the Teichm\"{u}ller space $\mathcal{TH}(2)$ and the mapping class group $\Mod$:
\[\mathcal{H}(2)=\mathcal{TH}(2)/\Mod.\]
 Let $\mathcal{H}_{A}(2)$  be the subset of  translation surfaces of area $A$. 

Given a translation surface $x=(M,\omega)\in\mathcal{H}_{1}(2)$, the \textit{absolute period map}\index{absolute period map} (abuse  notation) $x:H_{1}(M;\mathbb{Z})\rightarrow\mathbb{C}$ is defined by 
\[x(\gamma)\coloneqq\int_{\gamma}\omega\]
whose image $x(H_{1}(M;\mathbb{Z}))$ is the group of \textit{absolute periods}\index{absolute period} of $x$. Via the absolute period map, $\mathcal{H}(2)$ can be locally identified with $H^{1}(M;\mathbb{C})$.

In \cite[\S7]{mcmullen2007dynamics}, McMullen showed that each $x\in \mathcal{H}_{1}(2)$ can be presented as a connected sum of two tori. More precisely, for $x\in\mathcal{H}_{1}(2)$, there are a pair of lattices $\Lambda_{1},\Lambda_{2}\subset x(H_{1}(M;\mathbb{Z}))\subset\mathbb{C}$, and a vector $v\in \mathbb{C}^{\ast}$, such that 
\[  [0,v]\cap\Lambda_{1}=\{0\},\ \ \  [0,v]\cap\Lambda_{2}=\{0,v\}\]
or vice versa. Then $E_{i}=\mathbb{C}/\Lambda_{i}$ are forms of genus $1$ (i.e. tori). The arcs  $J_{i}=[0,v]\subset E_{i}$ are straight geodesics on $E_{i}$; in particular, we get a loop in $E_{2}$ (or vice versa). Slitting the two tori $E_{i}$ open along $J_{i}$, and gluing corresponding edges using $J_{i}$, we obtain the \textit{connected sum}\index{connected sum}
\[x=  E_{1}\stackrel[{[0,v]}]{}{\#} E_{2}.\]
We also write the connected sum as $x=\Lambda_{1}\stackrel[{[0,v]}]{}{\#} \Lambda_{2}$, and call it a \textit{splitting}\index{splitting} of $x$. 

Moreover,   for $g\in \GL_{2}^{+}(\mathbb{R})$, we define  
\[ gx= g.\Lambda_{1}\stackrel[g\cdot {[0,v]}]{}{\#} g. \Lambda_{2}.\]
It leads to a $\GL^{+}_{2}(\mathbb{R})$-action on $\mathcal{H}(2)$, and a $\SL_{2}(\mathbb{R})$-action on  $\mathcal{H}_{1}(2)$. Further, there is a finite $\SL_{2}(\mathbb{R})$-measure, called \textit{Masur-Veech  measure}\index{Masur-Veech  measure}, on $\mathcal{H}_{1}(2)$. Thus, we study the dynamical properties of this action.

Let $G=\SL_{2}(\mathbb{R})$, and
\[a_{t}=\left[
            \begin{array}{cccc}
   e^{\frac{1}{2}t} & 0  \\
   0 & e^{-\frac{1}{2}t}   \\
            \end{array}
          \right],  \ \ \ u_{r}=\left[
            \begin{array}{cccc}
   1 & r  \\
   0 & 1   \\
            \end{array}
          \right].\]
For any subset $I\subset\mathbb{R}$, denote $a_{I}=\{a_{t}:t\in I\}$, and  $u_{I}=\{u_{r}:r\in I\}$. Let $P=a_{\mathbb{R}}u_{\mathbb{R}}$. The seminal work of Eskin and Mirzakhani in \cite{eskin2018invariant} shows that every $P$-orbit closure in $\mathcal{H}_{1}(2)$ is in fact $G$-invariant. Further, by work of McMullen \cite{mcmullen2007dynamics}, the $G$-orbit closure in $\mathcal{H}_{1}(2)$ are either Teichm\"{u}ller curves or $\mathcal{H}_{1}(2)$ itself  (see Theorem \ref{effective2024.07.02}). In \cite{mcmullen2005teichmullerDiscriminant}, a much more detailed description for Teichm\"{u}ller curves in $\mathcal{H}_{1}(2)$ is available. In this paper, we are mainly interested in the Teichm\"{u}ller curve $\Omega W_{D}$ generated by square-tiled surfaces (i.e. the discriminant $D$ is a square, see Section \ref{effective2024.9.2}).

 For $\tilde{x}\in\mathcal{TH}(2)$, Let $\|\cdot\|_{\tilde{x}}$ be the \textit{Avila-Gou\"{e}zel-Yoccoz norm}\index{Avila-Gou\"{e}zel-Yoccoz norm} (or \textit{AGY norm}\index{AGY norm} for short) on $H^{1}(M;\mathbb{C})$,  (see Definition \ref{effective2024.08.13}). It induces a metric $d$ on $\mathcal{H}_{1}(2)$. 
 
 For $x\in\mathcal{H}_{1}(2)$, let $\ell(x)$ be the shortest length of a saddle connection. For $\eta>0$, define
 \[      \mathcal{H}^{(\eta)}_{1}(2)\coloneqq\{x\in\mathcal{H}_{1}(2):\ell(x)\geq\eta\}
\]  
  (see (\ref{effective2024.08.12})). Similarly, for a pair of lattices $\Lambda=(\Lambda_{1},\Lambda_{2})\in X$, let  $\ell(\Lambda)$ be the shortest length of a vector in $\Lambda_{1}$ and $\Lambda_{2}\subset\mathbb{C}$, and let 
   \[X_{\eta}\coloneqq\{\Lambda\in X:\ell(\Lambda)\geq\eta\}.\] 
 
 In this paper, we will show: 
 \begin{thm}\label{effective2024.9.1} For \hypertarget{2024.08.C1}  any $\eta>0$, $L>0$, there \hypertarget{2024.08.k1}  exists $\kappa_{1}>0$, $C_{1}=C_{1}(\eta)>0$ such that for $D\geq L^{\hyperlink{2024.08.k1}{\kappa_{1}}}$, and   
\begin{equation}\label{effective2024.07.156}
t>\hyperlink{2024.08.C1}{C_{1}}D^{\hyperlink{2024.08.k1}{\kappa_{1}}},
\end{equation}
 the following holds: Suppose that  
$x\in\mathcal{H}_{1}(2)$   satisfying
\[ \inf_{y_{D}\in (\Omega_{1}W_{D})_{\eta}}d(y_{D},x)< e^{- t}.\] 
Then we have 
\[  d(z, a_{t}u_{[0,1]}x)  \leq L^{-1} \]
 for every $z\in \mathcal{H}_{1}^{(L^{-1})}(2)$.
 \end{thm}

The speed of density  (\ref{effective2024.07.156}) is in fact related to the spectral gaps of Teichm\"{u}ller curves in $\mathcal{H}_{1}(2)$.
McMullen made the following conjecture:
 \begin{conj}[McMullen's expansion conjecture]\label{effective2024.07.144}
   The family of graphs associated to arithmetic Veech groups in $\mathcal{H}(2)$ is expander. In other words, the spectral gaps of the arithmetic Veech groups possess a uniform lower bound.
 \end{conj}
If Conjecture \ref{effective2024.07.144} is correct,  then (\ref{effective2024.07.156}) in Theorem \ref{effective2024.9.1}   can be improved to 
\begin{equation}\label{effective2024.07.157}
 t>\hyperlink{2024.08.C1}{C_{1}}\log D.
\end{equation}
In other words, we would have that $a_{t}u_{[0,1]}x$ is dense with a polynomial error rate.
 
 Roughly speaking, Theorem \ref{effective2024.9.1} indicates that if a point $x$ is extremely close to a given Teichm\"{u}ller curve $\Omega_{1}W_{D}$, then its $P$-orbit $Px$ is $D^{\frac{1}{\hyperlink{2024.08.k1}{\kappa_{1}}}}$-dense in $\mathcal{H}_{1}(2)$.

 Let $x\in\mathcal{H}_{1}(2)$ with a splitting $x=\Lambda_{1}\stackrel[{[0,v]}]{}{\#} \Lambda_{2}$. Let $\Area(\Lambda_{i})$ denote the covolume of $\Lambda_{i}$ (i.e. the area of $E_{i}=\mathbb{C}/\Lambda_{i}$). Let $G=\SL_{2}(\mathbb{R})$, $\Gamma=\SL_{2}(\mathbb{Z})$, $X=G/\Gamma\times G/\Gamma$. Locally speaking,  $x=\Lambda_{1}\stackrel[{[0,v]}]{}{\#} \Lambda_{2}\in\mathcal{H}_{1}(2)$ is dominated by the ratio of areas $\Area(\Lambda_{1})/\Area(\Lambda_{2})$ and the shape of the lattices $(\Lambda_{1},\Lambda_{2})\in X$ (after forgetting the areas). By work of McMullen \cite{mcmullen2007dynamics}, the Teichm\"{u}ller curves are only determined by the absolute periods (see Theorem \ref{effective2023.9.12}). 
 For an even square $D>0$, let $Q_{D}=G.(\Lambda_{1},\Lambda_{2})\subset X$ be the space of  absolute periods of $\Omega_{1} W_{D}$, where $\Lambda_{1}\stackrel[{[0,v]}]{}{\#} \Lambda_{2}\in \Omega W_{D}$ is a splitting with $\Area(\Lambda_{1})=\Area(\Lambda_{2})=\frac{1}{2}$ (see Section \ref{effective2024.9.6}).

 Now  given $D=4d^{2}>0$, we consider a surface $x=\Lambda_{1}\stackrel[{[0,v]}]{}{\#} \Lambda_{2}$ with   $(\Lambda_{1},\Lambda_{2})\in Q_{D}$. Then $x\in \Omega_{1} W_{D}$. It follows that  the $P$-orbit $Px\subset\Omega_{1}W_{D}$. In particular, it meets the requirement of Theorem \ref{effective2024.9.1}.  Note however that $Px$ cannot be extremely close to another Teichm\"{u}ller curve  $\Omega_{1} W_{D^{\prime}}$ ($D^{\prime}\neq D$).
 
 On the other hand, if $(\Lambda_{1},\Lambda_{2})\not\in Q_{D}$ for any $D>0$, then $\overline{Px}=\mathcal{H}_{1}(2)$. It means  that $a_{t}u_{[0,1]}x$ can be extremely close to any Teichm\"{u}ller curve $\Omega_{1}W_{D}$, for sufficiently large $t$ depending on $D$. 
 Also, recall that for any surface $x\in\mathcal{H}_{1}(2)$ with  $\overline{Px}=\mathcal{H}_{1}(2)$, there is a splitting $x=\Lambda_{1}\stackrel[{[0,v]}]{}{\#} \Lambda_{2}$ so that $\overline{P.(\Lambda_{1},\Lambda_{2})}= X$. 
 
 Therefore, we observe that the absolute periods $(\Lambda_{1},\Lambda_{2})$ of $x$ connect to the behavior of $Px$. 
 This enlightens us about using the effective results on $X$ to study the density on $\mathcal{H}_{1}(2)$. In \cite{lindenstrauss2023polynomial,lindenstrauss2022effective}, Lindenstrauss, Mohammadi, and Wang established the effective density and equidistribution of $P$-orbits in $X$. Let $\|\cdot\|$ be a norm (e.g. the maximum norm) on $X$. For $D=4d^{2}>0$, $\varrho>0$, $\eta>0$, consider 
 \[J_{d,t}(\varrho)=\left\{ r\in [0,1]: \|(\tilde{\Lambda}_{1},\tilde{\Lambda}_{2})-a_{t}u_{r}(\Lambda_{1},\Lambda_{2})\|\leq  \varrho,\ (\tilde{\Lambda}_{1},\tilde{\Lambda}_{2})\in (Q_{D})_{\eta}\right\}.\]  
 Then the effective equidistribution (Theorem \ref{effective2024.07.107}) implies that 
 \[|J_{d,t}(\varrho)|\geq \frac{1}{2}\varrho^{3}\]
 for any sufficiently large $t$, then $a_{t}u_{r}(\Lambda_{1},\Lambda_{2})$ can be $\varrho$-close to $(Q_{D})_{\eta}$ for $r\in J_{d,t}(\varrho)$.

 We try to quantify the above observation.  
Let    $\varkappa:\mathbb{R}^{+}\rightarrow \mathbb{R}^{+}$ be a monotonic decreasing function so that 
\[\lim_{t\rightarrow\infty}\varkappa(e^{t})=0.\]

Let   $J\subset[0,e^{t}]$ be an interval, and
 \begin{align}
J^{\prime}(\Lambda_{1},\Lambda_{2},D,t,\eta,\varrho,J)  &=   \left\{r\in J: 
\begin{tabular}{m{5.3cm}}
  $\|(\tilde{\Lambda}_{1},\tilde{\Lambda}_{2})-u_{r}a_{t}(\Lambda_{1},\Lambda_{2})\|\leq  \varrho$, \\ for some $(\tilde{\Lambda}_{1},\tilde{\Lambda}_{2})\in (Q_{D})_{\eta}$
\end{tabular}
 \right\} ,\;\nonumber\\ 
 J^{\prime\prime}(x,\vartheta,D,t,\eta,\varkappa,J)   & = \left\{r\in J: 
\begin{tabular}{m{7.5cm}}
  $  \ell(u_{r}a_{t}x)^{\vartheta}  \|(\tilde{\Lambda}_{1},\tilde{\Lambda}_{2})-u_{r}a_{t}x\|_{u_{r} a_{t}x}\leq \varkappa(|J|)$, \\ for some $(\tilde{\Lambda}_{1},\tilde{\Lambda}_{2})\in (Q_{D})_{\eta}$
\end{tabular}
 \right\} .\;  \nonumber
\end{align} 
Roughly speaking, $J^{\prime}$ indicates the moments when  $u_{r}a_{t}(\Lambda_{1},\Lambda_{2})$ is close to $Q_{D}$ in homogeneous dynamics, and $J^{\prime\prime}$ implies the moments when  $u_{r}a_{t}x$ is close to $W_{D}$ in Teichm\"{u}ller dynamics.

\begin{thm}\label{effective2024.07.148} 
  There exists an absolute $\delta_{0}>0$, $C_{1}>0$ so that the following holds.
 Suppose that  $x\in \mathcal{H}_{1}(2)$ has a splitting $x=\Lambda_{1}\stackrel[{[0,v]}]{}{\#} \Lambda_{2}$ satisfying $\Area(\Lambda_{1})=\Area(\Lambda_{2})$. Then there exists  $L_{0}=L_{0}(\delta,\ell(x),\ell(\Lambda_{1},\Lambda_{2}))>0$, and an $R>0$ such that the following holds.
 For every $\delta\in(0,\delta_{0})$,   $L>L_{0}$,   $\eta>0$, $D=4d^{2}\geq L^{\hyperlink{2024.08.k1}{\kappa_{1}}}$, $\varrho>0$, $\vartheta>0$, and   function $\varkappa_{\vartheta}:\mathbb{R}^{+}\rightarrow \mathbb{R}^{+}$ with 
\begin{equation}\label{effective2024.10.12}
\lim_{t\rightarrow\infty}\varkappa_{\vartheta}(e^{t})=0,
\end{equation} 
let $\xi=\hyperlink{2024.08.k1}{\kappa_{1}}+\vartheta$, and  
\[\varkappa_{\vartheta}^{\star}:r\mapsto\underbrace{\varkappa_{\vartheta}( \cdots \varkappa_{\vartheta}(\varkappa_{\vartheta}(r)^{-\xi^{2}})^{-\xi^{2}}\cdots ) ^{\xi^{2}}  }_{R^{2} \text{ copies}}. \]
 Then there exists an interval  
 $J^{\star\star}\subset[0,e^{t}]$ with $|J^{\star\star}|\geq(\varkappa_{\vartheta}^{\star}(e^{t}))^{-1}$, such that for 
   \begin{itemize} 
   \item $J^{\prime}=J^{\prime}(\Lambda_{1},\Lambda_{2},D,t,\eta^{R^{2}+1},\varrho,J^{\star\star})$,
   \item $J^{\prime\prime}=J^{\prime\prime}(x,\vartheta,D,t,\eta^{R^{2}+1},\varkappa_{\vartheta},J^{\star\star})$,
   \item $t>0$ so that  
   \begin{equation}\label{effective2024.10.15}
   \log \varkappa_{\vartheta}^{\star}(e^{t})^{-1}> \hyperlink{2024.08.C1}{C_{1}}D^{\hyperlink{2024.08.k1}{\kappa_{1}}} \geq  \hyperlink{2024.08.C1}{C_{1}} L^{\hyperlink{2024.08.k1}{\kappa_{1}}^{2}},
   \end{equation} 
 (or $   \varkappa_{\vartheta}^{\star}(e^{t})^{-1}>  D^{\hyperlink{2024.08.C1}{C_{1}}}\geq L^{\hyperlink{2024.08.C1}{C_{1}}\hyperlink{2024.08.k1}{\kappa_{1}} }$ assuming Conjecture \ref{effective2024.07.144} is correct), 
 \end{itemize}
 at least one of the following holds:

   \begin{enumerate}[\ \ \ (1)] 
     \item  $\left|J^{\prime}\right|\geq \frac{1}{2}   \varrho^{3} |J^{\star\star}|$, and for every $z\in \mathcal{H}_{1}^{(L^{-1})}(2)$, we have
     \[d(z,a_{2t}u_{[0,2]}x)\leq  L^{-1}.\]
     \item There exists a pair of lattices $(\Lambda_{1}^{\prime},\Lambda_{2}^{\prime})\in X$ such that $G.(\Lambda_{1}^{\prime},\Lambda_{2}^{\prime})$ is periodic with $\vol(G.(\Lambda_{1}^{\prime},\Lambda_{2}^{\prime}))\leq e^{\delta t}$ and
         \[\|(\Lambda_{1}^{\prime},\Lambda_{2}^{\prime})-(\Lambda_{1},\Lambda_{2})\|\leq e^{-\frac{1}{2}t}.\]  
  \item   $\left|J^{\prime}\right|\geq \frac{1}{2}    \varrho^{3} |J^{\star\star}|$, and  
    \[   \ell(u_{r}a_{t}x)\geq  (\varkappa_{\vartheta}(|J^{\star\star}|))^{\xi}, \ \ \  J^{\prime}\cap J^{\prime\prime}=\emptyset.  \]  
   \end{enumerate} 
\end{thm}

Next, we discuss the case (2) of Theorem \ref{effective2024.07.148}. Let $y_{1}=(N_{1},\omega_{1})$, $y_{2}=(N_{2},\omega_{2})$ be two forms of genus $1$ (i.e. tori). Then we say that $x=(M,\omega)$ can be presented as an \textit{algebraic sum}\index{algebraic sum} $(N_{1},\omega_{1})+(N_{2},\omega_{2})$ if by a symplectic isomorphism
\[H_{1}(M;\mathbb{Z})\cong H_{1}(N_{1};\mathbb{Z})\oplus H_{1}(N_{2};\mathbb{Z}),\]
 we have
\[x=y_{1}+y_{2}\]
on passing to cohomology with coefficients in $\mathbb{C}$. Two algebraic sums of $x$ are \textit{equivalent} if they come from the same unordered splitting $H_{1}(M)\cong H_{1}(N_{1})\oplus H_{1}(N_{2})$.

 Note that the $1$-forms of an algebraic sum $(M,\omega)\cong(N_{1},\omega_{1})+(N_{2},\omega_{2})$ are uniquely determined by the splitting 
\[H_{1}(M;\mathbb{Z})= H_{1}(N_{1};\mathbb{Z})\oplus H_{1}(N_{2};\mathbb{Z}).\]
 In fact, we have $(N_{i},\omega_{i})=(\mathbb{C}/\Lambda_{i},dz)$ where $\Lambda_{i}=\omega_{i}(H_{1}(N_{i};\mathbb{Z}))$. Note also that any connected sum of $x=  E_{1}\stackrel[{[0,v]}]{}{\#} E_{2}$ gives rise to a natural isomorphism $H_{1}(M;\mathbb{Z})=H_{1}(E_{1};\mathbb{Z})\oplus H_{1}(E_{2};\mathbb{Z})$, and so an algebraic splitting $x=E_{1}+E_{2}$.

Since the set of surfaces  $x\in \mathcal{H}_{1}(2)$ having a splitting with equal areas is dense in $\mathcal{H}_{1}(2)$. With a bit more effort, we obtain an alternation of the case (2) of Theorem \ref{effective2024.07.148}:
\begin{thm}\label{effective2024.07.159}
Let the notation be as in Theorem \ref{effective2024.07.148}. Then the case (2) of Theorem \ref{effective2024.07.148} can be replaced by
   \begin{enumerate}[\ \ \ (2)$^{\prime}$] 
     \item There exists a surface $x^{\prime\prime}\in\mathcal{H}_{1}(2)$ such that 
       \begin{itemize}
         \item  $d(x^{\prime\prime},x)<e^{-\frac{1}{4}t}$,
         \item it can be presented as an algebraic sum $x^{\prime\prime}=\Lambda^{\prime\prime}_{1}+\Lambda^{\prime\prime}_{2}$ such that $\Area(\Lambda^{\prime\prime}_{1})=\Area(\Lambda^{\prime\prime}_{2})$, and $G.(\Lambda^{\prime\prime}_{1},\Lambda^{\prime\prime}_{2})$ is periodic with $\vol(G.(\Lambda^{\prime\prime}_{1},\Lambda^{\prime\prime}_{2}))\leq e^{\delta t}$.
       \end{itemize}  
   \end{enumerate} 
\end{thm}

Finally, we shall deduce a criterion  for the Teichm\"{u}ller curves in $\mathcal{H}(2)$ via Ratner's theorem. More precisely, Ratner's theorem gives us certain rigid information of periodic orbits in $X$. Suppose that there is  a surface $x^{\prime\prime}\in\mathcal{H}_{1}(2)$ that can be presented as an algebraic sum $x^{\prime\prime}=\Lambda^{\prime\prime}_{1}+\Lambda^{\prime\prime}_{2}$ such that   $G.(\Lambda^{\prime\prime}_{1},\Lambda^{\prime\prime}_{2})$ is periodic. If we put some additional rationality on $\Area(\Lambda^{\prime\prime}_{1})/\Area(\Lambda^{\prime\prime}_{2})$, then by a criterion obtained in \cite[\S6]{mcmullen2007dynamics}, $x^{\prime\prime}$ generates a Teichm\"{u}ller curve.  Moreover, a detailed analysis about the quantities leads to a discriminant control of this Teichm\"{u}ller curve.

Then together with Theorem \ref{effective2024.07.159}, we obtain the following density theorem that is close to \cite[Theorem 1.1]{lindenstrauss2023polynomial} in  homogeneous dynamics (see also the effective equidistribution (Theorem \ref{effective2024.07.107})).
\begin{thm}\label{effective2024.08.1} 
Let the notation be as in Theorem \ref{effective2024.07.148}. \hypertarget{2024.08.k2} There exists an absolute constant $\kappa_{2}>0$  such that the following holds. For any $\delta\in(0,\delta_{0})$,   $x\in \mathcal{H}_{1}(2)$,  $L>L_{0}(\delta,\frac{1}{2}\ell(x),\frac{1}{2}\ell(x))$, and sufficiently large $t>0$ as in (\ref{effective2024.10.15}), we have $|J^{\star\star}|\geq(\varkappa_{\vartheta}^{\star}(e^{t}))^{-1}$, and
 at least one of the   following holds:  
     \begin{enumerate}[\ \ \ (1)] 
     \item  $\left|J^{\prime}\right|\geq \frac{1}{2}   \varrho^{3} |J^{\star\star}|$, and for every $z\in \mathcal{H}_{1}^{(L^{-1})}(2)$, we have
     \[d(z,a_{2t}u_{[0,2]}x)\leq  L^{-1}.\]
     \item There exists a surface $x^{\prime\prime\prime}\in\mathcal{H}_{1}(2)$ such that 
       \begin{itemize}
         \item  $d(x^{\prime\prime\prime},x)<e^{-\frac{1}{4}t}$,
         \item $x^{\prime\prime\prime}$ generates a Teichm\"{u}ller curve with discriminant $D$ so that 
 \begin{itemize}
           \item  either $D<   e^{\hyperlink{2024.08.k2}{\kappa_{2}}\delta t}$,
           \item or $D$ is a square and $D<  e^{\hyperlink{2024.08.k2}{\kappa_{2}}t}$.
         \end{itemize}
       \end{itemize} 
 \item   $\left|J^{\prime}\right|\geq \frac{1}{2}    \varrho^{3} |J^{\star\star}|$, and  
    \[   \ell(u_{r}a_{t}x)\geq  (\varkappa_{\vartheta}(|J^{\star\star}|))^{\xi}, \ \ \  J^{\prime}\cap J^{\prime\prime}=\emptyset.  \]  
   \end{enumerate} 
\end{thm}

It is interesting to know whether the case (3) really happens for any function $\varkappa_{\vartheta}:\mathbb{R}^{+}\rightarrow \mathbb{R}^{+}$ with 
\[\lim_{t\rightarrow\infty}\varkappa_{\vartheta}(e^{t})=0.\]
As shown in \cite{forni2021limits}, there exists a set $Z\subset\mathbb{R}$ of zero upper density such that
\begin{equation}\label{effective2024.10.16}
\lim_{t\not\in Z} \int_{0}^{1} f(a_{t}u_{r}x)dr =\int fd\mu
\end{equation}
for $f\in C_{c}(\mathcal{H}_{1}(2))$,
where $\mu$ is a $G$-invariant measure on $\mathcal{H}_{1}(2)$. And Forni conjectured that (\ref{effective2024.10.16}) holds with $Z=\emptyset$. The existence of $\varkappa_{\vartheta}$ so that Theorem \ref{effective2024.08.1}(3) does not occur is linked to this problem.

Recently, Rached established the closing lemma for the $P$-action on $\mathcal{H}_{1}(2)$  \cite{rached2024separation}. Note that the closing lemma in the homogeneous dynamics serves as the ``initial dimension phase" for the proofs of effective density and equidistribution \cite{lindenstrauss2023polynomial,lindenstrauss2022effective}. It would be interesting to know if there is a way to improve the dimension to the unstable direction, and then apply the result of Sanchez \cite{sanchez2023effective}  to get effective equidistribution of $Px$ provided the surface $x$ not too close to a small Teichm\"{u}ller curve, similar to the proofs in \cite{lindenstrauss2023polynomial,lindenstrauss2022effective}.

\subsection{Outline of the proof of Theorem \ref{effective2024.9.1}}

For simplicity, we assume that $z\in \mathcal{H}_{1}^{(L^{-1})}(2)$  has a splitting $z=\Lambda_{1}(z)\stackrel[I(z)]{}{\#} \Lambda_{2}(z)$ satisfies $\Area(\Lambda_{1}(z))=\Area(\Lambda_{2}(z))$, so that we may ignore the error coming from the areas of tori. 

Then for a sufficiently large discriminant $D$, the Teichm\"{u}ller curve $\Omega_{1}W_{D}$ is $L^{-1}$-close to  $z$ (see Figure \ref{effective2024.08.4}). 
In fact, there exists $z_{D}=\Lambda_{1}(z_{D})\stackrel[I(z_{D})]{}{\#} \Lambda_{2}(z_{D})\in \Omega_{1}W_{D}$ with $\Area(\Lambda_{1}(z_{D}))=\Area(\Lambda_{2}(z_{D}))$, such that 
\begin{equation}\label{effective2024.08.7}
  d(z,z_{D})\asymp\|(\Lambda_{1}(z),\Lambda_{2}(z))-(\Lambda_{1}(z_{D}),\Lambda_{2}(z_{D}))\|_{z}<L^{-1}.
\end{equation}

Fix some $\eta>0$. Then on this Teichm\"{u}ller curve $\Omega_{1}W_{D}$, for any $y\in (\Omega_{1}W_{D})_{\eta}$, we apply the effective equidistribution on  $\Omega_{1}W_{D}$ and obtain that   for  any $s$, there exists some   $r^{\prime}\in[0,1]$ such that 
\begin{equation}\label{effective2024.08.5}
d(z_{D}, a_{s}u_{r^{\prime}}y)\leq e^{-\aleph_{1} s} 
\end{equation}  
for some $\aleph_{1}>0$ (depending on $D$). 
Then  for  sufficiently large $s$ (depending on $D$), the right hand side is $<L^{-1}$.

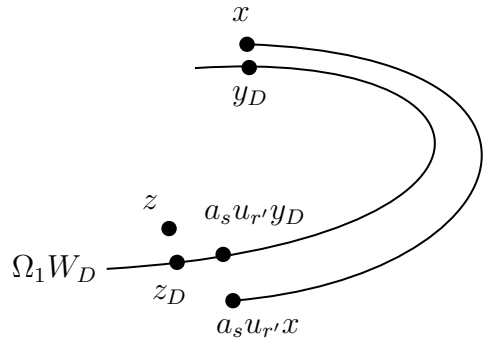
\begin{figure}[H]
\centering

\tikzset{every picture/.style={line width=0.75pt}} %set default line width to 0.75pt        

\begin{tikzpicture}[x=0.75pt,y=0.75pt,yscale=-1,xscale=1]
%uncomment if require: \path (0,505); %set diagram left start at 0, and has height of 505

%Curve Lines [id:da32697035227550697] 
\draw    (78.9,238.29) .. controls (292.9,225.29) and (285.9,126.29) .. (122.9,137.29) ;
%Curve Lines [id:da8380594373408679] 
\draw    (141.9,254.29) .. controls (285.9,242.29) and (324.9,131.29) .. (148.9,125.29) ;
\draw [shift={(148.9,125.29)}, rotate = 181.95] [color={rgb, 255:red, 0; green, 0; blue, 0 }  ][fill={rgb, 255:red, 0; green, 0; blue, 0 }  ][line width=0.75]      (0, 0) circle [x radius= 3.35, y radius= 3.35]   ;
\draw [shift={(141.9,254.29)}, rotate = 355.24] [color={rgb, 255:red, 0; green, 0; blue, 0 }  ][fill={rgb, 255:red, 0; green, 0; blue, 0 }  ][line width=0.75]      (0, 0) circle [x radius= 3.35, y radius= 3.35]   ;
%Straight Lines [id:da5075055383293023] 
\draw    (110,218) ;
\draw [shift={(110,218)}, rotate = 0] [color={rgb, 255:red, 0; green, 0; blue, 0 }  ][fill={rgb, 255:red, 0; green, 0; blue, 0 }  ][line width=0.75]      (0, 0) circle [x radius= 3.35, y radius= 3.35]   ;
%Straight Lines [id:da24442212746847836] 
\draw    (137,231) ;
\draw [shift={(137,231)}, rotate = 0] [color={rgb, 255:red, 0; green, 0; blue, 0 }  ][fill={rgb, 255:red, 0; green, 0; blue, 0 }  ][line width=0.75]      (0, 0) circle [x radius= 3.35, y radius= 3.35]   ;
%Straight Lines [id:da9345028774157227] 
\draw    (150,137) ;
\draw [shift={(150,137)}, rotate = 0] [color={rgb, 255:red, 0; green, 0; blue, 0 }  ][fill={rgb, 255:red, 0; green, 0; blue, 0 }  ][line width=0.75]      (0, 0) circle [x radius= 3.35, y radius= 3.35]   ;
%Straight Lines [id:da0030772468892708016] 
\draw    (114,235) ;
\draw [shift={(114,235)}, rotate = 0] [color={rgb, 255:red, 0; green, 0; blue, 0 }  ][fill={rgb, 255:red, 0; green, 0; blue, 0 }  ][line width=0.75]      (0, 0) circle [x radius= 3.35, y radius= 3.35]   ;

% Text Node
\draw (140,105) node [anchor=north west][inner sep=0.75pt]    {$x$};
% Text Node
\draw (30,230) node [anchor=north west][inner sep=0.75pt]    {$\Omega_{1}W_{D}$};
% Text Node
\draw (95,200) node [anchor=north west][inner sep=0.75pt]    {$z$};
% Text Node
\draw (140,145) node [anchor=north west][inner sep=0.75pt]    {$y_{D}$};
% Text Node
\draw (100,245) node [anchor=north west][inner sep=0.75pt]    {$z_{D}$};
% Text Node
\draw (126,205) node [anchor=north west][inner sep=0.75pt]    {$a_{s} u_{r^{\prime }} y_{D}$};
% Text Node
\draw (132,262) node [anchor=north west][inner sep=0.75pt]    {$a_{s} u_{r^{\prime}} x$};

\end{tikzpicture}

  \caption{Outline of the proof of Theorem \ref{effective2024.9.1}.}
\label{effective2024.08.4}
\end{figure}

Now assume that for given surface $x\in\mathcal{H}_{1}(2)$, there is $y_{D}\in (\Omega_{1}W_{D})_{\eta}$ so that  
\begin{equation}\label{effective2024.08.6}
  d(y_{D},x)< e^{- t}.
\end{equation} 
Combining (\ref{effective2024.08.5}) and (\ref{effective2024.08.6}), we obtain 
\begin{align}
d(z_{D}, a_{s}u_{r^{\prime}}x)&\leq d(z_{D}, a_{s}u_{r^{\prime}}y_{D})+d(a_{s}u_{r^{\prime}}y_{D}, a_{s}u_{r^{\prime}}x)\;\nonumber\\
&\leq d(z_{D}, a_{s}u_{r^{\prime}}y_{D})+e^{\aleph_{2} s}d(y_{D},x)\;\nonumber\\
 & \leq L^{-1}+ e^{\aleph_{2} s- t}\;  \nonumber
\end{align}
for some $\aleph_{2}>0$.
Then for sufficiently large $t>0$  (depending on $s$), the right hand side is $<2L^{-1}$.

Finally, by (\ref{effective2024.08.7}), we conclude that 
\[d(z, a_{s}u_{r^{\prime}}x)\leq d(z,z_{D})+d(z_{D}, a_{s}u_{r^{\prime}}x)<3L^{-1}.\]  
See Section \ref{effective2024.07.161} for more details.

\subsection{Outline of the proof of Theorem \ref{effective2024.07.148}} 
  The idea of Theorem \ref{effective2024.07.148} is that when the absolute periods $a_{t}u_{r}\Lambda(x)$ of  $a_{t}u_{r}x$ is closed to $Q_{D}$ in the sense of a fixed norm $\|\cdot\|$, we require that $a_{t}u_{r}x$ is close to  $Q_{D}$ in the sense of AGY norm $\|\cdot\|$ for at least one time $r$.
   
   In order to show that   $a_{t}u_{[0,1]}x$  can be close to the Teichm\"{u}ller curve $\Omega_{1}W_{D}$ for at least one time $r$. We need that $a_{t}u_{r}x$ does not go to infinity. However, this does not always seem to be the case.

\begin{figure}[H]
\centering

\tikzset{every picture/.style={line width=0.75pt}} %set default line width to 0.75pt        

\begin{tikzpicture}[x=0.75pt,y=0.75pt,yscale=-1,xscale=1]
%uncomment if require: \path (0,505); %set diagram left start at 0, and has height of 505

%Curve Lines [id:da9889653377849716] 
\draw [color={rgb, 255:red, 255; green, 255; blue, 255 }  ,draw opacity=1 ][fill={rgb, 255:red, 248; green, 231; blue, 28 }  ,fill opacity=1 ]   (394,124) .. controls (478.29,136.29) and (503.29,146.29) .. (574.29,169.29) .. controls (645.29,192.29) and (666.29,18.29) .. (592.29,36.29) .. controls (518.29,54.29) and (498.29,51.29) .. (453.29,57.29) .. controls (408.29,63.29) and (398.29,54.29) .. (396,83) .. controls (393.71,111.71) and (399.29,95.29) .. (394,124) -- cycle ;
%Curve Lines [id:da6864967670513094] 
\draw    (574.29,169.29) .. controls (478.29,134.29) and (417.29,123.29) .. (296.29,115.29) ;
%Curve Lines [id:da7832184492131393] 
\draw    (592.29,36.29) .. controls (518.29,55.29) and (396.29,67.29) .. (299.29,62.29) ;
%Curve Lines [id:da40749193234094605] 
\draw    (569.29,69.29) .. controls (485.29,94.29) and (449.29,80.29) .. (357.29,83.29) .. controls (265.29,86.29) and (285.29,99.29) .. (363.29,103.29) .. controls (441.29,107.29) and (514.29,109.29) .. (562.29,144.29) ;
%Straight Lines [id:da17388295767688877] 
\draw    (296,92) ;
\draw [shift={(296,92)}, rotate = 0] [color={rgb, 255:red, 0; green, 0; blue, 0 }  ][fill={rgb, 255:red, 0; green, 0; blue, 0 }  ][line width=0.75]      (0, 0) circle [x radius= 3.35, y radius= 3.35]   ;
%Straight Lines [id:da14348685137011152] 
\draw    (396,83) ;
\draw [shift={(396,83)}, rotate = 0] [color={rgb, 255:red, 0; green, 0; blue, 0 }  ][fill={rgb, 255:red, 0; green, 0; blue, 0 }  ][line width=0.75]      (0, 0) circle [x radius= 3.35, y radius= 3.35]   ;
%Straight Lines [id:da24143578540652766] 
\draw    (569.29,69.29) ;
\draw [shift={(569.29,69.29)}, rotate = 0] [color={rgb, 255:red, 0; green, 0; blue, 0 }  ][fill={rgb, 255:red, 0; green, 0; blue, 0 }  ][line width=0.75]      (0, 0) circle [x radius= 3.35, y radius= 3.35]   ;

% Text Node
\draw (253,75) node [anchor=north west][inner sep=0.75pt]    {$a_{t} u_{r} x$};
% Text Node
\draw (558,99) node [anchor=north west][inner sep=0.75pt]    {$\mathcal{H}_{1}^{(\eta)}(2)$};
% Text Node
\draw (396,65) node [anchor=north west][inner sep=0.75pt]    {$a_{t} u_{r_{1}} x$};
% Text Node
\draw (554,52) node [anchor=north west][inner sep=0.75pt]    {$a_{t} u_{r_{2}} x$};

\end{tikzpicture}

  \caption{Quantitative nondivergence of horocycle flows on $\mathcal{H}_{1}(2)$.}
\label{effective2024.08.9}
\end{figure}
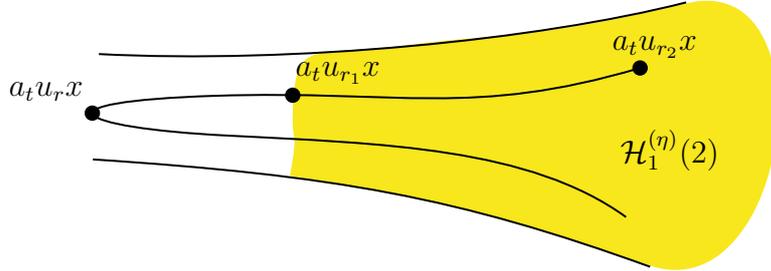

 To fix this,  we observe the quantitative nondivergence   of horocycle flows (Figure \ref{effective2024.08.9}). In fact, if $\ell(a_{t}u_{r}x)$ is very small, then (by the nature of $(C,\alpha)$-good functions) one may expect that there is a considerable interval $[r_{1},r_{2}]\subset[0,1]$ near $r$, so that $a_{t}u_{[r_{1},r_{2}]}x$ does not go to infinity. Then we can apply the effective equidistribution to $[r_{1},r_{2}]$, so that the absolute periods are close to the Teichm\"{u}ller curve again. More precisely, for any $t>0$, there is some $r^{\star}\in[r_{1},r_{2}]$, and $(\Lambda_{1,D},\Lambda_{2,D})\in Q_{D}$ such that 
 \[ \|(\Lambda_{1,D},\Lambda_{2,D})-a_{t}u_{r^{\star}}\tilde{x}\|_{a_{t}u_{r^{\star}}\tilde{x}}\leq e^{-\aleph_{3} t}\]
for some $\aleph_{3}>0$.  Eventually, we obtain a surface with sufficiently large injectivity radius, which guarantees a surface $y_{D}\in (\Omega_{1}W_{D})_{\eta}$ such that 
\[ \|\tilde{y}_{D}-a_{t}u_{r^{\star}}\tilde{x}\|_{a_{t}u_{r^{\star}}\tilde{x}}\leq e^{-\aleph_{3} t}.\]
 The details of this are in Section \ref{effective2024.07.147}.
 
\subsection{Structure of the paper}
In Section \ref{closing2024.3.23} we recall basic definitions, including  some basic material on the translation surfaces and Teichm\"{u}ller curves (in Section \ref{closing2024.3.24}, Section \ref{closing2024.3.25}). In particular, we study the period map via triangulation (Section \ref{closing2024.3.17}) and gives quantitative estimates in terms of Avila-Gou\"{e}zel-Yoccoz norm (Section \ref{closing2024.3.26}).

In Section \ref{effective2024.07.01}, we review the dynamics over $\mathcal{H}(2)$. In particular, we recall the McMullen's classification of Teichm\"{u}ller curves in $\mathcal{H}(2)$. From this, we deduce some   quantitative estimates of  Teichm\"{u}ller curves. 

In Section \ref{effective2024.08.10}, we review  the effective results in homogenous dynamics, which  shall serve as an effective estimate of absolute periods of surfaces in $\mathcal{H}(2)$.

In Section \ref{effective2024.07.161}, we prove Theorem \ref{effective2024.9.1}, as well as Theorems \ref{effective2024.07.148} and \ref{effective2024.07.159} by assuming certain nondivergence property. The proofs connect the effective estimates of absolute periods provided in  Section \ref{effective2024.08.10}, and the quantitative observations obtained in Section \ref{effective2024.07.01}.

 In Section \ref{effective2024.07.147}, we show that the nondivergence assumption made in Section \ref{effective2024.07.161} can actually be removed. It relies on Minsky and Weiss's work on the nondivergence results of horocycle flows on $\mathcal{H}(2)$ \cite{minsky2002nondivergence}.  In particular, we review the technique of sparse covers by $(C,\alpha)$-good functions and obtain a large interval to apply the equidistribution again.

  Finally,   we present in Section    \ref{effective2024.08.11}  the proof of Theorem \ref{effective2024.08.1}. More precisely,  via Ratner's theorem and  a criterion by McMullen, we conclude that if a surface has  $G$-closed absolute periods   and   some additional rationality on the areas of tori, then it generates a Teichm\"{u}ller curve.

 \noindent
 \textbf{Acknowledgements.} 
This work began with discussions of the author with Pengyu Yang. I would like to thank Pengyu for sharing with me his insights. 
  I am also grateful to   Curtis McMullen for comments that improved the exposition and correctness. 
 I am very grateful to Alex Eskin and Amir Mohammadi for pointing out an error in the first draft of this paper. It was the intuition of Alex that led to a significant change in the paper. 
I would also like to thank Giovanni Forni for the helpful discussions. Last, I would like to express my appreciation for the conversation with David Fisher, and the support of Junyi Xie and  Disheng Xu.

\section{Preliminaries}\label{closing2024.3.23}  
\subsection{Notation}
We will denote the metric on all relevant metric spaces by $d(\cdot,\cdot)$; where this may cause confusion, we will give the metric space as a subscript, e.g. $d_{X}(\cdot,\cdot)$ etc. Next, $B(x,r)$ denotes the open ball of radius $r$ in the metric space $x$ belongs to; where needed, the space we work in will be given as a subscript, e.g. $B_{T}(x,r)$. We will assume implicitly that for any $x\in X$ (as well as any other locally compact metric space we will consider) and $r>0$ the ball $B_{X}(x,r)$ is relatively compact.

We will use the asymptotic notation $A=O(B)$, $A\ll B$, or $A\gg B$,    for  positive quantities $A,B$ to mean  the estimate $|A|\leq CB$ for some constant $C$ independent of $B$. In some cases, we will need this constant $C$ to depend on a parameter (e.g. $d$), in which case we shall indicate this dependence by subscripts, e.g. $A=O_{d}(B)$ or  $A\ll_{d} B$. We also use $A\asymp B$ as a synonym for $A\ll B\ll A$.

Let $G=\SL_{2}(\mathbb{R})$, $\Gamma=\SL_{2}(\mathbb{Z})$.    Besides, we define 
\[B_{G}(T)\coloneqq\{g\in G:\|g-e\|\leq T\}\]
 where $e$ is the identity of $G$ and $\|\cdot\|$ is a fixed norm on the Euclidean space, e.g. $\|\cdot\|$ can be defined by
 \[\|g\|\coloneqq\max_{ij}\{|g_{ij}|,|g_{ij}^{-1}|\}\]
  where $g_{ij}$ is the $ij$-th entry of $g$. 
  Let $X=G/\Gamma\times G/\Gamma$. Let $d_{X}(\cdot,\cdot)$ be a right-invariant metric on $G\times G$ and the induced   metric on $X$.

  \subsection{Translation surfaces}\label{closing2024.3.24}
 Let $M$ be a compact oriented surface of genus $g$, and let $\Sigma\subset M$ be a nonempty finite set, called the set of zeroes. We make the convention that the points of $\Sigma$ are labeled. Let $\alpha=\{\alpha_{\sigma}:\sigma\in\Sigma\}$ be a partition of $2g-2$, so $\sum_{\sigma\in\Sigma} \alpha_{\sigma}=2g-2$.
 \begin{defn}[Translation surface]
   A surface $M$ is called a \textit{translation surface of type $\alpha$}\index{translation surface} if it has an affine atlas, i.e. a family of orientation preserving charts $\{(U_{a},z_{a})\}_{a}$ such that
   \begin{itemize}
     \item   the $U_{a}\subset M\setminus\Sigma$ are open and cover $M\setminus\Sigma$,
     \item the transition maps $z_{a}\circ z_{b}^{-1}$ have the form $z\mapsto z+c$.
   \end{itemize}  
   In addition, the planar structure of $M$ in a neighborhood of each $\sigma\in\Sigma$ completes to a cone angle singularity of total cone angle $2\pi(\alpha_{\sigma}+1)$. 
 \end{defn} 
 There are many equivalent definitions of a translation surface, and a convenient one is a pair $(M,\omega)$ consisting of a compact Riemann surface and  a holomorphic $1$-form $\omega$. We shall use these definitions interchangeably.
 
 There is a natural $\GL^{+}_{2}(\mathbb{R})$ action on the translation surfaces. Let $M$ be a translation surface with an atlas $\{(U_{a},z_{a})\}_{a}$.
  Since any matrix $h\in\GL^{+}_{2}(\mathbb{R})$ acts  on $\mathbb{C}=\mathbb{R}+i\mathbb{R}$, we obtain a new atlas $\{(U_{a},h\circ z_{a})\}_{a}$, which induces a new translation surface $hM$.

An \textit{affine isomorphism}\index{affine isomorphism} is an orientation preserving homeomorphism $f:M_{1}\rightarrow M_{2}$ which is affine  in each chart. If $M_{1}= M_{2}$, it is called an \textit{affine automorphism}\index{affine automorphism} instead. Let $\Aff(M)$ denote the set of affine automorphisms of $M$. 
 If an affine isomorphism whose linear part is $\pm\Id$ (for translation surfaces, $\Id$), it is called a \textit{translation equivalence}\index{translation equivalence}. Let $\mathcal{H}(\alpha)=\Omega \mathcal{M}_{g}(\alpha)$ denote the space of equivalence classes of translation surfaces of type $\alpha$. We refer to $\mathcal{H}(\alpha)$ as the \textit{moduli space of translation surfaces of type $\alpha$}\index{moduli space}.
 
  A \textit{saddle connection}\index{saddle connection} of $M$ is a geodesic segment joining two zeroes in $\Sigma$  or a zero to itself which has no zeroes in its interior.

 We fix a compact surface $(S,\Sigma)$ and  refer to it  as the model surface. 
   A \textit{marking map}\index{marking map} of a   surface $M$ is a homeomorphism $\varphi:(S,\Sigma)\rightarrow(M,\Sigma_{M})$ which preserves labels on $\Sigma$.  (We sometimes drop the subscript and use the same symbol $\Sigma$ to denote finite subsets of $S$ and of $M$, if no confusion  arises.) Two marking maps $\varphi_{1}:(S,\Sigma)\rightarrow (M_{1},\Sigma_{M_{1}})$ and  $\varphi_{2}:(S,\Sigma)\rightarrow (M_{2},\Sigma_{M_{2}})$ are said to be \textit{equivalent}\index{equivalent marking maps} if there is a translation equivalence $f:M_{1}\rightarrow M_{2}$ such that
   \begin{itemize}
     \item   $f\circ \varphi_{1}$ is isotopic to $\varphi_{2}$,
     \item $f$ maps $\Sigma_{M_{1}}\rightarrow \Sigma_{M_{2}}$ respecting the labels.
   \end{itemize}  
   An equivalent class of translation surfaces with marking maps is a \textit{marked translation surface}\index{marked translation surface}. The space of marked translation surfaces of type $\alpha$ is denoted by $\mathcal{TH}(\alpha)=\Omega \mathcal{T}_{g}(\alpha)$. We refer to $\mathcal{TH}(\alpha)$ as the \textit{Teichm\"{u}ller space of marked translation surfaces of type $\alpha$}\index{Teichm\"{u}ller space}. By forgetting the marking maps, we get a natural map  $\pi:\mathcal{TH}(\alpha)\rightarrow \mathcal{H}(\alpha)$.

  We can locally identify $\mathcal{TH}(\alpha)$ (and so $\mathcal{H}(\alpha)$) with $H^{1}(M,\Sigma;\mathbb{C})$. Let $\tilde{x}\in\mathcal{TH}(\alpha)$ be a marked translation surface with the marking $\varphi:(S,\Sigma)\rightarrow(M,\Sigma)$. 
   Suppose $M$ is equipped with a holomorphic $1$-form $\omega$.  Then the \textit{period map}\index{period map} 
   \[\omega^{\prime}\mapsto \left(\gamma\mapsto\int_{\gamma}\omega^{\prime}\right)\]
   from a neighborhood of $\omega$ to $H^{1}(M,\Sigma;\mathbb{C})$ gives a local homeomorphism.  
   Let $\tilde{x}\in\mathcal{TH}(\alpha)$ be a marked translation surface with the marking $\varphi:(S,\Sigma)\rightarrow(M,\Sigma)$. Suppose $M$ is equipped with a holomorphic $1$-form $\omega$.  
   Then after using the marking map $\varphi$ to pullback $\omega$, we get a distinguished element $\hol_{\tilde{x}}=\varphi^{\ast}(\omega)\in H^{1}(S,\Sigma;\mathbb{R}^{2})\cong H^{1}(S,\Sigma;\mathbb{C})$. Thus, if     $\gamma\in H_{1}(S,\Sigma;\mathbb{Z})$ is an oriented curve in $S$ with  endpoints in $\Sigma$, then 
  \[\hol_{\tilde{x}}(\gamma)=\tilde{x}(\gamma)\coloneqq\omega(\varphi(\gamma)).\]   
  We also refer to the map $\hol:\mathcal{TH}(\alpha)\rightarrow H^{1}(S,\Sigma;\mathbb{C})$ as the \textit{developing map}\index{developing map} or \textit{period map}\index{period map}. It is a local homeomorphism (see  Lemma \ref{effective2023.9.20}).
  If we fix  $2g+|\Sigma|-1$ curves $\gamma_{1},\ldots, \gamma_{2g+|\Sigma|-1}$ that form a basis for $H_{1}(S,\Sigma;\mathbb{Z})$, then  it defines the  \textit{period coordinates}\index{period coordinates} $\phi:\mathcal{TH}(\alpha)\rightarrow  \mathbb{C}^{2g+|\Sigma|-1}$ by
   \[\phi:\tilde{x}\mapsto\left(  \hol_{\tilde{x}}(\gamma_{i})\right)_{i=1}^{2g+|\Sigma|-1}.\]
    It is convenient to assume that the basis is obtained by fixing a triangulation $\tau$ of the surface  by saddle connections of $x$ (see Definition \ref{closing2024.1.3}).  Via the \textit{Gauss-Manin connection}\index{Gauss-Manin connection}, period coordinates endow $\mathcal{TH}(\alpha)$ with a canonical complex affine structure.

  Let $\Mod(M,\Sigma)$ be the group of isotopy classes of homeomorphisms $M$ which fix $\Sigma$ pointwise for a representative $(M,\Sigma)$ of the stratum $\alpha$. We will call this group the \textit{mapping class group}\index{mapping class group}. It acts on the right on $\mathcal{TH}(\alpha)$: letting $\tilde{x}\in \mathcal{TH}(\alpha)$ with a marking map $\varphi: (S,\Sigma)\rightarrow(M,\Sigma)$, $\gamma\in\Mod(M,\Sigma)$, we have the action \[\gamma.\varphi=\varphi\circ\gamma.\]
    The $\Mod(M,\Sigma)$-action on $\mathcal{TH}(\alpha)$ is properly discontinuous (e.g. \cite[Theorem 12.2]{farb2011primer}). Hence, $\mathcal{H}(\alpha)=\mathcal{TH}(\alpha)/\Mod(M,\Sigma)$ has an orbifold structure. 
   We   choose a fundamental domain $\mathcal{D}$ on $\mathcal{TH}(\alpha)$ for the action of $\Mod(M,\Sigma)$. Note that $\Mod(M,\Sigma)$ also acts on the right by linear automorphisms on $H^{1}(S,\Sigma;\mathbb{R})$. Let $R:\Mod(M,\Sigma)\rightarrow\Aut(H^{1}(S,\Sigma;\mathbb{R}))\cong\GL(2g+|\Sigma|-1,\mathbb{R})$. Since each element of $\Mod(M,\Sigma)$ is represented by an orientation-preserving homeomorphism of $S$, it follows that the image of $R$ lies in $\SL(2g+|\Sigma|-1,\mathbb{R})$.

    Let $\tilde{x}\in \mathcal{D}\subset \mathcal{TH}(\alpha)$ and $h\in\GL^{+}_{2}(\mathbb{R})$. Then there is a unique element $\gamma\in\Gamma$ so that $h\tilde{x}\gamma\in \mathcal{D}$. The   \textit{Kontsevich-Zorich cocycle}\index{Kontsevich-Zorich cocycle} is then defined by
   \[R(h,\tilde{x})\coloneqq R(\gamma).\] 
    Then for $\tilde{x}\in \mathcal{D}\subset \mathcal{TH}(\alpha)$,   the $G$-action becomes 
\begin{equation}\label{closing2023.08.6}
  Dh_{\tilde{x}}:\begin{bmatrix}
x_{1} & \cdots & x_{2g+|\Sigma|-1}\\
y_{1} & \cdots & y_{2g+|\Sigma|-1}
\end{bmatrix}\mapsto h\begin{bmatrix}
x_{1} & \cdots & x_{2g+|\Sigma|-1}\\
y_{1} & \cdots & y_{2g+|\Sigma|-1}
\end{bmatrix}R(h,\tilde{x}).
\end{equation}  
See e.g. \cite[\S2]{filip2016semisimplicity} for more details.
   
   In the literature, we sometimes refer to $\mathcal{TH}(\alpha)$ and $\mathcal{H}(\alpha)$ as a stratum of $\mathcal{TH}^{g}$ and $\mathcal{H}^{g}$, namely the Teichm\"{u}ller and  moduli spaces of translation surfaces of genus $g$, respectively. This is because we have the stratification
   \[\mathcal{TH}^{g}=\bigsqcup_{\alpha_{1}+\cdots+\alpha_{\sigma}=2g-2}\mathcal{TH}(\alpha),\ \ \ \ \ \ \mathcal{H}^{g}=\bigsqcup_{\alpha_{1}+\cdots+\alpha_{\sigma}=2g-2}\mathcal{H}(\alpha).\]
  
  On the other hand, let $\mathcal{T}_{g}$, $\mathcal{M}_{g}$ denote the Teichm\"{u}ller and moduli spaces  of Riemann surfaces of genus $g$ respectively. Let   $\Omega(M)$ denote the $g$-dimensional vector space of all holomorphic $1$-forms of $M$. Then we may consider $\mathcal{TH}^{g}$ and $\mathcal{H}^{g}$ as vector bundles over  $\mathcal{T}_{g}$, $\mathcal{M}_{g}$:
  \[\mathcal{TH}^{g}=\Omega \mathcal{T}_{g}\rightarrow \mathcal{T}_{g},\ \ \ \mathcal{H}^{g}=\Omega \mathcal{M}_{g}\rightarrow \mathcal{M}_{g}\] 
   whose fiber over $M$ is $\Omega(M)\setminus\{0\}$.
  
  Suppose that  $x=(M,\omega)\in\mathcal{H}(\alpha)$ is a translation surface of type $\alpha$.   Let $\Area(M,\omega)$ be the area of translation surface given by
  \[\Area(M,\omega)\coloneqq\frac{i}{2}\int_{M}\omega\wedge\bar{\omega}=\frac{i}{2}\sum_{j=1}^{g}(A_{j}(\omega)\bar{B}_{j}(\omega)-B_{j}(\omega)\bar{A}_{j}(\omega))\]
where $A_{j}(\omega), B_{j}(\omega)$ form a canonical basis of absolute periods of $\omega$, i.e.
\[A_{j}(\omega)=\int_{\alpha_{j}}\omega,\ \ \ B_{j}(\omega)=\int_{\beta_{j}}\omega \]
and $\{\alpha_{j},\beta_{j}\}_{j=1}^{g}$ is a symplectic basis of $H_{1}(M;\mathbb{R})$ (with respect to the intersection form). Let 
\[\mathcal{H}_{1}(\alpha)\coloneqq\{(M,\omega)\in\mathcal{H}(\alpha):\Area(M,\omega)=1\}.\] We see that  the normalized stratum $\mathcal{H}_{1}(\alpha)$ resembles more a ``unit hyperboloid".  Note that  $\mathcal{H}_{1}(\alpha)$ is a codimension one sub-orbifold of $\mathcal{H}(\alpha)$ but it is \textbf{not} an affine sub-orbifold.  Let $\pi_{1}:\mathcal{H}(\alpha)\rightarrow \mathcal{H}_{1}(\alpha)$ be the normalization of the area.  We  abuse notation and use the same symbol $\pi_{1}:\mathcal{TH}(\alpha)\rightarrow \mathcal{H}_{1}(\alpha)$ to refer to the composition of the projection and normalization. 
   
  Let $\lambda$ be the measure on $\mathcal{H}(\alpha)$ which is given by the pullback of the Lebesgue measure on $H^{1}(S,\Sigma;\mathbb{C})\cong \mathbb{C}^{2g+|\Sigma|-1}$.  We refer to $\lambda$  as the \textit{Lebesgue} or the \textit{Masur-Veech measure}\index{Masur-Veech measure} on $\mathcal{H}(\alpha)$. Let $\lambda_{(1)}$ be the $\SL(2,\mathbb{R})$-invariant Lebesgue (probability) measure on the ``hyperboloid" $\mathcal{H}_{1}(\alpha)$ defined by the disintegration of the Lebesgue measure $\lambda$ on  $\mathcal{H}_{1}(\alpha)$, namely
      \[d\lambda=r^{2g+|\Sigma|-2} d r\cdot d\lambda_{(1)}.\]

  \subsection{Teichm\"{u}ller curves}\label{closing2024.3.25} As we have seen, there is a natural
  $G=\SL_{2}(\mathbb{R})$ action on $\mathcal{H}_{1}(\alpha)$. We are then interested in its smallest $G$-orbit closure:
  \begin{defn}[Teichm\"{u}ller curve]
      A \textit{Teichm\"{u}ller curve}\index{Teichm\"{u}ller curve} $f:V\rightarrow\mathcal{M}_{g}$ is a finite volume hyperbolic Riemann surface $V$ equipped with a holomorphic, totally geodesic, generically 1-1 immersion into moduli space.  
  \end{defn}
 Let $(M,\omega)\in\mathcal{H}^{g}$ be a translation surface.  Recall that $\Aff(M)$ denotes the set of affine automorphisms. Consider the map $D:\Aff(M)\rightarrow G$ which assigns to an affine automorphism its linear part. It  has a finite kernel $\Gamma_{M}$, consisting of translation equivalences of $M$. The image $\SL(M,\omega)\coloneqq D(\Aff(M))$ is called the \textit{Veech group}\index{Veech group} of $M$. Then we have the short exact sequence:
 \begin{equation}\label{closing2024.3.20}
   0\rightarrow \Gamma_{M}\rightarrow\Aff(M)\rightarrow \SL(M,\omega)\rightarrow0.
 \end{equation} 
      The equivalent conditions for the lattice property of $\SL(M,\omega)$ has been studied by a vast literature (e.g. \cite{smillie2010characterizations} and references therein):  
  \begin{thm}\label{effective2023.10.3}
     For $x\in\Omega\mathcal{M}_{g}$, the following are equivalent:
     \begin{itemize}
       \item  The group $\SL(x)$ is a lattice in $G=\SL_{2}(\mathbb{R})$.
       \item The orbit $G.x$ is closed in $\Omega\mathcal{M}_{g}$.
       \item The projection of the orbit to $\mathcal{M}_{g}$ is a Teichm\"{u}ller curve.
     \end{itemize}
  \end{thm}
  In this case, we say $x$ \textit{generates}\index{generates the Teichm\"{u}ller curve} the Teichm\"{u}ller curve $V=\mathbb{H}/\SL(x)\rightarrow\mathcal{M}_{g}$. It follows that 
  \[\Omega V=\GL_{2}^{+}(\mathbb{R}).x\cong \GL_{2}^{+}(\mathbb{R})/\SL(x)\]
  can be regarded as a bundle over $V$. Thus, we also abuse notation and refer to $f(V)$, or the $\mathbb{C}^{\ast}$-bundle $\Omega V$ (and the circle bundle $\Omega_{1} V$) as a Teichm\"{u}ller curve, if no confusion arise.
  
  \subsection{Nondivergence}\label{closing2024.3.65}
  First, we review the quantitative nondivergence in the homogeneous dynamics.  Let $G=\SL_{2}(\mathbb{R})$, $\Gamma=\SL_{2}(\mathbb{Z})$ and $X=G/\Gamma\times G/\Gamma$. Note that any  $\Lambda_{1}\in G/\Gamma$  corresponds to a lattice in $\Lambda_{1}\subset\mathbb{R}^{2}$. For $\Lambda_{1}\in G/\Gamma$, define the systole function $\ell:G/\Gamma\rightarrow\mathbb{R}^{+}$ by
  \[\ell(\Lambda_{1})\coloneqq\min\{r:\Lambda_{1}\subset\mathbb{R}^{2}\text{ contains a vector of length}\leq r\}.\]
  Next, abuse notation and define $\ell:X\rightarrow\mathbb{R}^{+}$ by
  \[\ell(\Lambda_{1},\Lambda_{2})\coloneqq \min\{\ell(\Lambda_{1}),\ell(\Lambda_{2})\}.\]
  Note that for all $\eta>0$, the set 
      \[X_{\eta}\coloneqq\{x\in X:\ell(x)\geq\eta\}\]
       is compact. In homogeneous dynamics, we have the following nondiverngence result,   ultimately attributed to Margulis, Dani, and Kleinbock.
       \begin{thm}[{\cite[Proposition 3.1]{lindenstrauss2023polynomial}}]\label{effective2024.07.34}
  There   exists $C\geq 1$ with the following property: Let $\epsilon,\eta\in(0,1)$, $(\Lambda_{1},\Lambda_{2})\in X$. Let $I\subset\mathbb{R}$ be an interval of length $|I|\geq \eta$. Then
        \[|\{r\in I:\ell(a_{t}u_{r}(\Lambda_{1},\Lambda_{2}))<\epsilon^{2}\}|<C\epsilon |I|\]
        so long as $t\geq |\log (\eta^{2}\inj(x))|+C$.
       \end{thm} 

     In the following, we shall develop a similar result in Teichm\"{u}ller dynamics. Define the systole function $\ell:\mathcal{TH}(\alpha)\rightarrow\mathbb{R}^{+}$ by the shortest length of a  saddle connection. Note that for all $\epsilon>0$, the set 
     \begin{equation}\label{effective2024.08.12}
      \mathcal{H}^{(\epsilon)}_{1}(\alpha)\coloneqq\{x\in\mathcal{H}_{1}(\alpha):\ell(x)\geq\epsilon\}
     \end{equation} 
       is compact. To say it differently, a sequence $x_{n}\in \mathcal{H}_{1}(\alpha)$ diverges to infinity iff $\ell(x_{n})\rightarrow0$. In addition, by the Siegel-Veech formula (see e.g. \cite{eskin2001asymptotic,avila2006exponential}), we have
       \begin{equation}\label{closing2023.07.13}
        \lambda_{(1)}(\mathcal{H}_{1}(\alpha)\setminus \mathcal{H}^{(\epsilon)}_{1}(\alpha))=\lambda_{(1)}\{x\in\mathcal{H}_{1}(\alpha):\ell(x)<\epsilon\}\asymp O(\epsilon^{2}).
       \end{equation}

      Now, we follow the idea in \cite{eskin2001asymptotic} and \cite{athreya2006quantitative} to  discuss the non-divergence results. See also \cite[\S6]{avila2013small} and \cite[\S2]{eskin2022effective}.
      
      \begin{thm}[{\cite{eskin2001asymptotic,athreya2006quantitative}}]\label{effective2023.11.2}
        There \hypertarget{2024.08.k3} exist \hypertarget{2024.08.C2} a continuous function $V:\mathcal{H}_{1}(\alpha)\rightarrow [2,\infty)$, a compact subset $K^{\prime}_{\alpha}\subset \mathcal{H}_{1}(\alpha)$ and some $\kappa_{3}>0$ with the following property. For every $t^{\prime}$ and every $x\in\mathcal{H}_{1}(\alpha)$, there exist 
        \[s\in[0,1/2],\ \ \ t^{\prime}\leq t\leq \max\{2t^{\prime},\hyperlink{2024.08.k3}{\kappa_{3}}\log V(x)\}\]
        such that $a_{t}u_{s}x\in K^{\prime}_{\alpha}$. Further, there exists a constant   $C_{2}>1$ such that 
        \[ \hyperlink{2024.08.C2}{C_{2}}^{-1}\leq \frac{V(x)}{\max\{\ell(x)^{-5/4},1\}}\leq \hyperlink{2024.08.C2}{C_{2}}.\]
        where $\ell$ denotes the systole function.
      \end{thm}

      We also need the following averaging nondivergence of horocyclic flows.  
      \begin{thm}[{\cite[Theorem 6.3]{minsky2002nondivergence}}]\label{closing2024.3.60}  
         There \hypertarget{2024.08.k4} are positive constants $C_{3},\kappa_{4},\rho_{0}$, depending only on $\alpha$, such that if $\tilde{x}\in\mathcal{TH}_{1}(\alpha)$, an interval  \hypertarget{2024.08.C3} $I\subset\mathbb{R}$, and $\rho\in(0,\rho_{0}]$ satisfy: 
         \[\sup_{s\in I}\ell(u_{s}\tilde{x})\geq\rho,\] 
         then  for any $\epsilon\in(0,\rho)$, we have 
         \begin{equation}\label{closing2024.3.61}
           |\{s\in I:\ell(u_{s}\tilde{x})<\epsilon\}|\leq \hyperlink{2024.08.C3}{C_{3}}\left(\frac{\epsilon}{\rho}\right)^{\hyperlink{2024.08.k4}{\kappa_{4}}}|I|.
         \end{equation} 
      \end{thm}
      In Section \ref{effective2024.07.147}, we shall revisit the technique in the proof of Theorem \ref{closing2024.3.60}, in order to analyze the behavior of points going to infinity.  
      
      Now we are in the position to establish the desired nondivergence result. (See also Theorem \ref{effective2024.07.34}.)
      \begin{cor}\label{closing2024.3.63}  
          There  \hypertarget{2024.08.k5} are \hypertarget{2024.08.C4} positive constants $C_{4},\kappa_{5}$, depending only on $\alpha$, with the following property. Let $\epsilon>0$, $\eta>0$, and  $\tilde{x}\in\mathcal{TH}_{1}(\alpha)$. Let $I\subset[-10,10]$ be  an interval  with $|I|\geq \eta$. Then  we have
         \[|\{r\in I:\ell(a_{t}u_{r}\tilde{x})<\epsilon\}|\leq \hyperlink{2024.08.C4}{C_{4}}\epsilon^{\hyperlink{2024.08.k4}{\kappa_{4}}}|I|\]
         whenever $t\geq \hyperlink{2024.08.k5}{\kappa_{5}}|\log \ell(x)|+2|\log \eta|+\hyperlink{2024.08.C4}{C_{4}}$.
      \end{cor}
      \begin{proof} Assume for simplicity that $I=[0,1]$. More general situation follows from a similar argument.
         Let $\epsilon_{1}>0$ satisfy $K^{\prime}_{\alpha}\subset \mathcal{H}_{1}^{(\epsilon_{1})}(\alpha)$. Let $\tilde{x}\in\mathcal{TH}_{1}(\alpha)$. Without loss of generality, we assume that $\ell(\tilde{x})\ll 1$. Let $t^{\prime}=\max\{1,\frac{\hyperlink{2024.08.k3}{\kappa_{3}}}{2}\log V(x)\}$. Then applying Theorem \ref{effective2023.11.2} to $x$ and $t^{\prime}$, there exist 
         \[s_{0}\in[0,1/2],\ \ \   t_{0}\in[1, \hyperlink{2024.08.k3}{\kappa_{3}}\log \hyperlink{2024.08.C2}{C_{2}}\ell(x)^{-5/4}]\]
         such that $x_{0}\coloneqq a_{t_{0}}u_{s_{0}}x\in K^{\prime}_{\alpha}$.
         
         Next, let $t^{\prime}=\max\{1,\frac{\hyperlink{2024.08.k3}{\kappa_{3}}}{2}\log V(x_{0})\}$, $C=\hyperlink{2024.08.k3}{\kappa_{3}}\log \hyperlink{2024.08.C2}{C_{2}}\epsilon_{1}^{-5/4}$. Applying Theorem \ref{effective2023.11.2} again to $x_{0}$ and $t^{\prime}$, we obtain that there exist
           \[s_{1}\in[0,1/2],\ \ \   t_{1}\in[1, C]\]
            such that $x_{1}\coloneqq a_{t_{1}}u_{s_{1}}x_{0}\in K^{\prime}_{\alpha}$. Repeating the argument, we further obtain   that there exist
           \[s_{i}\in[0,1/2],\ \ \   t_{i}\in[1, C]\]
            such that $x_{i}\coloneqq a_{t_{i}}u_{s_{i}}x_{i-1}\in K^{\prime}_{\alpha}$ for all $i\in\mathbb{N}$. One calculates that $x_{i}= a_{t(i)}u_{s(i)}x$ where
            \[t(i)\coloneqq t_{i}+\cdots +t_{0},\ \ \ \text{ and }\ \ \  s(i)\coloneqq\sum_{j=0}^{i}\frac{s_{j}}{e^{t_{j-1}+\cdots+t_{0}}}\leq\sum_{j=0}^{i}\frac{s_{j}}{e^{j}}\leq 1.\]
            
            Then we see that 
              \[\sup_{s\in [0,1]}\ell(u_{e^{t(i)}s}a_{t(i)}\tilde{x})=\sup_{s\in [0,1]}\ell(a_{t(i)}u_{s}\tilde{x})\geq\epsilon_{1}.\] 
              Moreover, note that $t(i)-t(i-1)=t_{i}\leq  C$. Then for any $t\geq t_{0}$, we have 
              \[\sup_{s\in [0,1]}\ell(u_{e^{t}s}a_{t}\tilde{x})=\sup_{s\in [0,1]}\ell(a_{t}u_{s}\tilde{x})\geq\epsilon_{1}e^{-C}.\] 
              Now applying Theorem \ref{closing2024.3.60}, we obtain that 
                   \[|\{s\in [0,1]:\ell(a_{t}u_{s}\tilde{x})<\epsilon\}|\leq \hyperlink{2024.08.C3}{C_{3}}\left(\epsilon_{1}e^{-C}\right)^{-\hyperlink{2024.08.k4}{\kappa_{4}}}\epsilon^{\hyperlink{2024.08.k4}{\kappa_{4}}}\]
                   whenever $t\geq \hyperlink{2024.08.k3}{\kappa_{3}}\log \hyperlink{2024.08.C2}{C_{2}}\ell(x)^{-5/4}$. Letting $\hyperlink{2024.08.k5}{\kappa_{5}}=\frac{5}{4}\hyperlink{2024.08.k3}{\kappa_{3}}$ and $\hyperlink{2024.08.C4}{C_{4}}=\hyperlink{2024.08.k3}{\kappa_{3}}\log \hyperlink{2024.08.C2}{C_{2}}$, we establish  (\ref{closing2024.3.61}).
      \end{proof}
          
    \subsection{Avila-Gou\"{e}zel-Yoccoz norm}\label{closing2024.3.26} 
 We now introduce the  AGY norm, first defined in  \cite{avila2006exponential}, some properties of which were further developed in \cite{avila2013small}. 
 \begin{defn}[AGY norm]\label{effective2024.08.13}
     For $\tilde{x}\in  \mathcal{TH}(\alpha)$ and any $c\in H^{1}(M,\Sigma;\mathbb{C})$, we define
   \[ \|c\|_{\tilde{x}} \coloneqq \sup_{\gamma}\frac{|c(\gamma)|}{|\int_{\gamma}\omega|}\]
   where $\gamma$ is a saddle connection of $\tilde{x}$. We refer to $\|\cdot\|_{\tilde{x}}$ as the \textit{Avila-Gou\"{e}zel-Yoccoz norm}\index{Avila-Gou\"{e}zel-Yoccoz norm} or \textit{AGY norm}\index{AGY norm} for short.
 \end{defn}
 
 Note first that by definition, for $w=a+ib\in H^{1}(M,\Sigma;\mathbb{C})$, we have 
 \begin{equation}\label{effective2024.6.05}
   \max\{\|a\|_{x},\|b\|_{x}\}\leq \|w\|_{x}\leq \|a\|_{x}+\|b\|_{x}.
 \end{equation} 
 
By construction, the AGY norm is invariant under the action of the mapping class group $\Gamma$. Thus, it induces a norm on the moduli space $\mathcal{H}(\alpha)$.

     It was shown in  \cite[\S2.2.2]{avila2006exponential} that   this defines a norm and the corresponding Finsler metric is complete.  
    For $\tilde{x},\tilde{y}\in \mathcal{TH}(\alpha)$, we define a distance  
     \[d (\tilde{x},\tilde{y})\coloneqq\inf_{\gamma}\int_{0}^{1}\|\gamma^{\prime}(r)\|_{\gamma(r)}dr\] where $\gamma$ ranges over smooth paths $\gamma:[0,1]\rightarrow \mathcal{TH}(\alpha)$ with $\gamma(0)=\tilde{x}$ and $\gamma(1)=\tilde{y}$. It also induces a quotient metric on $\mathcal{H}(\alpha)$.
     
 Due to the splitting
   \[H^{1}(M,\Sigma;\mathbb{C})=H^{1}(M,\Sigma;\mathbb{R})\oplus iH^{1}(M,\Sigma;\mathbb{R}),\] 
   we often write an element of $H^{1}(M,\Sigma;\mathbb{C})$ as  $a+ib$ for $a,b\in H^{1}(M,\Sigma;\mathbb{R})$. 
    Let $\tilde{x}\in \mathcal{TH}(\alpha)$. For every $r>0$, define 
    \[R(\tilde{x},r)\coloneqq\{\phi(\tilde{x})+a+ib: a,b\in H^{1}(M,\Sigma;\mathbb{R}),\ \|a+ib\|_{\tilde{x}}\leq r\}.\]
   Let   $r>0$ be so that $\phi^{-1}$ is a homeomorphism on $R_{\tilde{x}}(r)$. Let 
    \[B(\tilde{x},r)\coloneqq \phi^{-1}(R(\tilde{x},r)).\]
    We call it a \textit{period box}\index{period box} of radius $r$ centered at $\tilde{x}$. Using   \cite[Proposition 5.3]{avila2013small}, $B(\tilde{x},r)$ is well defined for all $r\in(0,1/2]$ and all $\tilde{x}\in\mathcal{TH}(\alpha)$.
  Let $\inj(\tilde{x})$ be the injectivity radius of $\tilde{x}$ under the affine exponential map.

     We have the following estimate:
     \begin{lem}\label{closing2023.07.1}
    Let  $\tilde{x}\in \mathcal{TH}(\alpha)$. Then  for all $\tilde{y},\tilde{z}\in B(\tilde{x}, \inj(\tilde{x})/50)$, we have
    \[\frac{1}{2}\|\tilde{y}-\tilde{z}\|_{\tilde{y}}\leq \|\tilde{y}-\tilde{z}\|_{\tilde{z}}\leq 2\|\tilde{y}-\tilde{z}\|_{\tilde{y}},\]
    and further
    \[\frac{1}{4}\|\tilde{y}-\tilde{z}\|_{\tilde{x}}\leq d(\tilde{y},\tilde{z})\leq 4\|\tilde{y}-\tilde{z}\|_{\tilde{x}}.\] 
     \end{lem} 
 \begin{proof}
   It is actually a rephrasing of   \cite[Proposition 5.3]{avila2013small}.   See also \cite[Lemma 3.3]{Chaika2023ergodic}. 
 \end{proof}
\begin{lem}\label{effective2024.07.145}
   For $g\in \GL_{2}(\mathbb{R})$, $w\in H^{1}(x)$, we have  $\|gw-w\|_{x}\leq\|g-I\|\|w\|_{x}$.
   Moreover,  we have $\|gx-x\|_{x}=\|g-I\|$.
\end{lem}
\begin{proof}
Clearly, one calculates,
\[\|gw-w\|_{x}=\sup_{\gamma}\frac{|(g-I)w(\gamma)|}{|x(\gamma)|}\leq \|g-I\|\sup_{\gamma}\frac{|w(\gamma)|}{|x(\gamma)|}=\|g-I\|\|w\|_{x}.\]
 Moreover, since the slopes of saddle connections are dense in $\mathbb{R}\cup\{\infty\}$ (see \cite[\S4]{masur2002rational}), we have
  \[\|g-I\|=\sup_{\gamma}\left|(g-I)\frac{x(\gamma)}{|x(\gamma)|}\right|=\sup_{\gamma}\frac{|(g-I)x(\gamma)|}{|x(\gamma)|}=\|gx-x\|_{x}.\]
  This establishes the claim.
\end{proof}
      \begin{lem}[{\cite[Lemma 5.1]{avila2013small}}]\label{closing2023.12.3}
      For $x\in\mathcal{H}(\alpha)$, $g\in G$, we have 
         \[d(x,gx)\leq d_{G}(e,g). \]
         In particular, via the Cartan decomposition $g=kak^{\prime}$, we have 
               \[d(x,gx)\leq \log\|a\|+4\pi. \]
      \end{lem}

      The following crude estimates are well known, e.g. \cite[Corollary 2.6]{chaika2020tremors}.
      \begin{thm}\label{effective2023.11.3}
           For all $s,t\in\mathbb{R}$, $x\in\mathcal{H}(\alpha)$, we have 
         \[ \|u_{s}v\|_{u_{s}x}\leq \left(1+\frac{s^{2}+|s|\sqrt{s^{2}+4}}{2}\right)\|v\|_{x}\]
         and 
         \[ \|a_{t}v\|_{a_{t}x}\leq e^{2|t|}\|v\|_{x}.\]
      \end{thm} 
      
      \begin{lem}[{\cite[Lemma 2.6]{eskin2022effective}}]\label{effective2023.11.4} \hypertarget{2024.08.k6}
        There exist    $\kappa_{6}=\kappa_{6}(\alpha)>0$ and  $C_{5}>1$  so that for all $x\in \mathcal{H}_{1}(\alpha)$, the following hold.
      For any $r\in(0,\hyperlink{2024.08.C5}{C_{5}}\ell(x)^{ \hyperlink{2024.08.k6}{\kappa_{6}}}]$, any lift $\tilde{x}$ of $x$, the restriction of the covering map $\pi:\mathcal{TH}(\alpha)\rightarrow \mathcal{H}(\alpha)$ to $B(\tilde{x},r)$ is injective. 
      \end{lem} 
      
      \subsection{Triangulation}\label{closing2024.3.17}
      
      In application, we usually choose a triangulation for the period coordinates, i.e. fix  a triangulation $\tau$ of the surface and choose a sequence of saddle connections from $\tau$ which form a basis for $H_{1}(M,\Sigma;\mathbb{Z})$.   In this section, we follow the idea in \cite{masur1991hausdorff} to  discuss the period coordinates, and find a lower bound of the non-degenerate deformations of a triangulation. In particular, this gives a  lower bound of the injectivity radius of the period map $\phi:\mathcal{TH}(\alpha)\rightarrow H^{1}(M,\Sigma;\mathbb{C})$. We adopt the notation introduced in \cite{chaika2020tremors}. 
      \begin{defn}[geodesic triangulation] \label{closing2024.1.3}
      We say $\tau$ is a \textit{geodesic triangulation}\index{geodesic triangulation} of $x$ if it is a decomposition of the surface into triangles whose sides are saddle connections,  and whose vertices are singular points, which need not be distinct.
      \end{defn} 
In \cite[$\S$4]{masur1991hausdorff}, Masur and Smillie showed that every translation surface $x\in\mathcal{H}(\alpha)$ admits a \textit{Delaunay triangulation}\index{Delaunay triangulation} $\tau_{x}$, which is  a typical geodesic triangulation. By the construction, each triangle $\Delta\in\tau_{x}$ can be inscribed in a disk of radius not greater than the diameter $d(M)$ of $M$ (cf. {\cite[Theorem 4.4]{masur1991hausdorff}}).

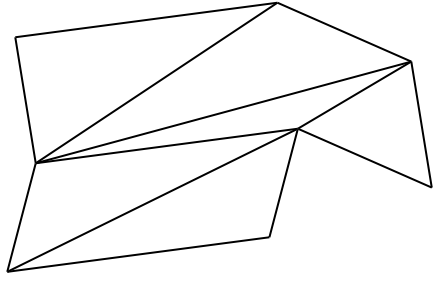
\begin{figure}[H]
\centering

\tikzset{every picture/.style={line width=0.75pt}} %set default line width to 0.75pt        

\begin{tikzpicture}[x=0.75pt,y=0.75pt,yscale=-1,xscale=1]
%uncomment if require: \path (0,398); %set diagram left start at 0, and has height of 398

%Straight Lines [id:da44589008481845993] 
\draw    (567.49,132.19) -- (577.7,195.57) ;
%Straight Lines [id:da4283775688094611] 
\draw    (380.11,183.33) -- (369.9,119.95) ;
%Straight Lines [id:da4920169573654245] 
\draw    (567.49,132.19) -- (500.7,102.57) ;
%Straight Lines [id:da8568086377512314] 
\draw [fill={rgb, 255:red, 255; green, 255; blue, 255 }  ,fill opacity=1 ]   (365.9,237.95) -- (380.11,183.33) ;
%Straight Lines [id:da9194703464184404] 
\draw [fill={rgb, 255:red, 255; green, 255; blue, 255 }  ,fill opacity=1 ]   (365.9,237.95) -- (496.7,220.57) ;
%Straight Lines [id:da9840572829681571] 
\draw    (510.9,165.95) -- (567.49,132.19) ;
%Straight Lines [id:da5462966643408345] 
\draw [fill={rgb, 255:red, 255; green, 255; blue, 255 }  ,fill opacity=1 ]   (365.9,237.95) -- (510.9,165.95) ;
%Straight Lines [id:da16992843832086857] 
\draw [fill={rgb, 255:red, 255; green, 255; blue, 255 }  ,fill opacity=1 ]   (369.9,119.95) -- (500.7,102.57) ;
%Straight Lines [id:da5455644335310927] 
\draw    (577.7,195.57) -- (510.9,165.95) ;
%Straight Lines [id:da4385840671260002] 
\draw [fill={rgb, 255:red, 255; green, 255; blue, 255 }  ,fill opacity=1 ]   (496.7,220.57) -- (510.9,165.95) ;
%Straight Lines [id:da2604655838903074] 
\draw [fill={rgb, 255:red, 255; green, 255; blue, 255 }  ,fill opacity=1 ]   (380.11,183.33) -- (510.9,165.95) ;
%Straight Lines [id:da009940966956249797] 
\draw    (380.11,183.33) -- (567.49,132.19) ;
%Straight Lines [id:da37931220167148494] 
\draw    (500.7,102.57) -- (380.11,183.33) ;

\end{tikzpicture}

  \caption{A Delaunay triangulation of a surface in $\mathcal{H}(2)$.}
\label{effective2024.07.132}
\end{figure}

  Let $\tilde{x}\in\mathcal{TH}(\alpha)$ be a marked translation surface with the marking map $\varphi:(S,\Sigma)\rightarrow(M,\Sigma)$, and $x=\pi(\tilde{x})=(M,\omega)\in\mathcal{H}(\alpha)$.
 Let $\tau_{\tilde{x}}$ denote the pullback of the Delaunay triangulation with vertices in $\Sigma$, from $(M,\Sigma)$ to  $(S,\Sigma)$. 
       
Note that the period map $\hol_{\tilde{x}}(\gamma)$  can be thought of as giving a map from the triangles of $\tau_{\tilde{x}}$ to triangles in $\mathbb{C}\cong\mathbb{R}^{2}$ (well-defined up to translation). Moreover, we can define a local inverse of the period map as follows.
 
 Let $U_{\tilde{x}}\subset H^{1}(S,\Sigma;\mathbb{C})$ be the collection of all cohomology classes which map each triangle of $\tau_{\tilde{x}}$ into a positively oriented non-degenerate triangle in $\mathbb{C}$. 
    Each $\nu\in U_{\tilde{x}}$ gives a translation surface $M_{\tilde{x},\nu}$ built by gluing together the corresponding triangles in $\mathbb{C}$ along parallel edges, and a   marking map $\varphi_{\tilde{x},\nu}:(S,\Sigma)\rightarrow (M_{\tilde{x},\nu},\Sigma)$, by taking each triangle of the triangulation $\tau_{\tilde{x}}$ of $S$ to the corresponding triangle of the triangulation of $M_{\tilde{x},\nu}$. 
    Let $\tilde{y}_{\tilde{x},\nu}\in\mathcal{TH}(\alpha)$ denote the marked translation surface corresponding to the marking map $\varphi_{\tilde{x},\nu}:(S,\Sigma)\rightarrow (M_{\tilde{x},\nu},\Sigma)$. 
    Let 
    \[V_{\tilde{x}}\coloneqq\{\tilde{y}_{\tilde{x},\nu}:\nu\in U_{\tilde{x}}\}\subset \mathcal{TH}(\alpha)\]
     and $\psi_{\tilde{x}}:U_{\tilde{x}}\rightarrow V_{\tilde{x}}$ be defined by
       \[\psi_{\tilde{x}}:\nu\mapsto\tilde{y}_{\tilde{x},\nu}.\] 
  Let $\phi:V_{\tilde{x}}\rightarrow U_{\tilde{x}}$ be the period map. 
 By construction, $\nu\in U_{\tilde{x}}$ agrees with $\phi(\tilde{y}_{\tilde{x},\nu})$  on edges of $\tau_{\tilde{x}}$, and these edges generate $H_{1}(S,\Sigma;\mathbb{Z})$. Thus, the map $\psi_{\tilde{x}}$ is an inverse to $\phi$. Thus, we obtain the following:
   \begin{lem}[{\cite[Lemma 1.1]{masur1991hausdorff}}]\label{effective2023.9.20} The map $\phi$ is injective and locally onto when restricted to $V_{\tilde{x}}$.
   \end{lem}

 Now let $x=(M,\omega)\in \mathcal{H}_{1}^{(\epsilon)}(\alpha)$; in other words, the shortest length of   saddle connections in $x$ is not smaller than $\epsilon$.  Let $\tau_{x}$ be  the Delaunay triangulation of $x$. Then by the construction, each triangle $\Delta\in\tau_{x}$ can be inscribed in a disk of radius not greater than the diameter $d(M)$ of $M$ (cf. {\cite[Theorem 4.4]{masur1991hausdorff}}). In addition, we can further control these quantities as follows.
 \begin{thm}[{\cite[Theorem 5.3, Proposition 5.4]{masur1991hausdorff}}]\label{effective2023.9.19}
 \hypertarget{2024.08.C6}   Let   $x=(M,\omega)\in \mathcal{H}_{1}(\alpha)$.  Then there exists a constant  $C_{6}>0$, such that for any $p\in\Sigma$, $M\setminus B(p,\hyperlink{2024.08.C6}{C_{6}})$ is contained in a union of disjoint metric cylinders.
   
   Moreover, the Delaunay triangulation $\tau_{x}$ of $x$ consists of edges which either have length $\leq \hyperlink{2024.08.C6}{C_{6}}$ or which cross a cylinder $C\subset M$ whose height $h$ is greater than its circumference $c$. If an edge crosses $C$, then its length $l$ satisfies $h\leq l\leq \sqrt{h^{2}+c^{2}}$. 
 \end{thm}
 After calculating the area, we immediately obtain:
 \begin{cor}\label{effective2023.9.22}
    Let $\epsilon>0$, $x=(M,\omega)\in \mathcal{H}_{1}^{(\epsilon)}(\alpha)$. Then the length $l$ of any edge of the Delaunay triangulation $\tau_{x}$ of $x$ is bounded above by $l\leq 2\epsilon^{-1}$. Also, the diameter $d(M)\leq 2 \epsilon^{-1}$.
 \end{cor}
 \begin{proof} Note that for any $x=(M,\omega)\in \mathcal{H}_{1}^{(\epsilon)}(\alpha)$, the circumference of a cylinder is not less than $\epsilon$.  Let $C\subset M$  be a cylinder  with height $h$ and circumference $c\geq\epsilon$. Then one can calculate the area $ch=\Area(C)\leq \Area(M)=1$.  The consequence follows from Theorem \ref{effective2023.9.19} immediately.
 \end{proof}

     We shall also need certain elementary analysis of inscribed triangles. Let $T_{1}$ be the space of ordered triples of points in $\mathbb{C}\cong\mathbb{R}^{2}$ modulo the action of the group of translations. Let $T_{2} \subset T_{1}$ be the set of triples with positive determinant. Let $T_{3}(\epsilon,d)\subset T_{2}$ be the set of isometry classes of triangles, with all edges of length not less than $\epsilon$, which can be inscribed in circles of radius not greater than $d$.  
      \begin{lem}[{\cite[Lemma 6.7]{masur1991hausdorff}}]\label{effective2023.9.21} \hypertarget{2024.08.C7}  There exists a constant $C_{7}>0$ satisfying the following property. Let $\epsilon,d>0$, $\Delta\in T_{3}(\epsilon,d)$. Let $\Delta^{\prime}\in T_{3}(\epsilon,d)$ be a triangle such that each vertex of $\Delta^{\prime}$ differs from the corresponding  vertex of $\Delta$ by at most $\hyperlink{2024.08.C7}{C_{7}}\epsilon^{2}/d$. Then  $\Delta^{\prime}$ is a non-degenerate triangle with the same orientation as $\Delta$.
      \end{lem}
      
      \begin{cor}\label{closing2024.1.4} Let $\epsilon>0$ be small enough, and let $x\in\mathcal{H}_{1}^{(\epsilon)}(\alpha)$.  Let $\tilde{x}\in\mathcal{TH}(\alpha)$ satisfy $\pi_{1}(\tilde{x})=x$. Let $\tilde{y}\in B(\tilde{x},\epsilon^{5})$. Suppose that $\Delta\in\tau_{\tilde{x}}$ is a triangle, and  $\delta_{1},\delta_{2}$ are two directed edges of     $\Delta$.  Then $\tilde{y}(\delta_{1}),\tilde{y}(\delta_{2})$  are not parallel.
      \end{cor}
      \begin{proof} Recall that $\tau_{\tilde{x}}$ is the Delaunay triangulation of $\tilde{x}$.  Then by Corollary \ref{effective2023.9.22}, the  lengths of edges  of the triangle $\Delta$ satisfy $\epsilon\leq |\tilde{x}(\gamma)|\leq 4\epsilon^{-1}$.  Then 
      \[|\tilde{y}(\gamma)-\tilde{x}(\gamma)|\leq \|\tilde{y}-\tilde{x}\|_{\tilde{x}}\cdot |\tilde{x}(\gamma)|\leq 4 \epsilon^{4} \ll \epsilon^{2}/d(x).\]  
      Then by Lemma \ref{effective2023.9.21}, we get that $\Delta^{\prime}\in\tau_{\tilde{y}}$ generated by $\tilde{y}(\delta_{1}),\tilde{y}(\delta_{2})$ is  a non-degenerate triangle with the same orientation as $\Delta$. 
      \end{proof}
      
      \begin{cor}\label{closing2024.3.62} \hypertarget{2024.08.k7} Let $\kappa_{7}=\kappa_{7}(\alpha)=\hyperlink{2024.08.k6}{\kappa_{6}}+5>0$. Then for $x\in\mathcal{H}_{1}(\alpha)$, the composition of the affine exponential map and the covering map is injective on $R(\tilde{x},\ell(x)^{\hyperlink{2024.08.k7}{\kappa_{7}}})$.
      \end{cor}

  \section{Dynamics over $\mathcal{H}(2)$}\label{effective2024.07.01}
  \subsection{McMullen's  classification}\label{effective2024.9.2}
  In this section, we shall recall  the dynamics of $G=\SL_{2}(\mathbb{R})$ over $\mathcal{H}(2)$, and in particular the McMullen's classification of Teichm\"{u}ller curves over  $\mathcal{H}(2)$.  
    We  refer to $\mathcal{H}(0)=\Omega\mathcal{M}_{1}$ as  the moduli space of holomorphic $1$-forms of genus $1$ with a marked point. We usually identify elements of $\mathcal{H}(0)$   with   lattices $\Lambda\subset\mathbb{C}$, via the correspondence $(M,\omega)=(\mathbb{C}/\Lambda,dz)$; in other words, $\Lambda$ is the image of the absolute periods $\omega(H_{1}(M;\mathbb{Z}))$. Thus, $\mathcal{H}(0)\cong \GL_{2}^{+}(\mathbb{R})/\Gamma=\mathbb{R}^{+}\times G/\Gamma$.  (Here $G/\Gamma$ is identified with the space of tori of area $1$.)
       
        In the sequel, we shall study tori with different areas: $\Lambda_{1}\in  \mathcal{H}_{A}(0)$ and   $\Lambda_{2}\in  \mathcal{H}_{1-A}(0)$. We consider $(\Lambda_{1},\Lambda_{2})\in G/\Gamma\times G/\Gamma$ as two corresponding tori with area $1$, after rescaling.  
        
    For $(X,\omega)\in\mathcal{H}(2)$, we will be interested in presenting forms of genus $2$ as connected sums of forms of genus $1$, 
    \begin{equation}\label{closing2024.3.8}
      (X,\omega)=(E_{1},\omega_{1})\stackrel[I]{}{\#} (E_{2},\omega_{2}).
    \end{equation} 
Here $(E_{1},\omega_{1}),(E_{2},\omega_{2}) \in\mathcal{H}(0)$, $v\in\mathbb{C}^{\ast}$ and $I=[0,v]\coloneqq [0,1]\cdot v$. We also say   that $(E_{1},\omega_{1})\stackrel[I]{}{\#} (E_{2},\omega_{2})$ is a splitting of $(X,\omega)$.

It is straightforward to check that the connected sum operation commutes with the action of $\GL_{2}^{+}(\mathbb{R})$: we have 
\begin{equation}\label{two2023.6.14}
   g.((Y_{1},\omega_{1})\stackrel[I]{}{\#}(Y_{2},\omega_{2}))=   g.(Y_{1},\omega_{1})\stackrel[g\cdot I]{}{\#} g. (Y_{2},\omega_{2})
\end{equation}
for all $g\in\GL_{2}^{+}(\mathbb{R})$.

Let   $S(2)$ denote the splitting space, consisting of triples $(\Lambda_{1},\Lambda_{2},v)\in\mathcal{H}(0)\times\mathcal{H}(0)\times\mathbb{C}^{\ast}$  
satisfying 
\begin{equation}\label{two2023.6.16}
   [0,v]\cap\Lambda_{1}=\{0\},\ \ \  [0,v]\cap\Lambda_{2}=\{0,v\}
\end{equation}
or vice versa.  The group $\GL^{+}_{2}(\mathbb{R})$ acts on the space of  triples $(\Lambda_{1},\Lambda_{2},v)$, leaving $S(2)$ invariant.  Clearly,  given a triple $(\Lambda_{1},\Lambda_{2},v)$, one defines
\[ x=\Lambda_{1}\stackrel[{[0,v]}]{}{\#} \Lambda_{2}\in \mathcal{H}(2)\]
and we obtain a natural map  $\Phi:S(2)\rightarrow \mathcal{H}(2)$. By {\cite[Theorem 7.2]{mcmullen2007dynamics}}, the connected sum mapping  $\Phi$ is a surjective, $\GL^{+}_{2}(\mathbb{R})$-equivariant local covering map.

    In \cite{mcmullen2007dynamics}, McMullen    classified the $\SL_{2}(\mathbb{R})$-orbit  closures of $\mathcal{H}_{1}(2)$:
  \begin{thm}[{\cite[Theorem 10.1]{mcmullen2007dynamics}}]\label{effective2024.07.02}
     Let $Z=\overline{G.x}$ be a $G$-orbit closure of some $x\in\mathcal{H}_{1}(2)$. Then either:
     \begin{itemize}
       \item $Z$ is a Teichm\"{u}ller curve, or
       \item $Z=\mathcal{H}_{1}(2)$.
     \end{itemize}
  \end{thm} 
   We also have a simple criterion for Teichm\"{u}ller curves in $\mathcal{H}(2)$:
   \begin{thm}[{\cite[Corollary 5.6, Theorem 5.10]{mcmullen2007dynamics}}]\label{closing2024.3.19} Whether a form $(X,\omega)\in\mathcal{H}(2)$ generates a Teichm\"{u}ller curve or not, is  completely determined by the absolute periods. 
   \end{thm}
  
In addition, in \cite{mcmullen2005teichmullerDiscriminant}, McMullen provided a complete list of  Teichm\"{u}ller curves in $\mathcal{H}(2)$.  
 We say $\Omega W_{D}\subset\Omega\mathcal{M}_{2}$ is a \textit{Weierstrass curve}\index{Weierstrass curve} if it is the locus of Riemann surfaces $M\in\mathcal{M}_{2}$ such that 
  \begin{enumerate}[\ \ \ (i)]
    \item  $\Jac(M)$ admits real multiplication by $\mathcal{O}_{D}$, where $\mathcal{O}_{D}\cong\mathbb{Z}[x]/(x^{2}+bx+c)$ is a quadratic order with  $b,c\in\mathbb{Z}$ and the \textit{discriminant}\index{discriminant} $D=b^{2}-4c>0$ (cf. \cite{mcmullen2003billiards});
    \item $M$ carries an eigenform $\omega$ with a double zero at one of the six Weierstrass points of $M$.
  \end{enumerate}
  Every \textbf{irreducible} component of $W_{D}$ is a Teichm\"{u}ller curve. When $D\equiv 1\bmod 8$, one can also define a \textit{spin invariant}\index{spin invariant} $\epsilon(M,\omega)\in\mathbb{Z}/2\mathbb{Z}$ which is constant along the components of $W_{D}$.  In \cite{mcmullen2005teichmullerDiscriminant}, McMullen showed that each Teichm\"{u}ller curve is uniquely determined by these two invariants:
  \begin{thm}[{\cite[Theorem 1.1]{mcmullen2005teichmullerDiscriminant}}]\label{effective2023.10.1}
     For any integer $D\geq 5$ with $D\equiv0$ or $1\bmod 4$, either:
     \begin{itemize}
       \item The Weierstrass curve $W_{D}$ is irreducible, or
       \item We have $D\equiv 1\bmod 8$ and $D\neq 9$, in which case $W_{D}=W_{D}^{0}\sqcup W_{D}^{1}$ has exactly two components, distinguished by their spin invariants.
     \end{itemize}
  \end{thm} 
  We are interested in the information provided by the absolute periods.
  \begin{thm}[{\cite[Theorem 3.1]{mcmullen2005teichmullerDiscriminant}}]\label{effective2023.9.12}
     Let $(M,\omega)=(E_{1},\omega_{1})\stackrel[I]{}{\#} (E_{2},\omega_{2})$. Then the following are equivalent:
     \begin{enumerate}[\ \ \ (i)]
       \item  $\omega$ is an eigenform for real multiplication by $\mathcal{O}_{D}$ on $\Jac(M)$;
       \item $\omega_{1}+\omega_{2}$ is an eigenform for real multiplication by $\mathcal{O}_{D}$ on $E_{1}\times E_{2}$.
     \end{enumerate}
  \end{thm}
  In particular, the discriminant $D$ of a Weierstrass curve $W_{D}$ is purely determined by the absolute period map $I_{\omega}:H_{1}(M;\mathbb{Z})\rightarrow\mathbb{C}$. In addition, we have a complete list of all possible  absolute period maps. More precisely, we define the locus
  \begin{multline} 
    \Omega Q_{D}\coloneqq\{(E_{1}\times E_{2},\omega)\in\Omega\mathcal{M}_{1}\times\Omega\mathcal{M}_{1}: \\
    \omega \text{ is an eigenform for real muliplication by }\mathcal{O}_{D}\}.\nonumber
  \end{multline} 
 Besides, we define the splitting space
  \[\Omega W_{D}^{s}\coloneqq\{(X,\omega,I):(X,\omega)\in\Omega W_{D}\text{ splits along }I\}. \]
  Then by Theorem \ref{effective2023.9.12}, there is a covering map 
  \begin{equation}\label{closing2024.3.10}
    \Pi:\Omega W_{D}^{s}\rightarrow\Omega Q_{D}
  \end{equation} 
   which records the summands $(E_{i},\omega_{i})$ in (\ref{closing2024.3.8}). 
   
   In particular, suppose that $(\Lambda_{1},\Lambda_{2},I)\in\Omega W_{D}^{s}$ determines a surface $x=\Lambda_{1}\stackrel[I]{}{\#} \Lambda_{2}\in \Omega_{1} W_{D}$ and satisfies $\Area(\Lambda_{1})=\Area(\Lambda_{2})=\frac{1}{2}$, and $(\Lambda_{1},\Lambda_{2})\in X_{\eta}$ for some $\eta>0$. Then by Lemma \ref{effective2024.07.145} and (\ref{closing2024.3.10}), the injectivity radius of $x$ in $\Omega_{1}W_{D}$ is not less than  $\eta$. In other words,  we have 
   \begin{equation}\label{effective2024.07.146}
     x\in (\Omega_{1} W_{D})_{\eta}
   \end{equation} 
   where 
   \begin{equation}\label{effective2024.08.15}
     (\Omega_{1} W_{D})_{\eta}\coloneqq\{y\in \Omega_{1} W_{D}:\text{the injectivity radius of } y \text{ in }  \Omega_{1} W_{D} \text{ is not less than }\eta\}.
   \end{equation} 
   (See e.g. \cite[\S3.6]{einsiedler2009effective} for more discussion.)
  \subsection{Prototypes of eigenforms}\label{effective2024.9.6}
Let us say a triple of integers $(e,\ell,m)$ is a \textit{prototype}\index{prototype} for real multiplication, with discriminant $D$, if 
\begin{equation}\label{effective2023.9.14}
D=e^{2}+4\ell^{2}m,\ \ \ \ell,m>0,\ \ \  \gcd(e,\ell)=1.
\end{equation} 
 We can associate a prototype  $(e,\ell,m)$ to each eigenform $(E_{1}\times E_{2},\omega)\in  \Omega Q_{D}$. Moreover, we have
\begin{thm}[{\cite[Theorem 2.1]{mcmullen2005teichmullerDiscriminant}}]\label{effective2023.9.24}
   The space $\Omega Q_{D}$ decomposes into a finite union
   \[\Omega Q_{D}=\bigcup \Omega Q_{D}(e,\ell,m)\]
   of closed $\GL_{2}^{+}(\mathbb{R})$-orbits, one for each prototype $(e,\ell,m)$. Besides, we have
   \[\Omega Q_{D}(e,\ell,m)\cong\GL_{2}^{+}(\mathbb{R})/\Gamma_{0}(m)\]
   where $\Gamma_{0}(m)$ is the \textit{Hecke congruence subgroup}\index{Hecke congruence subgroup} of level $m$:
   \[\Gamma_{0}(m)\coloneqq \left\{\left[
            \begin{array}{cccc}
   a & b  \\
   c & d   \\
            \end{array}
          \right]\in\Gamma:c\equiv 0\bmod m\right\}.\]
\end{thm}  
Let $\lambda=(e+\sqrt{D})/2$. Define a pair of lattices in $\mathbb{C}$ by 
\begin{equation}\label{effective2023.9.13}
 \Lambda_{1}=\mathbb{Z}(\ell m,0)\oplus\mathbb{Z}(0,\ell),\ \ \         \Lambda_{2}=\mathbb{Z}(\lambda,0)\oplus\mathbb{Z}(0,\lambda).
\end{equation} 
Let $(E_{i},\omega_{i})=(\mathbb{C}/\Lambda_{i},dz)$ be the corresponding forms of genus $1$, and let 
\[(A,\omega)=(E_{1}\times E_{2},\omega_{1}+\omega_{2}).\]
Then $(A,\omega)$ is an eigenform with invariant $(e,\ell,m)$, and we refer to it as the \textit{prototypical example}\index{prototypical example} of type $(e,\ell,m)$.
\begin{cor}\label{effective2023.9.11} 
   Every eigenform $(E_{1}\times E_{2},\omega)\in\Omega Q_{D}$ is equivalent, under the action of $\GL_{2}^{+}(\mathbb{R})$, to a unique prototypical example.
\end{cor}
  \subsection{Prototypes of splittings}
Moreover, we can assign a prototypical splitting to a quadruple of integers $(a,b,c,e)$.
First, we 
 say a quadruple of integers $(a,b,c,e)$ is a \textit{prototype}\index{prototype} of discriminant $D$, if   
\begin{alignat}{5}
 &  D=e^{2}+4bc,\ \ \ & & 0\leq a<\gcd(b,c),\ \ \  & &   c+e<b, \nonumber\\
& b>0, & & c>0, &  & \gcd(a,b,c,e)=1. \label{closing2024.3.9}
\end{alignat} 
We then assign a prototypical splitting to a prototype quadruple $(a,b,c,e)$ as follows. Let 
\[
 \Lambda_{1}=\mathbb{Z}(b,0)\oplus\mathbb{Z}(a,c),\ \ \         \Lambda_{2}=\mathbb{Z}(\lambda,0)\oplus\mathbb{Z}(0,\lambda)\]
 and $\lambda=(e+\sqrt{D})/2$ and $D=e^{2}+4bc$ (see Figure \ref{closing2024.3.7}).
\begin{figure}[H]
\centering

\tikzset{every picture/.style={line width=0.75pt}} %set default line width to 0.75pt        

\begin{tikzpicture}[x=0.75pt,y=0.75pt,yscale=-1,xscale=1]
%uncomment if require: \path (0,398); %set diagram left start at 0, and has height of 398

%Straight Lines [id:da006401447578689323] 
\draw    (165,71) -- (301.11,71) ;
%Straight Lines [id:da5307513707542981] 
\draw    (165,71) -- (121.11,140.67) ;
\draw [shift={(121.11,140.67)}, rotate = 122.21] [color={rgb, 255:red, 0; green, 0; blue, 0 }  ][fill={rgb, 255:red, 0; green, 0; blue, 0 }  ][line width=0.75]      (0, 0) circle [x radius= 3.35, y radius= 3.35]   ;
%Straight Lines [id:da9614056656746255] 
\draw    (301.11,71) -- (257.22,140.67) ;
\draw [shift={(257.22,140.67)}, rotate = 122.21] [color={rgb, 255:red, 0; green, 0; blue, 0 }  ][fill={rgb, 255:red, 0; green, 0; blue, 0 }  ][line width=0.75]      (0, 0) circle [x radius= 3.35, y radius= 3.35]   ;
\draw [shift={(301.11,71)}, rotate = 122.21] [color={rgb, 255:red, 0; green, 0; blue, 0 }  ][fill={rgb, 255:red, 0; green, 0; blue, 0 }  ][line width=0.75]      (0, 0) circle [x radius= 3.35, y radius= 3.35]   ;
%Straight Lines [id:da6845906896875882] 
\draw    (121.11,140.67) -- (257.22,140.67) ;
%Straight Lines [id:da3923303932562874] 
\draw    (121.11,140.67) -- (205.44,140.67) ;
%Straight Lines [id:da3823616546326596] 
\draw    (121.11,225) -- (121.11,140.67) ;
%Straight Lines [id:da09598806243363556] 
\draw    (121.11,225) -- (205.44,225) ;
\draw [shift={(121.11,225)}, rotate = 0] [color={rgb, 255:red, 0; green, 0; blue, 0 }  ][fill={rgb, 255:red, 0; green, 0; blue, 0 }  ][line width=0.75]      (0, 0) circle [x radius= 3.35, y radius= 3.35]   ;
%Straight Lines [id:da6275964796737887] 
\draw    (205.44,225) -- (205.44,140.67) ;
\draw [shift={(205.44,140.67)}, rotate = 270] [color={rgb, 255:red, 0; green, 0; blue, 0 }  ][fill={rgb, 255:red, 0; green, 0; blue, 0 }  ][line width=0.75]      (0, 0) circle [x radius= 3.35, y radius= 3.35]   ;
\draw [shift={(205.44,225)}, rotate = 270] [color={rgb, 255:red, 0; green, 0; blue, 0 }  ][fill={rgb, 255:red, 0; green, 0; blue, 0 }  ][line width=0.75]      (0, 0) circle [x radius= 3.35, y radius= 3.35]   ;
%Straight Lines [id:da05664177276475235] 
\draw    (165,71) -- (249.33,71) ;
\draw [shift={(249.33,71)}, rotate = 0] [color={rgb, 255:red, 0; green, 0; blue, 0 }  ][fill={rgb, 255:red, 0; green, 0; blue, 0 }  ][line width=0.75]      (0, 0) circle [x radius= 3.35, y radius= 3.35]   ;
\draw [shift={(165,71)}, rotate = 0] [color={rgb, 255:red, 0; green, 0; blue, 0 }  ][fill={rgb, 255:red, 0; green, 0; blue, 0 }  ][line width=0.75]      (0, 0) circle [x radius= 3.35, y radius= 3.35]   ;

% Text Node
\draw (145,45) node [anchor=north west][inner sep=0.75pt]    {$( a,c)$};
% Text Node
\draw (267,135) node [anchor=north west][inner sep=0.75pt]    {$( b,0)$};
% Text Node
\draw (105,180) node [anchor=north west][inner sep=0.75pt]    {$\lambda$};
% Text Node
\draw (158,205) node [anchor=north west][inner sep=0.75pt]    {$\lambda$};

\end{tikzpicture}

  \caption{Prototypical splitting of type $(a,b,c,e)$.}
\label{closing2024.3.7}
\end{figure}
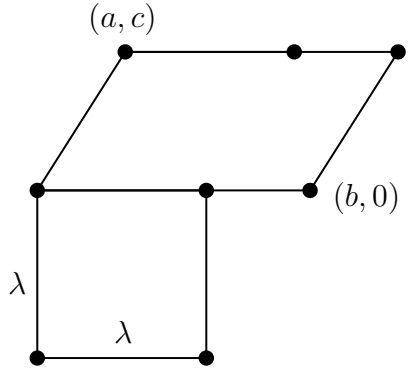
\noindent
Then the splittings $(X,\omega)=\Lambda_{1}\stackrel[{[0,v]}]{}{\#} \Lambda_{2}$ is said to be a \textit{prototypical splitting of type $(a,b,c,e)$}\index{prototypical splitting}. 
   In {\cite[\S3]{mcmullen2005teichmullerDiscriminant}}, McMullen  showed that all prototypical splittings are well-defined, and all other  splittings of $\Omega W^{s}_{D}$ can be generated by these prototypes via $\GL_{2}^{+}(\mathbb{R})$-action.

\begin{ex}[L-shaped polygon]\label{effective2023.9.30}
  For later use, we consider certain Teichm\"{u}ller curves generated by L-shaped polygons. See also \cite{mcmullen2003billiards}, \cite{calta2004veech}, \cite[\S3]{weiss2014twisted}.
   
   Let $4<D=d^{2}\in\mathbb{N}$ be an even square. Define a pair of lattice in $\mathbb{C}$ by
   \[\Lambda_{1,D}\coloneqq\mathbb{Z}(\sqrt{D}/2,0)\oplus\mathbb{Z}(0,1),\ \ \ \Lambda_{2,D}\coloneqq\mathbb{Z}(1,0)\oplus\mathbb{Z}(0,\sqrt{D}/2).\]
   Let $v_{D}\coloneqq (1,0)$. 
    Then $(\Lambda_{1,D},\Lambda_{2,D},v_{D})\in D(2)$ indicates an L-shaped polygon. In the notation of \cite{mcmullen2003billiards}, it corresponds to the L-shaped polygon $P(1+\sqrt{D}/2,\sqrt{D}/2)$.
     Let $y_{D} \in \mathcal{H}_{1}(2)$ correspond to  the surface $\Phi(\Lambda_{1,D},\Lambda_{2,D},v_{D})$ under the projection $\mathcal{H}(2)\rightarrow\mathcal{H}_{1}(2)$. Then it is easy to check that $\overline{G.(\Lambda_{1,D},\Lambda_{2,D})}=G.(\Lambda_{1,D},\Lambda_{2,D})$ and $y_{D}$ generates a Teichm\"{u}ller curve. Moreover, it corresponds to the prototype  $(0,1,D/4)$ in the sense of (\ref{effective2023.9.13}), and the prototype  $(0,D/4,1,0)$ in the sense of (\ref{closing2024.3.9}). Thus, by Theorem \ref{effective2023.9.24}, we have  
   \[Y_{D}\coloneqq G.(\Lambda_{1,D},\Lambda_{2,D})\cong G/\Gamma_{0}(D/4).\]
   Note that for $n\geq 2$, the index is given by the formula:
   \[[\SL_{2}(\mathbb{Z}):\Gamma_{0}(n)]=n\cdot\prod_{\substack{p|n\\ p \text{ prime} }}\left(1+\frac{1}{p}\right).\]
   In particular, letting $n=D/4$, we have
   \[ D \ll \vol(G.(\Lambda_{1,D},\Lambda_{2,D}))   \ll D^{2}.\] 
   In what follows, we write 
   \[Q_{D}\coloneqq\Omega  Q_{D}(0,1,D/4).\]
  
   On the other hand, note that $D$ is even and so by Theorem \ref{effective2023.10.1}, the Weierstrass curve $\Omega_{1}W_{D}$ is irreducible, i.e. a Teichm\"{u}ller curve. 
   Now note that the L-shaped polygon can be tiled by squares.  Thus, $\SL(y_{D})$ is a lattice commensurable to $\SL_{2}(\mathbb{Z})$ (\cite[Theorem 5.5]{gutkin2000affine}). In addition, one can show that $\SL(y_{D})\subset\SL_{2}(\mathbb{Z})$ (\cite[Proposition 3.13]{weiss2014twisted}).  
   
   Moreover, in \cite{eskin2003billiards}, Eskin, Masur, and Schmoll give the formula  
   \[[\SL_{2}(\mathbb{Z}):\SL(y_{D})]=\frac{3}{8}(d-2)d^{2}\prod_{\substack{p|d\\ p \text{ prime} }}\left(1-\frac{1}{p^{2}}\right).\]
   It follows  that  
   \[\vol(\Omega_{1}W_{D})=\vol(G/\SL(y_{D}))\ll  D^{3}.\]
\end{ex}

\subsection{Quantitative estimates of Teichm\"{u}ller curves}
In this section, we approach     a given splitting through Teichm\"{u}ller curves in terms of the areas of tori. With the help of the classification of Teichm\"{u}ller curves, one   obtains the following  lemma by an elementary calculation.

\begin{lem}\label{effective2024.07.128}
   Given a splitting $x=\Lambda_{1} \stackrel[I]{}{\#} \Lambda_{2}\in \mathcal{H}_{1}^{(\eta)}(2)$ induced by  the  Delaunay triangulation, there is a sufficiently large $ D_{\eta}>0$ so that for any $D\geq D_{\eta}$, the Teichm\"{u}ller curve $\Omega_{1}W_{D}$ has a prototypical splitting 
   \[    x_{D}=\Lambda_{1}(D) \stackrel[I(D)]{}{\#} \Lambda_{2}(D)\in \Omega_{1} W_{D}\]
   such that 
   \begin{equation}\label{effective2024.07.101}
    \frac{\Area(\Lambda_{1}(D))}{\Area(\Lambda_{2}(D))}=\frac{\Area(\Lambda_{1})+O(\eta^{-3}D^{-\frac{1}{2}})}{\Area(\Lambda_{2})+O(\eta^{-3}D^{-\frac{1}{2}})}.
   \end{equation} 
Here the implicit constant of $O(\eta^{-3}D^{-\frac{1}{2}})$ can be chosen to be $<3$. Moreover, we have the volume estimate
\begin{equation}\label{effective2024.07.138}
  D\ll\vol(G.(\Lambda_{1}(D),\Lambda_{2}(D)))\ll D^{2}.
\end{equation} 
\end{lem}
\begin{proof} Let 
\[\lambda=\frac{\Area(\Lambda_{1})}{\Area(\Lambda_{2})}\in [\eta^{4},\eta^{-4}].\]
Consider the prototype $(a,b,c,e)=(0,b,1,e)$ with $e=\left\lfloor(\lambda-1)\lambda^{-\frac{1}{2}}b^{\frac{1}{2}}\right\rfloor$.   (Note  that it is always possible when $b$ is sufficiently large.) Then $e=(\lambda-1)\lambda^{-\frac{1}{2}}b^{\frac{1}{2}}-\epsilon$ with $\epsilon\in[0,1)$.  Let 
\[c_{1}=\epsilon b^{-\frac{1}{2}}(-2(\lambda-1)\lambda^{-\frac{1}{2}}+\epsilon b^{-\frac{1}{2}})=O(\eta^{-2}b^{-\frac{1}{2}}).\] 
Then one calculates
 \begin{align}
  \frac{\Area(\Lambda_{1}(D))}{\Area(\Lambda_{2}(D))}=& \frac{(e^{2}+4b)^{\frac{1}{2}}+e}{(e^{2}+4b)^{\frac{1}{2}}-e}\;\nonumber\\ 
  =&   \frac{((\lambda-1)^{2}\lambda^{-1}b+c_{1}b+4b)^{\frac{1}{2}}+e}{((\lambda-1)^{2}\lambda^{-1}b+c_{1}b+4b)^{\frac{1}{2}}-e}\;\nonumber\\
  =&   \frac{((\lambda+1)^{2}\lambda^{-1}b+c_{1}b)^{\frac{1}{2}}+e}{((\lambda+1)^{2}\lambda^{-1}b+c_{1}b)^{\frac{1}{2}}-e}\;\nonumber\\
   =&  \frac{\lambda^{-\frac{1}{2}}b^{\frac{1}{2}}((\lambda+1)^{2}+c_{1}\lambda)^{\frac{1}{2}}+e}{\lambda^{-\frac{1}{2}}b^{\frac{1}{2}}((\lambda+1)^{2}+c_{1}\lambda)^{\frac{1}{2}}-e}\;\nonumber\\
  =&  \frac{\lambda^{-\frac{1}{2}}b^{\frac{1}{2}}((\lambda+1)+O(c_{1}\lambda))+e}{\lambda^{-\frac{1}{2}}b^{\frac{1}{2}}((\lambda+1)+O(c_{1}\lambda))-e}\;\nonumber\\
    =&  \frac{2\lambda\lambda^{-\frac{1}{2}}b^{\frac{1}{2}}+\lambda^{-\frac{1}{2}}b^{\frac{1}{2}}\cdot O(c_{1}\lambda)-\epsilon}{2\lambda^{-\frac{1}{2}}b^{\frac{1}{2}}+\lambda^{-\frac{1}{2}}b^{\frac{1}{2}}\cdot O(c_{1}\lambda)+\epsilon}\;\nonumber\\
    =&  \frac{\lambda+\frac{1}{2}\cdot O(c_{1}\lambda)-\frac{1}{2}\epsilon\lambda^{\frac{1}{2}}b^{-\frac{1}{2}}}{1+\frac{1}{2}\cdot O(c_{1}\lambda)+\frac{1}{2}\epsilon\lambda^{\frac{1}{2}}b^{-\frac{1}{2}}} .\;  \label{effective2024.07.100}
\end{align}
Here the implicit constant of $O(c_{1}\lambda)$ can be chosen to be $<3$. Now the consequence follows from (\ref{effective2024.07.100}) and $D=e^{2}+4b=O(\eta^{-2}b)$.

Finally, the absolute periods of  $\Omega W_{D}$ of  type $(0,b,1,e)$ corresponds to $\Omega Q_{D}$ of type $(e,1,b)$. Then by Theorem \ref{effective2023.9.24}, we get
\[\vol(G.(\Lambda_{1}(D),\Lambda_{2}(D)))=\vol(G/\Gamma_{0}(b))\asymp b\cdot\prod_{\substack{p|b\\ p \text{ prime} }}\left(1+\frac{1}{p}\right)\in[D, D^{2}]. \]
This establishes the lemma.
\end{proof}
         Let    $ x=\Lambda_{1}\stackrel[{[0,v]}]{}{\#} \Lambda_{2}\in \mathcal{H}_{1}(2)$  be a splitting induced by  the  Delaunay triangulation satisfying
         \begin{equation}\label{effective2024.07.135}
           [0,v]\cap\Lambda_{1}=\{0\},\ \ \  [0,v]\cap\Lambda_{2}=\{0,v\}.
         \end{equation} 
          Using the identification $H_{1}(M;\mathbb{Z})=\Lambda_{1}\oplus\Lambda_{2}$, we define
          \[x_{1}(\lambda_{1}+\lambda_{2})\coloneqq x(\lambda_{1}).\]
           Then for $\epsilon$ with $|\epsilon|$ sufficiently small, we see that $x+\epsilon x_{1}=(1+\epsilon)\Lambda_{1}\stackrel[{[0,v]}]{}{\#} \Lambda_{2}$ is  a change of the areas:
          \[\frac{\Area((1+\epsilon)\Lambda_{1})}{\Area(\Lambda_{2})}=(1+\epsilon)\frac{\Area(\Lambda_{1})}{\Area(\Lambda_{2})}.\]
      
       \begin{lem}\label{effective2024.07.134}
          For   $\epsilon$ with $|\epsilon|$ sufficiently small, for $x\in \mathcal{H}_{1}^{(\eta)}(2)$, we have 
          \[\|x-(x+\epsilon x_{1})\|_{x}=\|\epsilon x_{1}\|_{x}\leq  \eta^{-11}|\epsilon|.\]
       \end{lem}
       \begin{proof}
        For  $ x=\Lambda_{1}\stackrel[{[0,v]}]{}{\#} \Lambda_{2}\in \mathcal{H}_{1}^{(\eta)}(2)$,  let     $\Lambda_{1}=\mathbb{Z}a_{1}\oplus\mathbb{Z}b_{1}$, $\Lambda_{2}=\mathbb{Z}a_{2}\oplus\mathbb{Z}b_{2}$,  for some $a_{i},b_{i}\in\mathbb{C}^{\ast}$ with $a_{2}=v$.  Then we can identify $x=\Lambda_{1}\stackrel[{[0,v]}]{}{\#} \Lambda_{2}$ as three parallelograms (see Figure \ref{effective2024.07.131}):
          \[P_{1}=a_{2}\times b_{2},\ \ \ P_{2}=a_{2}\times b_{1},\ \ \ P_{3}=(a_{1}-a_{2})\times b_{1}.\]
     
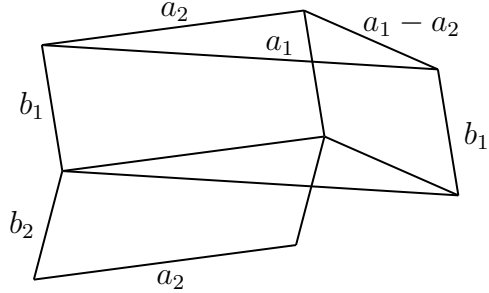
\begin{figure}[H]
\centering

\tikzset{every picture/.style={line width=0.75pt}} %set default line width to 0.75pt        

\begin{tikzpicture}[x=0.75pt,y=0.75pt,yscale=-1,xscale=1]
%uncomment if require: \path (0,398); %set diagram left start at 0, and has height of 398

%Straight Lines [id:da5541632611975358] 
\draw [color={rgb, 255:red, 255; green, 255; blue, 255 }  ,draw opacity=1 ][fill={rgb, 255:red, 255; green, 255; blue, 255 }  ,fill opacity=1 ]   (198.7,100.57) -- (208.9,163.95) -- (275.7,193.57) -- (265.49,130.19) -- cycle ;
%Straight Lines [id:da5139839880051282] 
\draw [color={rgb, 255:red, 255; green, 255; blue, 255 }  ,draw opacity=1 ][fill={rgb, 255:red, 255; green, 255; blue, 255 }  ,fill opacity=1 ]   (78.11,181.33) -- (208.9,163.95) -- (198.7,100.57) -- (67.9,117.95) ;
%Straight Lines [id:da17960335486654055] 
\draw [color={rgb, 255:red, 255; green, 255; blue, 255 }  ,draw opacity=1 ][fill={rgb, 255:red, 255; green, 255; blue, 255 }  ,fill opacity=1 ]   (78.11,181.33) -- (63.9,235.95) -- (194.7,218.57) -- (208.9,163.95) -- cycle ;
%Straight Lines [id:da157004987503361] 
\draw    (265.49,130.19) -- (275.7,193.57) ;
%Straight Lines [id:da9044738417151523] 
\draw    (78.11,181.33) -- (67.9,117.95) ;
%Straight Lines [id:da700596629119506] 
\draw    (265.49,130.19) -- (198.7,100.57) ;
%Straight Lines [id:da43808485798885455] 
\draw [fill={rgb, 255:red, 255; green, 255; blue, 255 }  ,fill opacity=1 ]   (63.9,235.95) -- (78.11,181.33) ;
%Straight Lines [id:da9471142520386411] 
\draw [fill={rgb, 255:red, 255; green, 255; blue, 255 }  ,fill opacity=1 ]   (63.9,235.95) -- (194.7,218.57) ;
%Straight Lines [id:da3140864511493031] 
\draw    (208.9,163.95) -- (198.7,100.57) ;
%Straight Lines [id:da7756949293620119] 
\draw    (67.9,117.95) -- (265.49,130.19) ;
%Straight Lines [id:da438615943385116] 
\draw    (78.11,181.33) -- (275.7,193.57) ;
%Straight Lines [id:da25479158576195693] 
\draw [fill={rgb, 255:red, 255; green, 255; blue, 255 }  ,fill opacity=1 ]   (78.11,181.33) -- (208.9,163.95) ;
%Straight Lines [id:da9065839889171219] 
\draw [fill={rgb, 255:red, 255; green, 255; blue, 255 }  ,fill opacity=1 ]   (67.9,117.95) -- (198.7,100.57) ;
%Straight Lines [id:da7223076075971544] 
\draw    (275.7,193.57) -- (208.9,163.95) ;
%Straight Lines [id:da2793542199357506] 
\draw [fill={rgb, 255:red, 255; green, 255; blue, 255 }  ,fill opacity=1 ]   (194.7,218.57) -- (208.9,163.95) ;

% Text Node
\draw (178,112) node [anchor=north west][inner sep=0.75pt]    {$a_{1}$};
% Text Node
\draw (55,140) node [anchor=north west][inner sep=0.75pt]    {$b_{1}$};
% Text Node
\draw (50,200) node [anchor=north west][inner sep=0.75pt]    {$b_{2}$};
% Text Node
\draw (123.61,230) node [anchor=north west][inner sep=0.75pt]    {$a_{2}$};
% Text Node
\draw (277,155) node [anchor=north west][inner sep=0.75pt]    {$b_{1}$};
% Text Node
\draw (125.61,95) node [anchor=north west][inner sep=0.75pt]    {$a_{2}$};
% Text Node
\draw (226.61,100) node [anchor=north west][inner sep=0.75pt]    {$a_{1} -a_{2}$};

\end{tikzpicture}

  \caption{Pair of splittings.}
\label{effective2024.07.131}
\end{figure}
\noindent
See e.g. \cite[\S7]{mcmullen2005teichmullerDiscriminant} for more discussion about this identification.
          
 Let $\|\cdot\|$ be a norm  on $H_{1}(M;\mathbb{Z})$ define by:
  \[\|z_{1}a_{1}+z_{2}b_{1}+z_{3}a_{2}+z_{4}b_{2}\|\coloneqq\max_{i}|z_{i}|.\]
          Suppose that we have a long saddle connection $\gamma$ passes through the parallelograms. We can use the sides of $P_{j}$ to subdivide $\gamma$ in $n$ segments $[X_{i},X_{i+1}]$ so that each of them lies in a single   parallelogram $P_{j(i)}$ for $j(i)=1,2,3$. Joining each $X_{i}$ to the nearest singularity of the side, we get a decomposition in homology $[\gamma]=\sum_{i=1}^{n}[\gamma_{i}]$, where $\gamma_{i}$ is a saddle connection in  $P_{j(i)}$. Clearly, $\|[\gamma_{i}]\|\leq 1$. Further, we have:
          \begin{cla}\label{effective2024.07.130}
           If $x([X_{i}X_{i+1}])<\eta^{10}$, then $\|[\gamma_{i}]\|=0$.
          \end{cla}
          \begin{proof}[Proof of Claim \ref{effective2024.07.130}]
           Let $\Delta$ be a triangle in one of the parallelograms $P_{j}$, with three sides $a,b,c$ and corresponding three angles $\alpha,\beta,\gamma$. Suppose that 
          \[\sin\gamma=\max\{\sin\alpha,\sin\beta,\sin\gamma\}\geq\frac{1}{2}.\]
          Since the sides of the parallelograms  have the length within $[\eta,2\eta^{-1}]$ (via Lemma \ref{effective2023.9.22}), by the law of sines, we know that any angle $\alpha$ of the parallelograms has 
          \[\sin\alpha=\frac{a}{c}\sin\gamma\geq \epsilon^{2}\sin\gamma\geq \frac{1}{2}\eta^{2}.\]
           
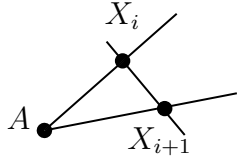
\begin{figure}[H]
\centering

\tikzset{every picture/.style={line width=0.75pt}} %set default line width to 0.75pt        

\begin{tikzpicture}[x=0.75pt,y=0.75pt,yscale=-1,xscale=1]
%uncomment if require: \path (0,505); %set diagram left start at 0, and has height of 505

%Straight Lines [id:da7060433198749665] 
\draw    (219.9,350.95) ;
\draw [shift={(219.9,350.95)}, rotate = 0] [color={rgb, 255:red, 0; green, 0; blue, 0 }  ][fill={rgb, 255:red, 0; green, 0; blue, 0 }  ][line width=0.75]      (0, 0) circle [x radius= 3.35, y radius= 3.35]   ;
%Straight Lines [id:da6827037192740633] 
\draw    (211.9,341.24) -- (250.9,388.24) ;
%Straight Lines [id:da4912806032208292] 
\draw [fill={rgb, 255:red, 255; green, 255; blue, 255 }  ,fill opacity=1 ]   (180.9,385.95) -- (276.9,367.24) ;
\draw [shift={(180.9,385.95)}, rotate = 348.97] [color={rgb, 255:red, 0; green, 0; blue, 0 }  ][fill={rgb, 255:red, 0; green, 0; blue, 0 }  ][line width=0.75]      (0, 0) circle [x radius= 3.35, y radius= 3.35]   ;
%Straight Lines [id:da10275162790487591] 
\draw [fill={rgb, 255:red, 255; green, 255; blue, 255 }  ,fill opacity=1 ]   (180.9,385.95) -- (246.9,327.24) ;
%Straight Lines [id:da7226511302765684] 
\draw    (240.9,374.95) ;
\draw [shift={(240.9,374.95)}, rotate = 0] [color={rgb, 255:red, 0; green, 0; blue, 0 }  ][fill={rgb, 255:red, 0; green, 0; blue, 0 }  ][line width=0.75]      (0, 0) circle [x radius= 3.35, y radius= 3.35]   ;

% Text Node
\draw (209.61,319.83) node [anchor=north west][inner sep=0.75pt]    {$X_{i}$};
% Text Node
\draw (220.9,382.6) node [anchor=north west][inner sep=0.75pt]    {$X_{i+1}$};
% Text Node
\draw (160.61,372.83) node [anchor=north west][inner sep=0.75pt]    {$A$};

\end{tikzpicture}

  \caption{Pair of splittings.}
\label{effective2024.07.133}
\end{figure}
          Suppose that some $|x([X_{i}X_{i+1}])|<\eta^{10}$.  Then together with the two sides of $P_{j(i)}$ that contain $X_{i}$ and $X_{i+1}$, we obtain a triangle $\Delta X_{i}AX_{i+1}$, where $A$ is the endpoint of the two sides of $P_{j(i)}$. By using the law of sines again, we see that 
          \[|x([X_{i}A])|=\frac{\sin\angle AX_{i+1}X_{i}}{\sin\angle X_{i}AX_{i+1}}|x([X_{i}X_{i+1}])|\leq \frac{1}{\frac{1}{2}\eta^{2}}\cdot\eta^{10}\leq 2\eta^{8}.\]
          Similarly, we have $|x([X_{i+1}A])|\leq 2\eta^{8}$. Since each side of $P_{j}$ has the length at least $\eta$, we conclude that $X_{i}$ and $X_{i+1}$ are both joined by $A$. Thus, $x([\gamma_{i}])=0$.
          \end{proof} 
        Thus, we obtain
          \[\|[\gamma]\|\leq \sum_{\|[\gamma_{i}]\|\neq0}\|[\gamma_{i}]\|\leq \eta^{-10} \sum_{\|[\gamma_{i}]\|\neq0} |x([X_{i}X_{i+1}])|\leq  \eta^{-10} |x([\gamma])|.\]
          Then we have
          \[\|\epsilon x_{1}\|_{x}=\sup_{\gamma}\frac{|\epsilon x_{1}([\gamma])|}{|x([\gamma])|}\leq \eta^{-10}\sup_{\gamma}\frac{|\epsilon x_{1}([\gamma])|}{\|[\gamma]\|}\leq  \eta^{-11}|\epsilon|.\]
          The consequence follows.
       \end{proof}
       \begin{rem}\label{effective2024.08.2}
          The proof of Lemma \ref{effective2024.07.134} only uses  the facts that $x\in\mathcal{H}_{1}^{(\eta)}(2)$, and the lengths of the sides of parallelograms $(\Lambda_{1},\Lambda_{2})$ are controlled by $\eta$ (more precisely, by $[\eta,2\eta^{-1}]$). Thus, if there is another splitting $x=\Lambda_{1}^{\prime}\stackrel[I^{\prime}]{}{\#} \Lambda_{2}^{\prime}\in \mathcal{H}_{1}^{(\eta)}(2)$ (not necessarily induced by the Delaunay triangulation) so that the lengths of the sides of parallelograms are controlled by some power of $\eta$, we can obtain a similar result to Lemma \ref{effective2024.07.134}.
       \end{rem}

Therefore, for $x\in \mathcal{H}_{1}^{(\eta)}(2)$ satisfying (\ref{effective2024.07.135}), for $\epsilon>0$, let 
\begin{equation}\label{effective2024.07.136}
  x(\epsilon)=(1+\epsilon)^{\frac{1}{2}}\Lambda_{1}\stackrel[{(1+\epsilon)^{-\frac{1}{2}}[0,v]}]{}{\#} (1+\epsilon)^{-\frac{1}{2}}\Lambda_{2}.
\end{equation} 
Then by Lemma \ref{effective2024.07.134}, we conclude that  $x(\epsilon)\in \mathcal{H}_{1}(2)$,
\[\frac{\Area((1+\epsilon)^{\frac{1}{2}}\Lambda_{1})}{\Area((1+\epsilon)^{-\frac{1}{2}}\Lambda_{2})}=(1+\epsilon)\frac{\Area(\Lambda_{1})}{\Area(\Lambda_{2})}.\]
  Now note that $x(\epsilon)=\left[
            \begin{array}{cccc}
   (1+\epsilon)^{-\frac{1}{2}} & 0 \\
   0 & (1+\epsilon)^{-\frac{1}{2}}   \\
            \end{array}
          \right](x+\epsilon x_{1})$, by Lemma   \ref{effective2024.07.145}, we get
  \begin{equation}\label{effective2024.07.137}
    \|x(\epsilon)-x\|_{x}\leq \|x(\epsilon)-(x+\epsilon x_{1})\|_{x}+\|(x+\epsilon x_{1})-x\|_{x}\leq |\epsilon|+ \eta^{-11}|\epsilon|.
  \end{equation}
  
\section{Effective results in homogeneous dynamics}\label{effective2024.08.10}

\subsection{Spectral gaps}
Let $M$ a $n$-dimensional locally symmetric space. Let $\Delta$ be the Laplace-Beltrami operator of $M$. Then $\Delta$ is self-adjoint and positive. Thus, its eigenvalues are real and nonnegative. Moreover, the constants are the only eigenfunctions with respect to the zero eigenvalue. Let $\lambda_{1}(M)$ be the smallest positive eigenvalue. For $\phi\in   C^{\infty}_{c}(M)$,
let $\mathcal{S}(\phi)$ be the Sobolev norm of $\phi$ with certain order
(See e.g. \cite[\S3.7]{einsiedler2009effective}, \cite[\S2.9]{venkatesh2010sparse} for more discussion on Sobolev norms).

 In \cite{ratner1987rate}, Ratner proved that it is related to the rate of mixing for geodesic flows:
\begin{thm}[Rate of mixing, {\cite[Theorem 1]{ratner1987rate}}]\label{effective2024.07.141}
  Let $G=\SL_{2}(\mathbb{R})$,   $\Gamma^{\prime}$   a  \hypertarget{2024.08.k8} lattice in $G$,  $M=G/\Gamma^{\prime}$, and $\mu$ the normalized Lebesgue measure on $M$. Then for any $\phi,\psi\in C^{\infty}_{c}(M)$, we have 
  \[\left|\int \phi(a_{t}x)\psi(x) d\mu(x)-\int \phi  d\mu \int  \psi  d\mu \right|\ll  \mathcal{S}(\phi) \mathcal{S}(\psi)e^{-\hyperlink{2024.08.k8}{\kappa_{8}}t}\]
  where $\kappa_{8}=\kappa_{8}(M)$ can be chosen to be larger than $\frac{1}{2}(1-\sqrt{1-\lambda_{1}(M)})$.
\end{thm}

The well known Cheeger's inequality provides a lower bound of $\lambda_{1}$ in terms of the  Cheeger constant on Riemannian manifolds \cite{cheeger1970lower}.  
Let $E$ be a $(n-1)$-submanifold which divides $M$ into two disjoint submanifolds $A$ and $B$. Let $\mu(E)$ be the area of $E$ and $\vol(A)$ and $\vol(B)$ the volumes of $A$ and $B$, respectively. The \textit{Cheeger constant}\index{Cheeger constant} of $M$ is defined to be 
\[h(M)=\inf_{E}\frac{\mu(E)}{\min\{\vol(A),\vol(B)\}}\]
 where $E$ runs over all possibilities for $E$ as above. 

\begin{thm}[Cheeger's inequality, {\cite{cheeger1970lower}}]\label{effective2024.07.140}
  For a locally symmetric space $M$, we have 
 \[\lambda_{1}(M)\geq \frac{1}{4}h^{2}(M).\]
\end{thm}

It is convenient to calculate the Cheeger constant through the combinatorics.   
 The \textit{Cheeger constant}\index{Cheeger constant} of a graph $\mathcal{G}=(V,E)$ is defined to be
 \[h(\mathcal{G})=\inf_{A,B\subset V}\frac{|E(A,B)|}{\min\{|A|,|B|\}}\]
 where the infimum runs over all the  disjoint partitions $V=A\sqcup B$, and $E(A,B)$ is the set of edges connecting vertices  in $A$ to vertices in $B$.  The Cheeger constant $h(\mathcal{G})$ is related to the expander graphs. More precisely, if $\mathcal{G}$ is a $k$-regular graph, then $\mathcal{G}$ is an $(|V|,k,\frac{1}{k}h(\mathcal{G}))$-expander (cf. \cite[Proposition 1.1.4]{lubotzky1994discrete}).

Let $D=d^{2}$ be an even square. Let $y_{D}\in\mathcal{H}_{1}(2)$ be a surface that is tiled by $d$ squares.  Then $\Gamma_{D}=\SL(y_{D})$ is a Veech group with discriminant $D$. Moreover,  $\Gamma_{D}\subset\Gamma=\SL_{2}(\mathbb{Z})$ \cite[Proposition 3.13]{weiss2014twisted}, and $\Gamma_{D}$ is arithmetic \cite[Theorem 5.5]{gutkin2000affine}. Let $\Omega_{1}W_{D}=G/\Gamma_{D}$. Then $\{\Omega_{1}W_{D}\}$ is a sequence of locally symmetric spaces that are the finite sheeted coverings of $X=G/\Gamma$. 
In the  context of the graph theory, the $\Gamma$-orbit of $y_{D}$ is closed. This leads to  a Schreier graph $\mathcal{G}_{D}=\mathcal{G}(\Gamma/\Gamma_{D},S)$ with $S=\left\{\left[
            \begin{array}{cccc}
   1 & 1  \\
   0 & 1   \\
            \end{array}
          \right],\left[
            \begin{array}{cccc}
   1 & 0  \\
   1 & 1   \\
            \end{array}
          \right]\right\}$. In \cite{brooks1986spectral}, Brooks proved the following:

\begin{lem}[{\cite[Lemma 2]{brooks1986spectral}}]\label{effective2024.07.139}
   There \hypertarget{2024.08.C8} is a constant $C_{8}=C_{8}(X)>0$ such that  
   \[h(\mathcal{G}_{D})\leq \hyperlink{2024.08.C8}{C_{8}} h(\Omega_{1}W_{D}).\]
\end{lem}

Since $\Gamma$ acts transitively on $\Gamma/\Gamma_{D}$, we have
 \[h(\mathcal{G}_{D})\geq\inf_{A,B\subset \Gamma/\Gamma_{D}}\frac{1}{\min\{|A|,|B|\}}\geq \frac{1}{[\Gamma:\Gamma_{D}]}\gg D^{-3}.\]
Then by Lemma \ref{effective2024.07.139} and Theorem \ref{effective2024.07.140}, we have 
\[\lambda_{1}(\Omega_{1}W_{D})\gg h^{2}(\mathcal{G}_{D})\gg D^{-6}.\] 
 Moreover, McMullen made the following conjecture:
 \begin{conj}[McMullen's expansion conjecture]\label{effective2024.07.144}
   The family of graphs $\mathcal{G}_{D}$ associated to arithmetic Veech groups is expander. In other words, $\lambda_{1}(\Omega_{1}W_{D})$ possess a uniform lower bound.
 \end{conj}
  Therefore, by Theorem \ref{effective2024.07.141}, there exists a \hypertarget{2024.08.C9}  constant $C_{9}>0$ such that  the rate of mixing 
 \begin{equation}\label{effective2024.07.142}
   \hyperlink{2024.08.k8}{\kappa_{8}}(\Omega_{1}W_{D})>\hyperlink{2024.08.C9}{C_{9}} D^{-6}.
 \end{equation} 
 If Conjecture \ref{effective2024.07.144} is correct,  then $\lambda_{1}(\Omega_{1}W_{D})$ possess a uniform lower bound. Then we have 
  \begin{equation}\label{effective2024.07.143}
   \hyperlink{2024.08.k8}{\kappa_{8}}(\Omega_{1}W_{D})>\hyperlink{2024.08.C9}{C_{9}}.
 \end{equation} 
 
\subsection{Effective equidistribution in homogeneous dynamics}
  In this section, we review the effective results of homogeneous dynamics which shall   serve as the estimates of the absolute periods.

  In \cite{lindenstrauss2022effective} (see also \cite{lindenstrauss2023polynomial}), Lindenstrauss, Mohammadi, and Wang derive the following quantitative behavior of orbits in homogeneous dynamics.
  \begin{thm}[Effective equidistribution on $X$]\label{effective2024.07.107}
     Let $X=G/\Gamma\times G/\Gamma$. There \hypertarget{2024.08.k9} exists $\delta_{0}=\delta_{0}(X)>0$, $\kappa_{9}=\kappa_{9}(X)>0$ such that for every $\Lambda\in X$,  $\delta\in(0,\delta_{0})$, there exists   $t_{0}=t_{0}(X,\ell(\Lambda))>0$  so that for any $t>\frac{1}{\delta} t_{0}$, 
        at least one of the the following holds:
   \begin{enumerate}[\ \ \ (1)]
     \item  For every $\varphi\in C^{\infty}_{c}(X)$, we have
     \[\left|\int_{0}^{1}\varphi (a_{t}u_{r}\Lambda)dr-\int\varphi dm_{X}\right|\leq \mathcal{S}(\varphi)e^{- \hyperlink{2024.08.k9}{\kappa_{9}}\delta t} \]
     where $m_{X}$ is the Haar measure on $X$ with $m_{X}(X)=1$.
     \item There exists $\Lambda^{\prime}\in X$ such that $G.\Lambda^{\prime}$ is periodic with $\vol(G\Lambda^{\prime})\leq e^{\delta t}$ and
         \[d_{X}(\Lambda^{\prime},\Lambda)\leq  e^{-\frac{1}{2}t}.\]
   \end{enumerate} 
  \end{thm}
  
 We also need the effective equidistribution of unstable foliations on Teichm\"{u}ller curves. $G/\Gamma\times G/\Gamma$. 
  See also {\cite[Proposition 2.4.8]{kleinbock1996bounded}}, \cite[Proposition 4.1]{lindenstrauss2023polynomial}.

  \begin{itemize}
  \item For even $D>0$, let $\Omega_{1}W_{D}\coloneqq G/\SL(y_{D})$ be the Teichm\"{u}ller curve of discriminant $D$.
  \item    Let $y_{D}\in \mathcal{H}(2)$ be a surface that generates a Teichm\"{u}ller curve of discriminant $D$. In what follows, we shall mainly study the  Teichm\"{u}ller curve with even discriminant $D$. In this case, $y_{D}$  determines a unique Veech group $\Gamma_{D}=\SL(y_{D})$ and a unique Teichm\"{u}ller curve $\Omega_{1}W_{D}=G/\Gamma_{D}$.
  \item  (Recall (\ref{effective2024.08.15})) For $\eta>0$, let  
  \[ (\Omega_{1}W_{D})_{\eta}\coloneqq\{y\in W_{D}:\text{the injectivity radius of } y \text{ in } \Omega_{1}W_{D} \text{ is not less than }\eta\}.\]
 
  \item Let $\vol$ be the ($G$-invariant) volume measure on $G/\Gamma$. Note that  any $\Gamma_{D}\subset\Gamma=\SL_{2}(\mathbb{Z})$ has finite index:
       \[[\Gamma:\SL(y_{D})]\leq D^{3} \] 
       (the   precisely formula for the indexes is given in \cite{eskin2003billiards}). Thus, it also induces a ($G$-invariant) volume measure on $\Omega_{1}W_{D}=G/\SL(y_{D})$.
  \end{itemize}
  \begin{prop}[Effective equidistribution on $\Omega_{1}W_{D}$, {\cite[Proposition 2.4.8]{kleinbock1996bounded}}]\label{effective2024.07.149}
    For any \hypertarget{2024.08.k10} even $D>0$, there    \hypertarget{2024.07.k11}  exist   \hypertarget{2024.08.C10}      $\kappa_{10}>0$,   $ \kappa_{11}>0$, and  $C_{10}=C_{10}(\hyperlink{2024.08.C4}{C_{4}},\hyperlink{2024.08.k4}{\kappa_{4}})>0$   so that the following holds:  Let $\eta\in(0,1)$, $t>0$, and $x\in (\Omega_{1}W_{D})_{\eta}$. Then for every $f \in C_{c}^{\infty}(X)+\mathbb{C}\cdot 1$, we have
\begin{align}
  \left|\int_{0}^{1}f(a_{t}u_{r}x)dr-\frac{1}{\vol(\Omega_{1}W_{D})}\int f d\vol\right|\leq & \hyperlink{2024.08.C10}{C_{10}} (\vol(\Omega_{1}W_{D}))^{ \hyperlink{2024.08.k10}{\kappa_{10}}}\eta^{-\frac{1}{\hyperlink{2024.08.k11}{\kappa_{11}}}}\mathcal{S}(f)e^{-\hyperlink{2024.08.k8}{\kappa_{8}}t}\;\nonumber\\
\leq & \hyperlink{2024.08.C10}{C_{10}} D^{3 \hyperlink{2024.08.k10}{\kappa_{10}}}\eta^{-\frac{1}{\hyperlink{2024.08.k11}{\kappa_{11}}}}\mathcal{S}(f)e^{-\hyperlink{2024.08.k8}{\kappa_{8}}t}.\;  \nonumber
\end{align}  
where $\hyperlink{2024.08.k8}{\kappa_{8}}=\hyperlink{2024.08.k8}{\kappa_{8}}(\Omega_{1}W_{D})$ is the rate of mixing given in Theorem  \ref{effective2024.07.141}.
  \end{prop} 
  \begin{proof}
     We adopt the notation as in \cite{lindenstrauss2023polynomial}. Let $\eta\in(0,1)$, $t>0$, and $x\in (\Omega_{1}W_{D})_{\eta}$. By \cite[Proposition 4.1]{lindenstrauss2023polynomial}, there exists $C>0$, and an absolute constant   $\hyperlink{2024.08.k11}{\kappa_{11}}$ such that for every $f \in C_{c}^{\infty}(X)+\mathbb{C}\cdot 1$, we have
\[\left|\int_{0}^{1}f(a_{t}u_{r}x)dr-\frac{1}{\vol(\Omega_{1}W_{D})}\int f d\vol\right|\leq  C\eta^{-\frac{1}{\hyperlink{2024.08.k11}{\kappa_{11}}}}\mathcal{S}(f)e^{-\hyperlink{2024.08.k8}{\kappa_{8}}t}.\]
Here the constant $C$ is dominated by 
\[C\leq L(\eta_{D}^{-1}\vol(\Omega_{1}W_{D}))^{L}\]
where $L$ is absolute and $\eta_{D}$ is so that the quantitative nondivergence
 \[|\{r\in [0,1]: a_{t}u_{r}x\in (\Omega_{1}W_{D})_{\eta_{D}}\}|\geq 0.99\]
 holds for $t\geq |\inj(x)|+\eta_{D}^{-1}$. However, since $\Omega_{1}W_{D}\subset\mathcal{H}_{1}(2)$, by Corollary \ref{closing2024.3.63}, one can choose $\eta_{D}=(0.01\hyperlink{2024.08.C4}{C_{4}}^{-1})^{\frac{1}{\hyperlink{2024.08.k4}{\kappa_{4}}}}$; in particular, it does not depend on $D$. The consequence follows from letting $\hyperlink{2024.08.k10}{\kappa_{10}}=L$ and $\hyperlink{2024.08.C10}{C_{10}}=L\eta_{D}^{-L}$.
  \end{proof}

   Also, we recall the density of periodic $G$-orbit on the homogeneous space $X=G/\Gamma\times G/\Gamma$.  
 \begin{thm}[{Density of periodic $G$-orbit on $X$, \cite[Theorem 1.3]{lindenstrauss2023polynomial}}]\label{effective2022.2.12} \hypertarget{2024.08.k12} 
  Let $G.\Lambda\subset X$ be \hypertarget{2024.08.C11} a periodic $G$-orbit in $X$.  
  Then there exist $\kappa_{12}=\kappa_{12}(X)>0$ and $C_{11}=C_{11}(X)>0$ such that 
  for every $z^{\ast}\in X_{\vol(G\Lambda)^{-\hyperlink{2024.08.k12}{\kappa_{12}}}}$, we have 
  \[d_{X}(z^{\ast},G\Lambda)\leq \hyperlink{2024.08.C11}{C_{11}}\vol(G\Lambda)^{-\hyperlink{2024.08.k12}{\kappa_{12}}}.\]  
\end{thm}
\section{Proof of main theorems}\label{effective2024.07.161}
 \begin{lem}\label{effective2024.07.151}
For $L>\hyperlink{2024.08.C11}{C_{11}}$, for even $D\geq L^{\frac{2(22+\hyperlink{2024.08.k7}{\kappa_{7}})}{\hyperlink{2024.08.k12}{\kappa_{12}}}}$, there exists $z_{D}\in \Omega_{1}W_{D}$ such that 
\begin{equation}\label{effective2024.07.154}
\|z-z_{D}\|_{z}\leq    D^{-\frac{1}{2}\hyperlink{2024.08.k12}{\kappa_{12}}}.
\end{equation} 
 \end{lem}
 \begin{proof} Let $\lambda=\Area(\Lambda_{1})/\Area(\Lambda_{2})$.
     By Lemma \ref{effective2024.07.128}, for $D\geq D_{L^{-1}}$,   the Teichm\"{u}ller curve $\Omega W_{D}$ has a prototypical splitting 
   \[    x_{D}=\Lambda_{1}(D) \stackrel[I(D)]{}{\#} \Lambda_{2}(D)\in \Omega W_{D}\]
   such that  
   \begin{align}
 \frac{\Area(\Lambda_{1}(D))}{\Area(\Lambda_{2}(D))}&  =\frac{\Area(\Lambda_{1}(z))+O(L^{3}D^{-\frac{1}{2}})}{\Area(\Lambda_{2}(z))+O(L^{3}D^{-\frac{1}{2}})}\;\nonumber\\
 & =\lambda\left(1+\frac{O(L^{3}D^{-\frac{1}{2}})-\lambda O(L^{3}D^{-\frac{1}{2}})}{\lambda(\Area(\Lambda_{2})+O(L^{3}D^{-\frac{1}{2}}))}\right).\;  \label{effective2024.07.129}
\end{align}
Note that $\Area(\Lambda_{2})\in[L^{-2},L^{2}]$, and $\lambda\in[L^{-4},L^{4}]$. Let 
\[\delta=\frac{O(L^{3}D^{-\frac{1}{2}})-\lambda O(L^{3}D^{-\frac{1}{2}})}{\lambda(\Area(\Lambda_{2})+O(L^{3}D^{-\frac{1}{2}}))}.\]
Then $|\delta|\leq L^{10}D^{-\frac{1}{2}}$. Let $z(\delta)$ be as in (\ref{effective2024.07.136}). Then by (\ref{effective2024.07.137}),  we have 
\[\|z-z(\delta)\|_{z}\leq  L^{12}|\delta|\leq L^{22}D^{-\frac{1}{2}}.\]

   On the other hand, by Theorem \ref{effective2022.2.12} and (\ref{effective2024.07.138}), there exists $\tilde{\Lambda}\in G(\Lambda_{1}(D),\Lambda_{2}(D))$ such that 
   \[\|z(\delta)-\tilde{\Lambda}\|_{z}\leq \hyperlink{2024.08.C11}{C_{11}}D^{-\hyperlink{2024.08.k12}{\kappa_{12}}}.\]  
   Since $L>\hyperlink{2024.08.C11}{C_{11}}$, $D\geq L^{\frac{2(22+\hyperlink{2024.08.k7}{\kappa_{7}})}{\hyperlink{2024.08.k12}{\kappa_{12}}}}$, it follows that 
   \[\|z-\tilde{\Lambda}\|_{z}\leq L^{22}D^{-\frac{1}{2}}+ \hyperlink{2024.08.C11}{C_{11}}D^{-\hyperlink{2024.08.k12}{\kappa_{12}}}\leq L^{22}D^{-\hyperlink{2024.08.k12}{\kappa_{12}}}\leq D^{-\frac{\hyperlink{2024.08.k12}{\kappa_{12}}}{2}}.\]
   (To simplify the notation, we assume without loss of generality that $\hyperlink{2024.08.k12}{\kappa_{12}}<\frac{1}{2}$.)
 By Corollary \ref{closing2024.3.62}, if $L^{22}D^{-\hyperlink{2024.08.k12}{\kappa_{12}}}<L^{-\hyperlink{2024.08.k7}{\kappa_{7}}}$, then there exists $z_{D}\in \Omega_{1}W_{D}$, such that the absolute period is $\tilde{\Lambda}$ and $\|z-z_{D}\|_{z}\leq   D^{-\frac{\hyperlink{2024.08.k12}{\kappa_{12}}}{2}}$. 
 \end{proof}

We are now in the position to prove the main theorems.

\begin{proof}[Proof of Theorem \ref{effective2024.9.1}]
We use the equidistribution  of  Teichm\"{u}ller curves to approach the  given surface $z=\Lambda_{1}(z)\stackrel[I(z)]{}{\#} \Lambda_{2}(z)\in\mathcal{H}_{1}^{(L^{-1})}(2)$.
\begin{itemize}    
  \item Let $\hyperlink{2024.08.k8}{\kappa_{8}}=\hyperlink{2024.08.k8}{\kappa_{8}}(\Omega_{1}W_{D})$ be the rate of mixing given in Theorem  \ref{effective2024.07.141}.
  \item  Let $\rho=e^{-\hyperlink{2024.08.k8}{\kappa_{8}}t/2N}$.
  \item Let $y_{D}\in(\Omega_{1}W_{D})_{\eta}$. 
  \item Let $z_{D}$ be as in Lemma \ref{effective2024.07.151}.
  \item Let $h_{\rho,z_{D}}$ be a function supported on $\mathsf{B}_{G}(\rho).z_{D}$ satisfying 
      \[\frac{1}{\vol(\Omega_{1}W_{D})}\int h_{\rho,z_{D}}d\vol=1\] and $\mathcal{S}(h_{\rho,z_{D}})\leq \rho^{-N}$, where $N$ is absolute. 
\end{itemize}
 Then by Proposition \ref{effective2024.07.149}, for $s>0$, we have
 \[\left|\int_{0}^{1}h_{\rho,z_{D}}(a_{s}u_{r}y_{D})dr-1\right|\leq  \hyperlink{2024.08.C10}{C_{10}} D^{3 \hyperlink{2024.08.k10}{\kappa_{10}}}\eta^{-\frac{R^{2}}{\hyperlink{2024.08.k11}{\kappa_{11}}}}e^{-\frac{\hyperlink{2024.08.k8}{\kappa_{8}}}{2}s}. \] 
Assuming $s$ is large enough, depending on $D$, the right side of the above is $<1/2$. Thus, $a_{s}u_{r^{\prime}}y_{D}\in\supp(g_{\rho,z_{D}})$ for some $r^{\prime}\in[0,1]$,  i.e. 
\begin{equation}\label{effective2024.07.152}
d(z_{D}, a_{s}u_{r^{\prime}}y_{D})\ll e^{-\frac{\hyperlink{2024.08.k8}{\kappa_{8}}}{2N}s}.
\end{equation} 

Thus, let 
$x\in\mathcal{H}_{1}(2)$ and $\tilde{y}_{D}\in (\Omega_{1}W_{D})_{\eta}$ satisfy
\[ \|\tilde{y}_{D}-\tilde{x}\|_{\tilde{x}}< e^{-  t}.\]  
 Then by Theorem \ref{effective2023.11.3}, we see that 
\begin{equation}\label{effective2024.07.153}
  \|a_{s}u_{r^{\prime}}y_{D}-a_{s}u_{r^{\prime}}\tilde{x}\|_{a_{s}u_{r^{\prime}}\tilde{x}}<e^{2s-\hyperlink{2024.08.k7}{\kappa_{7}} t}.
\end{equation} 
Therefore, combining (\ref{effective2024.07.154})(\ref{effective2024.07.152})(\ref{effective2024.07.152})(\ref{effective2024.07.153}), we conclude that 
\[ d(z, a_{s}u_{r^{\prime}}x)\leq D^{-\frac{1}{2}\hyperlink{2024.08.k12}{\kappa_{12}}}+ e^{-\frac{\hyperlink{2024.08.k8}{\kappa_{8}}}{4N}s}+e^{2s- t}.\]  
Finally, we simplify the right hand side by taking 

\begin{itemize}
  \item $D$ satisfying $D\geq L^{\frac{2(22+\hyperlink{2024.08.k7}{\kappa_{7}})}{\hyperlink{2024.08.k12}{\kappa_{12}}}}$ and $D^{-\frac{1}{2}\hyperlink{2024.08.k12}{\kappa_{12}}}<\frac{1}{3}L^{-1}$,
  \item $s$ satisfying $\hyperlink{2024.08.C10}{C_{10}} D^{3 \hyperlink{2024.08.k10}{\kappa_{10}}}\eta^{-\frac{R^{2}}{\hyperlink{2024.08.k11}{\kappa_{11}}}}e^{-\frac{\hyperlink{2024.08.k8}{\kappa_{8}}}{2}s}<1/2$ and $e^{-\frac{\hyperlink{2024.08.k8}{\kappa_{8}}}{4N}s}<\frac{1}{3}L^{-1}$,
  \item $t$ satisfying  $e^{2s- t}<\frac{1}{3}L^{-1}$. 
\end{itemize} 
Then  by (\ref{effective2024.07.142}), one calculates that there  exist  $\hyperlink{2024.08.C1}{C_{1}}>0$, and $\hyperlink{2024.08.k1}{\kappa_{1}}>0$ such that if $L>L_{0}$, $D>L^{\hyperlink{2024.08.k1}{\kappa_{1}}}$ and $t>\hyperlink{2024.08.C1}{C_{1}}D^{\hyperlink{2024.08.k1}{\kappa_{1}}}$, then 
\begin{equation}\label{effective2024.07.155}
  d(z, a_{t}u_{[0,1]}x)  \leq L^{-1}
\end{equation} 
Moreover, if  Conjecture \ref{effective2024.07.144} is correct, then by (\ref{effective2024.07.143}), we only require  $t>\hyperlink{2024.08.C1}{C_{1}}\log L$ to get (\ref{effective2024.07.155}). 
\end{proof}

Next, we discuss the proof of Theorem \ref{effective2024.07.148}.
 Let $x=\Lambda_{1}\stackrel[I]{}{\#} \Lambda_{2}\in\mathcal{H}_{1}(2)$ with $\Area(\Lambda_{1})=\Area(\Lambda_{2})$.  
 \begin{itemize}
   \item  Let $\delta$ be small that will be determined later (see (\ref{effective2024.07.120})). 
   \item Let $t$ be sufficiently large.
 \end{itemize}
 We apply Theorem \ref{effective2024.07.107} with $\Lambda(x)=(\Lambda_{1},\Lambda_{2}),\delta,t$.
 Suppose that Theorem \ref{effective2024.07.107}(1) occurs. Let $L>0$, $z\in \mathcal{H}_{1}^{(L^{-1})}(2)$. Let $z=\Lambda_{1}(z)\stackrel[I(z)]{}{\#} \Lambda_{2}(z)$ be a splitting of $z$ induced by  the  Delaunay triangulation.

\begin{proof}[Proof of Theorem \ref{effective2024.07.148}] 
  First, we apply Theorem \ref{effective2024.07.107} with $\Lambda(x)=(\Lambda_{1},\Lambda_{2}),\delta,t$.
Then we   use the equidistribution of $x=\Lambda_{1}\stackrel[I]{}{\#} \Lambda_{2}\in\mathcal{H}_{1}(2)$ in absolute periods to approach   a Teichm\"{u}ller curve.  More precisely, 
\begin{itemize}
  \item Let $d\in\mathbb{N}^{\ast}$ and $D=4d^{2}$ satisfying $D\geq L^{\hyperlink{2024.08.k1}{\kappa_{1}}}$. 
  \item   Let   $\varrho=\varrho_{d}$ be as in (\ref{effective2024.10.12}). 
  \item   Let $\varphi_{\varrho,d}$ be a function so that
  \[\chi_{\mathsf{B}_{G\times G}(\frac{9}{10}\varrho).(Q_{D})_{\eta}}\leq \varphi_{\varrho,d}\leq \chi_{\mathsf{B}_{G\times G}(\varrho).(Q_{D})_{\eta}}\]  
  with $\mathcal{S}(\varphi_{\varrho,d}) \leq \varrho^{-N}$, where $N\geq 3$ is absolute. In particular, we have 
   \[\int \varphi_{\varrho,d}dm_{X}\asymp \varrho^{3}.\] 
\end{itemize}  Then   at least one of the   following holds:
\begin{enumerate}[\ \ \ (1)]
  \item  For any $\tilde{\Lambda}\in X_{\eta}$,  we see that  
\[\left|\int_{0}^{1}\varphi_{\varrho,d}(a_{t}u_{r}(\Lambda_{1},\Lambda_{2}))dr-\int \varphi_{\varrho,d}dm_{X}\right|\leq  \varrho^{-N} e^{-\hyperlink{2024.08.k9}{\kappa_{9}}\delta t}\leq  e^{-\frac{\hyperlink{2024.08.k9}{\kappa_{9}}\delta}{2}t}.\] 
  \item   There exists $\Lambda^{\prime}\in X$ such that $G.\Lambda^{\prime}$ is periodic with $\vol(G\Lambda^{\prime})\leq e^{\delta t}$ and
      \begin{equation}\label{effective2024.10.3}
       d_{X}(\Lambda^{\prime},\Lambda)\leq  e^{-\frac{1}{2}t}.
      \end{equation} 
\end{enumerate} Suppose that the case (2) of Theorem \ref{effective2024.07.148}  does not occur.  
Suppose also that $t$ is large enough, so that  
\[ \varrho^{ 3}-e^{-\frac{\hyperlink{2024.08.k9}{\kappa_{9}}\delta}{2}t}\ll\int_{0}^{1}\varphi_{\varrho,d}(a_{t}u_{r}(\Lambda_{1},\Lambda_{2}))dr\ll \varrho^{3}+e^{-\frac{\hyperlink{2024.08.k9}{\kappa_{9}}\delta}{2}t}.\] 
Thus, there is a subset $J_{d}\subset [0,1]$ with $|J_{d}|\asymp \varrho^{3}$  such that $a_{t}u_{r}(\Lambda_{1},\Lambda_{2})\in\supp(\varphi_{\varrho,d})$ for $r\in J_{d}$. In other words, for 
\[J_{d}=J_{d}(t)=\left\{ r\in [0,1]: \|\tilde{\Lambda}-a_{t}u_{r}(\Lambda_{1},\Lambda_{2})\|\leq  \varrho,\ \tilde{\Lambda}\in (Q_{D})_{\eta}\right\},\] 
we have
\begin{equation}\label{effective2024.9.5}
 \left|J_{d}(\varrho)\right|\asymp   \varrho^{3}
\end{equation}  
for $\rho_{d}\geq e^{-\frac{\hyperlink{2024.08.k9}{\kappa_{9}}\delta}{2N}t}$.

Suppose that   the case (3) of Theorem \ref{effective2024.07.148}  does not occur. Then there is   $r_{0}\in J_{d}$, and $\tilde{\Lambda}_{d}\in  (Q_{D})_{\eta}$  such that 
\begin{equation}\label{effective2024.10.4}
\ell(a_{t}u_{r_{0}}\tilde{x})^{\vartheta}\|\tilde{\Lambda}_{d}-a_{t}u_{r_{0}}\tilde{x}\|_{a_{t}u_{r}\tilde{x}}  < \varkappa_{\vartheta}(e^{t}).
\end{equation}

Fix    $\xi=\frac{1}{\hyperlink{2024.08.k7}{\kappa_{7}}+\vartheta}$. 
Suppose that $x_{0}=a_{t}u_{r_{0}}x$ satisfies
 \begin{equation}\label{effective2024.07.105}
 \ell(x_{0})>\varkappa_{\vartheta}(e^{t})^{\xi}.
 \end{equation} 
   Then by Corollary \ref{closing2024.3.62}, there exists a surface $y^{\prime}_{d}\in (\Omega_{1}W_{D})_{\eta}$ with the   absolute periods   $\tilde{\Lambda}_{d}$ such that 
 \begin{equation}\label{effective2024.07.104}
   \|y_{d}^{\prime}-x_{0}\|_{x_{0}}< \ell(x_{0})^{-\vartheta}\varkappa_{\vartheta}(e^{t})\leq \ell(x_{0})^{\hyperlink{2024.08.k7}{\kappa_{7}}}. 
 \end{equation}  
 
In Section \ref{effective2024.07.147}, we shall see that even if (\ref{effective2024.07.105}) fails,    Proposition \ref{effective2024.07.127} indicates that there exists  and a  function $\varkappa_{\vartheta}^{\star}:\mathbb{R}^{+}\rightarrow\mathbb{R}^{+}$  depending on $\varkappa_{\vartheta}$, $\hyperlink{2024.08.k7}{\kappa_{7}}$, $\vartheta$ with $\lim_{t\rightarrow\infty}\varkappa^{\star}_{\vartheta}(e^{t})=0$, and  an interval  
 $J^{\star\star}\subset[0,e^{t}]$ with $|J^{\star\star}|\geq(\varkappa_{\vartheta}^{\star}(e^{t}))^{-1}$, such that  for  
 \[J_{d}^{\prime}=\left\{ r\in J^{\star\star}: \|\tilde{\Lambda}-u_{r}a_{t}\Lambda(x)\|\leq  \varrho,\ \tilde{\Lambda}\in (Q_{D})_{\eta^{R^{2}+1}}\right\},\]  
 we have 
\[ \left|J_{d}^{\prime}\right|\gg   \varrho^{3} |J^{\star\star}| ,\] 
and   at least  one of the following holds:
    \begin{enumerate}[\ \ \ (i)]
    \item  For any $r\in J_{d}^{\prime}$, and $\tilde{\Lambda}_{d}\in  (Q_{D})_{\eta^{R^{2}}}$, we have
    \[    \ell(u_{r}a_{t}\tilde{x})\geq  \varkappa_{\vartheta}(|J^{\star\star}|)^{\hyperlink{2024.08.k7}{\kappa_{7}}+\vartheta}, \ \ \   \ell(u_{r}a_{t}\tilde{x})^{\vartheta} \|\tilde{\Lambda}_{d}-u_{r}a_{t}\tilde{x}\|_{u_{r}a_{t}\tilde{x}}\geq \varkappa_{\vartheta}(|J^{\star\star}|).  \]  
      \item   There is a surface  a time $r^{\star}\in[0,1]$,     surface  $x^{\star}=a_{t}u_{r^{\star}}x\in\mathcal{H}_{1}(2)$, and a surface $y_{d}^{\star}\in (\Omega_{1}W_{D})_{\eta^{R^{2}}}$, such that
            \[ \|y_{d}^{\star}-x^{\star}\|_{x^{\star}}<\varkappa_{\vartheta}(|J^{\star\star}|) \leq \ell(x^{\star})^{\hyperlink{2024.08.k7}{\kappa_{7}}}.\] 
    \end{enumerate} 
Case (i) implies  Theorem \ref{effective2024.07.148} (3). Case (ii) implies that the Teichm\"{u}ller curve $\Omega_{1}W_{D}$ is $\varkappa_{\vartheta}(|J^{\star\star}|)$-close to $a_{t}u_{[0,1]}x$. 
  
  Then by Theorem \ref{effective2024.9.1}, we conclude that  for any $L>0$, if 
  \[ \log \varkappa_{\vartheta}^{\star}(e^{t})^{-1}> \hyperlink{2024.08.C1}{C_{1}}D^{\hyperlink{2024.08.k1}{\kappa_{1}}} \geq  \hyperlink{2024.08.C1}{C_{1}} L^{\hyperlink{2024.08.k1}{\kappa_{1}}^{2}},\]
 (or $   \varkappa_{\vartheta}^{\star}(e^{t})^{-1}>  D^{\hyperlink{2024.08.C1}{C_{1}}}\geq L^{\hyperlink{2024.08.C1}{C_{1}}\hyperlink{2024.08.k1}{\kappa_{1}} }$ assuming Conjecture \ref{effective2024.07.144} is correct),  then 
     \[d(z,a_{2t}u_{[0,2]}x)\leq  L^{-1} \] 
  for every $z\in \mathcal{H}_{1}^{(L^{-1})}(2)$. 
\end{proof}

 As we have seen above, we can use Teichm\"{u}ller curves to approximate a given surface. We can apply it to $x$ and obtain Theorem \ref{effective2024.07.159}.
 \begin{proof}[Proof of Theorem \ref{effective2024.07.159}] Let $\delta\in(0,\delta_{0})$, $x\in\mathcal{H}_{1}(2)$, $L\geq L_{0}(\delta,\frac{1}{2}\ell(x),\frac{1}{2}\ell(x))$ and $ t>\hyperlink{2024.08.C1}{C_{1}}L^{\hyperlink{2024.08.k1}{\kappa_{1}}}\log L$. Discussing as in the proof of Theorem \ref{effective2024.07.148}, one can find an arithmetic Teichm\"{u}ller curve $\Omega_{1}W_{D^{\prime}}$ for sufficiently large $D^{\prime}$, such that there exists $x^{\prime}\in \Omega_{1}W_{D^{\prime}}$ with $d(x^{\prime},x)<e^{-100t}$. Then $\ell(x^{\prime})>\frac{1}{2}\ell(x)$. In addition, $x^{\prime}$ has a splitting $x^{\prime}=\Lambda_{1}^{\prime}\stackrel[I^{\prime}]{}{\#} \Lambda_{2}^{\prime}$ satisfying $\Area(\Lambda_{1}^{\prime})=\Area(\Lambda_{2}^{\prime})$ (e.g. Example \ref{effective2023.9.30}). Then by Lemma \ref{effective2024.07.145} and (\ref{closing2024.3.10}), we have $\ell(\Lambda_{1}^{\prime},\Lambda_{2}^{\prime})=\ell(x^{\prime})$. 
 
 Now apply Theorem \ref{effective2024.07.148} to $\delta$, $x^{\prime}$, $(\Lambda_{1}^{\prime},\Lambda_{2}^{\prime})$, $L$, and $t$. Note that although the $G$-orbit $G.(\Lambda_{1}^{\prime},\Lambda_{2}^{\prime})\subset X$ is closed, it is far from of small volume. Thus, Theorem \ref{effective2024.07.148} implies that either $x^{\prime}$ is effectively dense or $(\Lambda_{1}^{\prime},\Lambda_{2}^{\prime})$ has a $G$-closed orbit of small volume nearby. If Theorem \ref{effective2024.07.148}(2) occurs, then there is  $(\Lambda^{\prime\prime}_{1},\Lambda^{\prime\prime}_{2})\in X$ such that $G.(\Lambda^{\prime\prime}_{1},\Lambda^{\prime\prime}_{2})$ is periodic with $\vol(G.(\Lambda^{\prime\prime}_{1},\Lambda^{\prime\prime}_{2}))\leq e^{\delta t}$ and
         \[\|(\Lambda^{\prime\prime}_{1},\Lambda^{\prime\prime}_{2})-(\Lambda_{1}^{\prime},\Lambda_{2}^{\prime})\|\leq e^{-\frac{1}{2}t}.\] 
 Since the injectivity radius of $x^{\prime}$ is greater than  $e^{-\frac{1}{2}t}$ by the choice of $t$, there is a surface $x^{\prime\prime}\in\mathcal{H}_{1}(2)$ with absolute periods $(\Lambda^{\prime\prime}_{1},\Lambda^{\prime\prime}_{2})$.
 \end{proof}

\section{Sparse cover argument}\label{effective2024.07.147}
In Section \ref{effective2024.07.161}, we have seen that  in order to make the proof of Theorem \ref{effective2024.07.148} work, we require that certain surface of the form $a_{t}u_{r}x$ does not go to infinity (see (\ref{effective2024.07.105})). More precisely, in the  proof of Theorem \ref{effective2024.07.148}, we use the equidistribution to get that the absolute periods of some $a_{t}u_{r}x$ are close to a Teichm\"{u}ller curve $\Omega_{1}W_{D}$. However, if $\ell(a_{t}u_{r}x)$ is too small, then   $a_{t}u_{r}x$  fails to approach   $\Omega_{1}W_{D}$. 

To fix this, we observe the nondivergence behavior of horocycle flows. For simplicity, we assume that $ \{s\in I:\ell(u_{s}a_{t}x)<\eta\}$ is dominated by a single $(C,\alpha)$-good function (cf. Theorem \ref{effective2024.07.11}). Then by the nature of $(C,\alpha)$-good functions, one observes that if  $\ell(u_{e^{t}r}a_{t}x)$ is very small, say $<e^{-\xi^{2}t}$, for some $r$ and $\xi$, then there is an interval $J\subset[0,e^{t}]$ near $e^{t}r$ with the length $|J|>e^{\xi^{2}t}$ so that $u_{s}a_{t}x\in \mathcal{H}^{(\eta)}_{1}(2)$ for all $s\in J$. Then we apply the equidistribution to $u_{s}a_{t}x$ for $s\in J$ to approach  $\Omega_{1}W_{D}$ again (with a weaker rate). Then we obtain the desired condition (see Proposition \ref{effective2024.07.127}).
 \subsection{Sparse covers}
 In this section, we review the sparse cover argument of horocycle flows in homogeneous and Teichm\"{u}ller dynamics. See the references \cite{kleinbock1998flows,minsky2002nondivergence} therein.

 For any $q\in\mathcal{TH}_{1}(\alpha)$, let $\mathscr{L}_{q}$ denote all the saddle connections for $q$. Recall that a saddle connection $\delta:[0,1]\rightarrow M$ is a geodesic segment joining two zeroes in $\Sigma$  or a zero to itself which has no zeroes in its interior.  Then $\delta$ associates to a vector
      \[q(\delta)=(x(\delta,q),y(\delta,q))\in\mathbb{C}^{\ast}.\]
 We say that two saddle connections $\delta_{1},\delta_{2}$ are \textit{disjoint}\index{disjoint saddle connections} if $\delta_{1}(s_{1})\neq\delta_{2}(s_{2})$   for any $s_{1},s_{2}\in(0,1)$. 
  For $\delta\in\mathscr{L}_{q}$ , define the length function
  \[l_{q,\delta}(t)\coloneqq\max\{|x(\delta,u_{t}q)|,|y(\delta,u_{t}q)|\}.\]
  When $q$ is clear  from the context, we abbreviate $l_{q,\delta}(t)$ by $l_{\delta}(t)$. 
\begin{lem}[{\cite[Lemma 4.4]{minsky2002nondivergence}}]\label{nondivergence2024.07.12} For each $q\in\mathcal{TH}_{1}(\alpha)$, and each $\delta\in \mathscr{L}_{q}$, either $l_{q,\delta}(t)$ is a constant function of $t$, or there are $t_{0}$ and $c>0$ such that 
\[l_{q,\delta}(t)=\max\{c,c|t-t_{0}|\}.\]
\end{lem}

Let $\mathscr{F}$ be a collection of continuous functions $\mathbb{R}\rightarrow\mathbb{R}^{+}$. For $\theta>0$, $f\in\mathscr{F}$ and $I\subset\mathbb{R}$ an interval, we let 
 \begin{align}
I_{f}(\theta) &\coloneqq \{s\in I:f(s)<\theta\},\;\nonumber\\
I_{\mathscr{F}}(\theta)   & \coloneqq \{s\in I: \exists f\in\mathscr{F},\ f(s)<\theta\}=\bigcup_{f\in\mathscr{F}}I_{f}(\theta) ,\;\nonumber\\
\|f\|_{I}  & \coloneqq \sup_{s\in I}f(s). \nonumber
\end{align} 
\begin{defn}[{$(C,\alpha)$-good function}]
  Let $C,\alpha,\rho$ be positive constants.
  \begin{itemize}
    \item  We say that $\mathscr{F}$ is \textit{$(C,\alpha,\rho)$-good}\index{$(C,\alpha,\rho)$-good} if for every interval $I\subset \mathbb{R}$ and every $f\in\mathscr{F}$, we have 
        \begin{equation}\label{nondivergence2024.07.01}
          \frac{|I_{f}(\epsilon) |}{|I|}\leq C\left(\frac{\epsilon}{\|f\|_{I}}\right)^{\alpha}
        \end{equation} 
        for $0<\epsilon<\rho$.
    \item We   say that $\mathscr{F}$ is \textit{$(C,\alpha)$-good}\index{$(C,\alpha)$-good} if it is $(C,\alpha,\rho)$-good for every $\rho$.
  \end{itemize} 
\end{defn}
  
\begin{lem}[{\cite[Lemma 4.5]{minsky2002nondivergence}}]\label{effective2024.07.18}
 The collection
  \[\mathscr{F}_{1}^{(\alpha)}\coloneqq\{l_{q,\delta}:q\in\mathcal{H}(\alpha),\ \delta\in\mathscr{L}_{q}\}\]
  is $(2,1)$-good.
\end{lem}

In particular, recall that $\mathcal{H}_{1}(0)=G/\Gamma$. Then $\mathscr{F}_{1}^{(0)}$ for $G/\Gamma$ is also $(2,1)$-good. Moreover, note that  any $\Lambda_{1}\in G/\Gamma$ with $\ell(\Lambda_{1})<1$ has a unique vector $v\in \Lambda_{1}\subset\mathbb{R}^{2}$ such that 
\[l_{\Lambda_{1},v}(0)=\ell(\Lambda_{1}).\]  
It follows that for any $\eta\in(0,1)$, the collection
\[\mathscr{F}_{2}(\Lambda_{1})\coloneqq\{l_{\Lambda_{1},v}\in\mathscr{F}_{1}^{(0)}:v\in \mathbb{R}^{2}\text{ is a vector on }\Lambda_{1}\}\]
satisfies
\[ \#\left\{f\in \mathscr{F}_{2}(\Lambda_{1}) :f(t)\leq  \eta\right\}\leq 1.\]
On the other hand, one can easily deduce
\[\{s\in I:\ell(u_{s}\Lambda_{1})<\eta\}\subset I_{\mathscr{F}_{1}^{(0)}}(\eta).\] 
Now recall that for $(\Lambda_{1},\Lambda_{2})\in X$, 
  \[\ell(\Lambda_{1},\Lambda_{2})= \min\{\ell(\Lambda_{1}),\ell(\Lambda_{2})\}.\]
Thus, we obtain:
\begin{prop}[Sparse cover for $X$]\label{effective2024.07.37} 
For any $(\Lambda_{1},\Lambda_{2})\in X$, any interval $I\subset\mathbb{R}$, there exists a collection $\mathscr{F}^{\prime\prime}$ of $(2,1)$-good functions 
\[\mathscr{F}^{\prime\prime}(\Lambda_{1},\Lambda_{2})\coloneqq \mathscr{F}_{2}(\Lambda_{1})\cup \mathscr{F}_{2}(\Lambda_{2})\]  such that  for any $\eta\in(0,1)$, for any $t\in I$, we have
\begin{equation}\label{effective2024.07.35}
  \#\left\{f\in\mathscr{F}^{\prime\prime}(\Lambda_{1},\Lambda_{2}) :f(t)<  1\right\}\leq 2,
\end{equation} 
  and
  \begin{equation}\label{effective2024.07.36}
 V_{0}(\eta)\coloneqq\{r\in I:\ell(u_{r}(\Lambda_{1},\Lambda_{2}))<\eta\}\subset I_{\mathscr{F}^{\prime\prime}(\Lambda_{1},\Lambda_{2})}(\eta).
\end{equation}   
\end{prop} 

In the following, we shall develop a similar sparse cover for the Teichm\"{u}ller spaces.

\begin{prop}[{\cite[Proposition 4.7]{minsky2002nondivergence}}]\label{nondivergence2024.07.04} There exists $R$ (depending only on $\alpha$) such that for all $q\in\mathcal{TH}_{1}(\alpha)$, if $\delta_{1},\ldots,\delta_{r}\in\mathscr{L}_{q}$ are disjoint, then $r\leq R$.
\end{prop}

  Let $R$ be as in Proposition \ref{nondivergence2024.07.04}, let $1\leq r\leq R$, and let $q\in\mathcal{TH}_{1}(\alpha)$.
   Define 
  \[\mathscr{E}_{r}\coloneqq\{E\subset\mathscr{L}_{q}:E\text{ consists of }r\text{ disjoint segments}\}.\]
     Define 
  \[l_{q,E}(t)\coloneqq\max_{\delta\in E}l_{q,\delta}(t),\ \ \ \alpha_{r}(t)= \alpha_{q,r}(t)\coloneqq\min_{E\in\mathscr{E}_{r}}l_{q,E}(t).\]
   For $E\in\mathscr{E}_{r}$, define $S(E)$ as the closure of the union of the simply connected components of $S\setminus\bigcup_{\delta\in E}\delta$.

\begin{prop}[{\cite[Lemma 6.1]{minsky2002nondivergence}}]\label{nondivergence2024.07.05} There exists $\eta_{0}=\eta_{0}(\alpha)>0$  such that for every $q\in\mathcal{TH}_{1}(\alpha)$, every $r\in\{1,\ldots,R\}$ and every $E\in\mathscr{E}_{r}$, if $S(E)=S$, then $l_{q,E}(0)\geq \eta_{0}$. In particular, $\alpha_{q,R}(0)\geq\eta_{0}$.
\end{prop} 
\begin{lem}[{\cite[Lemma 6.2]{minsky2002nondivergence}}]\label{nondivergence2024.07.03} Let $q\in\mathcal{TH}_{1}(\alpha)$, $1\leq r\leq R-1$. Suppose that $E\in\mathscr{E}_{r}$ such that $l_{q,E}(0)<\frac{\theta}{3\sqrt{2}}$ and $\alpha_{r+1}(0)\geq\theta$ for some $\theta$, and suppose that $\delta$ is a saddle connection on $\partial S(E)$. Then for any $\delta^{\prime}\in\mathscr{L}_{q}$ such that $\delta^{\prime}\neq\delta$ and $\delta^{\prime}\cap\delta\neq\emptyset$, we have $l_{q,\delta^{\prime}}(0)\geq\frac{\sqrt{2}}{3}\theta$. 
\end{lem}

\begin{lem}[{\cite[Lemma 6.4]{minsky2002nondivergence}}]\label{nondivergence2024.07.13}
  Let $f$ and $\tilde{f}$ be two functions of the form $t\mapsto\max\{c,c|t-t_{0}|\}$. Suppose that for some $b>0$ and $s\in\mathbb{R}$, we have $f(s)<b/3$ and $\tilde{f}(s)<b/3$. Then, possibly after exchanging $f$ and $\tilde{f}$, $f(t)<b$ whenever $\tilde{f}(t)<b/3$. 
\end{lem}

 In \cite{minsky2002nondivergence}, Minsky and Weiss implicitly gives a sparse cover for the set $V(\eta)\coloneqq\{s\in I:\ell(u_{s}q)<\eta\}=\{t\in I:\alpha_{1}(t)<\eta\}$ (cf. Proposition \ref{effective2024.07.37}): 
 \begin{thm}[Sparse cover for $\mathcal{H}_{1}(\alpha)$]\label{effective2024.07.11}  Let $R$ be  \hypertarget{2024.08.C12} as in Proposition \ref{nondivergence2024.07.04}, and $\eta_{0}$   as in Proposition \ref{nondivergence2024.07.05}.
For any $q\in\mathcal{TH}_{1}(\alpha)$, any interval $I\subset\mathbb{R}$, there exists  an absolute  constant  $C_{12}>0$ so that the following statement holds. For any $0<C<\min\{\eta_{0},\hyperlink{2024.08.C12}{C_{12}}\}$,  there are subcollections $\mathscr{F}_{0}(k,q)\subset\mathscr{F}_{1}^{(\alpha)}$ of length functions for $k=1,\ldots,R$ such that  for any   $t\in I$, we have
\begin{equation}\label{effective2024.07.14}
  \#\left\{f\in\mathscr{F}_{0}(k,q) :f(t)\leq \frac{\sqrt{2}}{9} C^{\frac{R-k-1}{R-1}}C\right\}\leq R,
\end{equation} 
  and
  \begin{equation}\label{effective2024.07.15}
 V(C^{2})=\{s\in I:\ell(u_{s}q)<C^{2}\}\subset \bigcup_{k=1}^{R}I_{\mathscr{F}_{0}(k,q)}(C^{\frac{R-k}{R-1}}C).
\end{equation}  
 \end{thm}
 Theorem \ref{effective2024.07.11} has been implicitly given in the proof of Theorem \ref{closing2024.3.60} ({\cite[Theorem 6.3]{minsky2002nondivergence}}). We provide a proof here for completion.
 \begin{proof}[Proof of Theorem \ref{effective2024.07.11}]  
 Let $\hyperlink{2024.08.C12}{C_{12}}$ be a sufficiently small constant that will be determined below. Let $C\in(0,\hyperlink{2024.08.C12}{C_{12}})$.  For $k=1,\ldots,R-1$, write 
 \[L_{k}\coloneqq C^{\frac{R-k}{R-1}}C. \]
 Note that $L_{1}=C^{2}$, $L_{R}=C$, and the $L_{k}$'s increase by a constant multiplicative factor:
\begin{equation}\label{nondivergence2024.07.09}
  \frac{L_{k}}{L_{k+1}}= C^{\frac{1}{R-1}}<\hyperlink{2024.08.C12}{C_{12}}^{\frac{1}{R-1}}.
\end{equation} 
For  $t\in V(C^{2})$, let 
\[r(t)\coloneqq\max\{k:\alpha_{k}(t)<L_{k}\}.\]
Since $\alpha_{R}(t)\geq\eta_{0}\geq C=L_{R}$ by Proposition \ref{nondivergence2024.07.05}, we have $r(t)\leq R-1$ and 
\[\alpha_{r(t)}(t)<L_{r(t)},\ \ \ \alpha_{r(t)+1}(t)\geq L_{r(t)+1}.\]
Let $V_{k}\coloneqq\{t\in V(C^{2}):r(t)=k\}$.
Then we have
\begin{equation}\label{effective2024.07.12}
 V(C^{2})=\bigsqcup_{k=1}^{R-1}V_{k}.
\end{equation} 
 Now, for $\delta\in\mathscr{L}_{q}$, let 
   \[H_{k}(\delta)\coloneqq\left\{t\in I:l_{q,\delta}(t)<L_{k}\text{ and }l_{q,\delta}(t)\geq \frac{\sqrt{2}}{3} L_{k+1}\text{ for }\delta\neq\delta^{\prime}\in \mathscr{L}_{q},\ \delta\cap\delta^{\prime}\neq\emptyset\right\}.\] 
 Finally, define 
\[\mathscr{F}_{0}(k,q)\coloneqq\{l_{q,\delta}\in \mathscr{F}_{1}^{(\alpha)}:V_{k}\cap H_{k}(\delta)\neq\emptyset\}.\]

\begin{cla}\label{nondivergence2024.07.10}
   We have 
   \[V_{k}\subset\bigcup_{\delta\in\mathscr{L}_{q}}H_{k}(\delta).\]
\end{cla}
\begin{proof}[Proof of Claim \ref{nondivergence2024.07.10}]
  Let $t\in V_{k}$. Making $\hyperlink{2024.08.C12}{C_{12}}$ small enough, in (\ref{nondivergence2024.07.09}), we get that $L_{k}<\frac{1}{3\sqrt{2}}L_{k+1}$. There is an $E\in \mathscr{E}_{k}$ such that 
  \[l_{q,E}(t)=\alpha_{r}(t)<L_{k}<\frac{1}{3\sqrt{2}}L_{k+1}.\]
  By Proposition \ref{nondivergence2024.07.05}, $S(E)\neq S$, so let $\delta\in E$ be on the boundary of $S(E)$. Then for any $\delta^{\prime}\in \mathscr{L}_{q}$, if $\delta^{\prime}\neq\delta$ and $\delta^{\prime}\cap\delta\neq\emptyset$, then by Lemma \ref{nondivergence2024.07.03}, we have 
  \[l_{q,\delta^{\prime}}(t)\geq\frac{\sqrt{2}}{3}L_{k+1}.\]
  Then $t\in H_{k}(\delta)$.
\end{proof}
\begin{cla}\label{nondivergence2024.07.11}
 For all $t\in I$, we have  
   \[\#\left\{l_{\delta}\in\mathscr{F}_{0}(k,q) :l_{\delta}(t)\leq\frac{\sqrt{2}}{9}L_{k+1}\right\}\leq R.\]
\end{cla}
\begin{proof}[Proof of Claim \ref{nondivergence2024.07.11}]  Making $\hyperlink{2024.08.C12}{C_{12}}$ small enough, in (\ref{nondivergence2024.07.09}), we get that $L_{k}<\frac{\sqrt{2}}{9}L_{k+1}$.   Suppose $t\in I$ and $l_{\delta},l_{\delta^{\prime}}\in\mathscr{F}_{0}(k,q) $. Then  $l_{\delta}(t),l_{\delta^{\prime}}(t)\leq\frac{\sqrt{2}}{9}T_{k+1}$. By Lemmas \ref{nondivergence2024.07.12} and \ref{nondivergence2024.07.13}, we obtain (possibly after exchanging $\delta$ and $\delta^{\prime}$) that for any $s\in I$, if $l_{\delta}(s)<\frac{\sqrt{2}}{9}L_{k+1}$, then  $l_{\delta^{\prime}}(s)<\frac{\sqrt{2}}{3}L_{k+1}$.

Now for $s\in V_{k}\cap H_{k}(\delta)$,  we must have $l_{\delta}(s),<L_{k}\leq \frac{\sqrt{2}}{9}L_{k+1}$. It follows that $\delta^{\prime}$ is disjoint from $\delta$. Since, by Proposition \ref{nondivergence2024.07.04}, the number of disjoint elements of $\mathscr{L}_{q}$ is at most $R$, the claim follows.
\end{proof}
Thus, (\ref{effective2024.07.14}) follows from Claim \ref{nondivergence2024.07.11}. Moreover, by Claim \ref{nondivergence2024.07.10},  we obtain that 
\begin{align}
V_{k}=  & \bigcup_{\delta\in \mathscr{L}_{q}}V_{k}\cap H_{k}(\delta)\;\nonumber\\ 
=  & \bigcup_{l_{\delta}\in \mathscr{F}_{0}(k,q)}V_{k}\cap H_{k}(\delta)\;\nonumber\\ 
  \subset & \bigcup_{l_{\delta}\in \mathscr{F}_{0}(k,q)}H_{k}(\delta)\;\nonumber\\  
\subset  & \bigcup_{l_{\delta}\in \mathscr{F}_{0}(k,q)}I_{l_{\delta},L_{k}} =I_{\mathscr{F}_{0}(k,q),L_{k}}  . \label{effective2024.07.13}
\end{align}
By (\ref{effective2024.07.12}) and (\ref{effective2024.07.13}), we establish Theorem \ref{effective2024.07.11}.
 \end{proof}

\subsection{Existence of large intervals} We shall use the following elementary lemma later to create large intervals as needed. 
For an interval $I=[a-b,a+b]$, $r>0$, we write $rI\coloneqq[a-rb,a+rb]$. 
 \begin{lem}\label{effective2024.07.27} Let $A,B,D,E\geq 10$ be numbers so that $A\leq B$ and $D\geq 6BE/A$. Let $J\subset\mathbb{R}$ be an interval. 
Let  $I,I^{\prime}$ be intervals  so that the midpoints of $I$ and $I^{\prime}$   coincide.
 \begin{enumerate}[\ \ \ (1).]
   \item  Suppose  that we have
 \begin{equation}\label{effective2024.07.29}
   |J|\leq B,\ \ \ A\leq|I|\leq B,\ \ \ J\cap I\neq \emptyset,
 \end{equation} 
 and 
 \begin{equation}\label{effective2024.07.28}
  2D|I|\leq |I^{\prime}|.
 \end{equation}
 Then there exists an interval $J^{\prime}$ such that 
 \begin{equation}\label{effective2024.07.46}
2J^{\prime}\supset 2J,\ \ \ |J^{\prime}|\geq \frac{E}{2}(|I|+|J|),\ \ \ \text{ and }\ \ \ J^{\prime}\subset I^{\prime}\setminus  I.
 \end{equation} 
 \item  Suppose that we have (\ref{effective2024.07.28}), and 
  \begin{equation}\label{effective2024.07.47}
   |J|\leq B,\ \ \ A\leq|I|,\ \ \ J\cap I\neq \emptyset,
 \end{equation} 
 and  there is an endpoint $a\in J$ such that $a\not\in I$.   Then there exists an interval $J^{\prime}$ such that 
 \begin{equation}\label{effective2024.07.48}
2J^{\prime}\supset 2J,\ \ \ |J^{\prime}|= E|J|,\ \ \ \text{ and }\ \ \ J^{\prime}\subset I^{\prime}\setminus I.
 \end{equation}  
 \end{enumerate}  
 \end{lem}
 \begin{proof}  This follows   from the idea  that $I$ always sits in the middle of  $I^{\prime}$.
    
  (1).  Say $J=[a,a+|J|]$. Since $J\cap I\neq\emptyset$, we have $I\subset [a-B,a+|J|+B]$. Since by (\ref{effective2024.07.29})(\ref{effective2024.07.28}), we see that 
    \[2DA\leq 2D|I|\leq  |I^{\prime}|.\]
   Since the midpoints of $I$ and $I^{\prime}$ coincide, let $x\in I\subset I^{\prime}$ be the midpoint of $I$ and $I^{\prime}$.    Then $x\in [a-B,a+|J|+B]$, and
    \begin{equation}\label{effective2024.07.22}
      [x,x+DA]\subset I^{\prime}.
    \end{equation} 
    
    Now define 
    \[J^{\prime}\coloneqq[a+|J|+2B,a+B+\frac{1}{2}DA].\]
    First, one calculates the length  
    \begin{align}
|J^{\prime}| &      =(B+\frac{1}{2}DA)-(|J|+2B) \;\nonumber\\ 
   & \geq B+3BE-3B\geq BE \geq \frac{E}{2} (|I|+|J|). \label{effective2024.07.23}
\end{align}
Then the left endpoint of $2J^{\prime}$ is 
\[(a+|J|+2B)-\frac{1}{2}|J^{\prime}|\leq a+3B-\frac{1}{2}BE<a-\frac{1}{2}|J|.\]
Since $J^{\prime}$ sits on the right of $J$, we conclude that $2J^{\prime}\supset 2J$.

   Next, note that
   \begin{equation}\label{effective2024.07.25}
     I\cap J^{\prime}\subset [a-B,a+|J|+B]\cap J^{\prime}=\emptyset.
   \end{equation} 
  To see $ J^{\prime}\subset I^{\prime}$,   we compare the right endpoints
    \[a+B+\frac{1}{2}DA\leq a-B+DA\leq x+DA.\]
    Thus, by (\ref{effective2024.07.22}), we have
    \begin{equation}\label{effective2024.07.24}
     J^{\prime}\subset  [x,x+DA]\subset I^{\prime}.
    \end{equation} 
    (\ref{effective2024.07.23}), (\ref{effective2024.07.25}) and (\ref{effective2024.07.24}) establish (\ref{effective2024.07.46}). Note that one may similarly construct an interval $J^{\prime}$ that sits on the left of $J$.
    
    (2).  Say $J=[a,a+|J|]$ and suppose that $a+|J|$ is the endpoint so that  $a+|J|\not\in I$.
    By (\ref{effective2024.07.47})(\ref{effective2024.07.28}), we see that 
    \[2DA\leq 2D|I|\leq  |I^{\prime}|.\]
   Let $x\in I\subset I^{\prime}$ be the midpoint of $I$ and $I^{\prime}$.    Then  
    \begin{equation}\label{effective2024.07.50}
      [x,x+DA]\subset [x,x+D|I|]\subset |I^{\prime}|.
    \end{equation} 
     Since $J\cap I\neq\emptyset$, we have 
    \begin{equation}\label{effective2024.07.51}
     I\subset [0,a+|J|),\ \ \ \text{ and }\ \ \ a\leq x+\frac{1}{2}|I|.
    \end{equation}

    Now define 
    \[J^{\prime}\coloneqq[a+|J|,a+|J|+E|J|].\]
    First, one calculates the length   
\begin{equation}\label{effective2024.07.49}
  |J^{\prime}|     = E |J|.
\end{equation}
Then the left endpoint of $2J^{\prime}$ is 
\[(a+|J|)-\frac{1}{2}|J^{\prime}|\leq a+|J|-\frac{1}{2} E |J|<a-\frac{1}{2}|J|.\]
Since $J^{\prime}$ sits on the right of $J$, we conclude that $2J^{\prime}\supset 2J$.

   Next, note that
   \begin{equation}\label{effective2024.07.52}
     I\cap J^{\prime}\subset (-\infty,a+|J|)\cap J^{\prime}=\emptyset.
   \end{equation} 
  To see $ J^{\prime}\subset I^{\prime}$,  by (\ref{effective2024.07.51}), one observes 
    \[a+(|J|+E|J|)\leq (x+\frac{1}{2}|I|)+(B+\frac{1}{6}DA)\leq x+D|I|.\]
    Thus, by (\ref{effective2024.07.50}), we have
    \begin{equation}\label{effective2024.07.53}
     J^{\prime}\subset  [x, x+D|I|]\subset I^{\prime}.
    \end{equation} 
    (\ref{effective2024.07.49}), (\ref{effective2024.07.52}) and (\ref{effective2024.07.53}) establish (\ref{effective2024.07.48}). Note that one may similarly construct an interval $J^{\prime}$ that sits on the left of $J$ if $a$ is the endpoint  so that  $a\not\in I$.
     \end{proof}

\subsection{New approximation}   
From (\ref{effective2024.07.105}), we have seen that we fail to proceed if $x_{0}$ has a small injectivity radius. In the following, we assume that
\begin{equation}\label{effective2024.08.14}
 \ell(x_{0})<\varkappa_{\vartheta}(e^{t})^{\xi} 
\end{equation}   
 and use sparse cover to find another qualified surface with  weaker estimates.
 \begin{prop}\label{effective2024.07.127} Assume that (\ref{effective2024.08.14}) holds.
    Then there exists  and a  function $\varkappa_{\vartheta}^{\star}:\mathbb{R}^{+}\rightarrow\mathbb{R}^{+}$  depending on $\varkappa_{\vartheta}$, $\hyperlink{2024.08.k7}{\kappa_{7}}$, $\vartheta$ with $\lim_{t\rightarrow\infty}\varkappa^{\star}_{\vartheta}(e^{t})=0$, and  an interval  
 $J^{\star\star}\subset[0,e^{t}]$ with $|J^{\star\star}|\geq(\varkappa_{\vartheta}^{\star}(e^{t}))^{-1}$, such that for
 \[J_{d}^{\prime}=\left\{ r\in J^{\star\star}: \|\tilde{\Lambda}-u_{r}a_{t}\Lambda(x)\|\leq  \varrho,\ \tilde{\Lambda}\in (Q_{D})_{\eta^{R^{2}+1}}\right\},\]  
 we have 
   \begin{equation}\label{effective2024.10.14}
     \left|J_{d}^{\prime}\right|\gg   \varrho^{3} |J^{\star\star}| ,
     \end{equation}
and  at least one of the following holds:
    \begin{enumerate}[\ \ \ (i)]
    \item  For any $r\in J_{d}^{\prime}$, and $\tilde{\Lambda}_{d}\in  (Q_{D})_{\eta^{R^{2}}}$, we have
    \[  \ell(u_{r}a_{t}\tilde{x})\geq  \varkappa_{\vartheta}(|J^{\star\star}|)^{\hyperlink{2024.08.k7}{\kappa_{7}}+\vartheta}, \ \ \   \ell(u_{r}a_{t}\tilde{x})^{\vartheta} \|\tilde{\Lambda}_{d}-u_{r}a_{t}\tilde{x}\|_{u_{r}a_{t}\tilde{x}}\geq \varkappa_{\vartheta}(|J^{\star\star}|).  \]  
      \item   There is a surface  a time $r^{\star}\in[0,1]$,     surface  $x^{\star}=a_{t}u_{r^{\star}}x\in\mathcal{H}_{1}(2)$, and a surface $y_{d}^{\star}\in (\Omega_{1}W_{D})_{\eta^{R^{2}}}$, such that
            \[ \|y_{d}^{\star}-x^{\star}\|_{x^{\star}}<\varkappa_{\vartheta}(|J^{\star\star}|) \leq \ell(x^{\star})^{\hyperlink{2024.08.k7}{\kappa_{7}}}.\] 
    \end{enumerate} 
 \end{prop}
 
 Let $y\coloneqq a_{s_{1}}x$. Then we see that   there exists  a $e^{t}r_{0}\in[0,e^{t}]$ such that 
 \begin{equation}\label{effective2024.07.31}
 \ell(u_{e^{t}r_{0}}y)<\varkappa_{\vartheta}(e^{t})^{\xi}.
 \end{equation}

 Now we apply the sparse cover. Let $C>0$ be a small constant that will be determined later. Applying Theorem \ref{effective2024.07.11}, set $q,I,C$ in the statement of the theorem equal respectively to $y,[0,e^{t}],C$. Then  we obtain  a collection $\mathscr{F}_{0}(k)$ of $(2,1)$-functions for $k=1,\ldots,R$ such that  for   any $t\in [0,e^{t}]$,  we have
\begin{equation}\label{effective2024.07.16}
 \#\left\{f\in\mathscr{F}_{0}(k) :f(t)\leq \frac{\sqrt{2}}{9} C^{\frac{R-k-1}{R-1}}C\right\}\leq R,
\end{equation} 
  and
  \begin{equation}\label{effective2024.07.17}
 V(C^{2})=\{r\in [0,e^{t}]:\ell(u_{r}y)<C^{2}\}\subset \bigcup_{k=1}^{R}I_{\mathscr{F}_{0}(k)}(C^{\frac{R-k}{R-1}}C).
\end{equation}   
Finally, we define
\[F\coloneqq    \bigcup_{k=1}^{R} \mathscr{F}_{0}(k).\]

   Define $\phi:F\rightarrow\mathbb{R}^{+}$    by 
     \[ \phi:f\mapsto C^{\frac{R-k}{R-1}}C,   \ \ \ \text{ if }    f\in \mathscr{F}_{0}(k),\]
and define $\psi:F\rightarrow\mathbb{R}^{+}$    by 
     \[ \psi:f\mapsto  \frac{\sqrt{2}}{9} C^{\frac{R-k-1}{R-1}}C,\ \ \ \text{ if }    f\in \mathscr{F}_{0}(k)  .\]
Then for any $f\in F$,   one calculates
\begin{equation}\label{effective2024.07.59}
  \frac{\phi(f)}{\psi(f)}\leq  \frac{9\sqrt{2}}{2}C^{\frac{1}{R-1}}.
\end{equation} 
Also, it follows from   the sparse cover (\ref{effective2024.07.16})(\ref{effective2024.07.17}) that  for  any $t\in I$,  we have
\begin{equation}\label{effective2024.07.55}
 \#\left\{f\in F :f(t)\leq \psi(f)\right\}\leq R^{2},
\end{equation} 
and 
\begin{equation}\label{effective2024.07.56}
V(C^{2})\subset  \bigcup_{f\in F}I_{f}(\phi(f)).
\end{equation}  
For  $\sigma>0$, define   
   \begin{align}
L_{\sigma}(r^{\ast})\coloneqq &\min \{r<r^{\ast}:\max_{r^{\prime}\in[r,r^{\ast}]}\ell(u_{r^{\prime}}y)<\sigma\},\;\nonumber\\
R_{\sigma}(r^{\ast})\coloneqq &\max \{r>r^{\ast}:\max_{r^{\prime}\in[r^{\ast},r]}\ell(u_{r^{\prime}}y)<\sigma\},\;\nonumber\\
J_{\sigma}(r^{\ast})\coloneqq& [L_{\sigma}(r^{\ast}),R_{\sigma}(r^{\ast})]. \nonumber
\end{align} 
Then by (\ref{effective2024.07.31}) and sufficiently large $T$, we have  
\[ |J_{\varkappa_{\vartheta}(e^{t})^{-\frac{1}{2}\xi}}(e^{t}r_{0})|\geq1.\] 
Since any length function is $(2,1)$-good (Lemma \ref{effective2024.07.18}),  we have 
\[\frac{1}{2}C^{2} \varkappa_{\vartheta}(e^{t})^{-\frac{1}{2}\xi}\leq \frac{1}{2}C^{2} \varkappa_{\vartheta}(e^{t})^{-\frac{1}{2}\xi}|J_{\varkappa_{\vartheta}(e^{t})^{\frac{1}{2}\xi}}(e^{t}r_{0})|\leq|J_{C^{2}}(e^{t}r_{0})|.\]
Write $J=J_{C^{2}}(e^{t}r_{0})$. For simplicity, say $J\subset \frac{1}{3}[0,e^{t}]$ (one may adjust the coefficients and follow the same argument as below for other situations).
 
Then,  $V(C^{2})$ contains the interval 
$J$ of length
\begin{equation}\label{effective2024.07.69}
|J|\geq \frac{1}{2}C^{2} \varkappa_{\vartheta}(e^{t})^{-\frac{1}{2}\xi}\geq  \varkappa_{\vartheta}(e^{t})^{-\xi^{2}}.
\end{equation}  
  
  Let $D_{1}>1$ be sufficiently large that will be specified later.
  For  $f\in F$, let $x\in  I_{f}(\phi(f))$ be the midpoint of $I_{f}(\phi(f))$ and $I_{f}(\psi(f))$. Let $I= I_{f}(\phi(f))$, and 
  \begin{equation}\label{effective2024.07.64}
    I^{\prime}=[x- D_{1}|I|,x+ D_{1}|I|].
  \end{equation} 
 Then since $f$ is $(2,1)$-good, by  (\ref{effective2024.07.59}), we have 
 \begin{equation}\label{effective2024.07.66}
   |I^{\prime}|= 2D_{1} |I|\leq 4D_{1}\cdot\frac{I_{f}(\phi(f))}{I_{f}(\psi(f))}|I_{f}(\psi(f))|\leq 18\sqrt{2}C^{\frac{1}{R-1}}D_{1}|I_{f}(\psi(f))|.
 \end{equation} 
 Let $C$ be sufficiently small so that $18\sqrt{2}C^{\frac{1}{R-1}}D_{1}<1$. Then $I^{\prime}\subset I_{f}(\psi(f))$.

 \begin{lem}\label{effective2024.07.33} Let $J\subset \frac{1}{3}[0,e^{t}]$ be an interval in $V(C^{2})$  with $|J|\geq 10R^{2}$. Suppose that for   any $f\in F$, we have 
 \begin{equation}\label{effective2024.07.67}
 |I_{f}(\phi(f))|<\frac{1}{6}(2D_{1})^{-1}e^{t}.
 \end{equation} 
  Then there \hypertarget{2024.07.C211} exists    an interval $J_{1}\subset[0,e^{t}]$, and  a function $f\in F$ such that  
 \begin{align}
2J_{1} &  \supset 2J,\;\label{effective2024.07.54}\\
|J_{1}| &\geq  \frac{D_{1}}{24R^{2}}|I_{f}( \phi(f))|+ \frac{D_{1}}{24R^{2}}|J|,\;\label{effective2024.07.61}\\
J_{1} & \subset I_{f}(\psi(f))\setminus I_{f}(\phi(f)). \label{effective2024.07.62}
\end{align} 
 \end{lem}
 \begin{proof}   
By the sparse cover (\ref{effective2024.07.56}),   we have
\[J\subset\bigcup_{f\in F}I_{f}(\phi(f)).\]
\begin{cla}\label{effective2024.07.57}
There exists a function $f\in F$ such that  
\begin{equation}\label{effective2024.07.21}
  I_{f}(\phi(f))\cap J\neq\emptyset, 
\end{equation} 
\begin{equation}\label{effective2024.07.26}
  |I_{f}(\phi(f))|\geq R^{-2}|J|\geq 10.
\end{equation} 
\end{cla}
\begin{proof}[Proof of Claim \ref{effective2024.07.57}] 
  Assume for contradiction that for any $f\in F$,
\begin{equation}\label{effective2024.07.58}
  |I_{f}(\phi(f))|< R^{-2}|J|.
\end{equation} 
Since $100 R^{2} \phi(f) \leq \psi(f)$, (\ref{effective2024.07.58})   indicates that  
\[\frac{|I_{f}(\phi(f))\cap J|}{|I_{f}(\psi(f))\cap J|}< R^{-2}.\]
Note by (\ref{effective2024.07.55}) that 
\[\sum_{f\in F}|I_{f}(\psi(f))\cap J|=\int_{J}\#\left\{f\in F :f(r)\leq \psi(f)\right\}dr\leq R^{2}|J|.\]
It follows that 
\[R^{2}|J| = R^{2}\sum_{f\in F}|I_{f}(\phi(f))\cap J| 
  <  \sum_{f\in F}|I_{f}(\psi(f))\cap J|\leq R^{2}|J|. \] 
This leads to a contradiction.
\end{proof} 
We choose $f$ satisfying (\ref{effective2024.07.21})(\ref{effective2024.07.26}) so that the length $|I_{f}(\phi(f))|$ is the maximal possible.
  
 Let $x^{\star}$ be the midpoint of $J$, and
  \[J^{\star}\coloneqq[x^{\star}-R^{2}|I_{f}(\phi(f))|,x^{\star}+R^{2}|I_{f}(\phi(f))|].\]

 Now we apply Lemma \ref{effective2024.07.27}(1), setting $A,B,D,E,I,I^{\prime},J$ in the statement of the lemma equal to $|I_{f}(\phi(f))|,2R^{2}|I_{f}(\phi(f))|, D_{1},\frac{D_{1}}{12R^{2}},I,I^{\prime},J^{\star}$, respectively. 
 To check the requirements, note that  
 \begin{equation}\label{effective2024.07.08}
D=D_{1}=6\cdot 2R^{2}\cdot\frac{D_{1}}{12R^{2}}= 6\cdot  \frac{B}{A} \cdot E.
 \end{equation}  
 On the other hand, recall from (\ref{effective2024.07.64}) that 
\begin{equation}\label{effective2024.07.20}
 2D_{1} |I|=|I^{\prime}|.
\end{equation} 
This establishes (\ref{effective2024.07.28}). Also,   by (\ref{effective2024.07.21})(\ref{effective2024.07.26}), the interval  
$J^{\star}$
  satisfies $J^{\star}\cap I_{f}(\phi(f))\neq\emptyset$. Thus, we establishes  (\ref{effective2024.07.29}).

Therefore, by (\ref{effective2024.07.46}),  there exists $J^{\prime}\subset \mathbb{R}$ such that 
  \begin{align}
2J^{\prime} &\supset 2J,\;\nonumber\\
|J^{\prime}| &\geq  \frac{D_{1}}{24R^{2}}|I|+ \frac{D_{1}}{24R^{2}}|J|,\;\nonumber\\
J^{\prime} &\subset I^{\prime}\setminus I. \nonumber
\end{align} 
 Since $|J^{\prime}|\leq|I^{\prime}|=2D_{1} |I|\leq \frac{1}{6}e^{t}$, $2J^{\prime} \supset J$, and $J \subset[\frac{1}{3}e^{t},\frac{2}{3}e^{t}]$, we conclude that $J^{\prime}\subset
 [0,e^{t}]$. The lemma follows from letting $J_{1}=J^{\prime}$. 
 \end{proof}
  
 Let $D_{2}=\frac{D_{1}}{1000R^{2}e}$. In particular,  we have
  \begin{equation}\label{effective2024.07.65}
  \frac{1}{6}(2D_{2})^{-1} \frac{D_{1}}{24R^{2}}>1,\ \ \ \text{ and }\ \ \ \frac{D_{2}e}{D_{1}}<\frac{1}{6}.
  \end{equation} 
 \begin{lem}\label{effective2024.07.45} Let the notation and assumptions be as in Lemma \ref{effective2024.07.33}.
    Suppose that  there exists $g\in F\setminus\{f\}$ such that 
    \begin{equation}\label{effective2024.07.60}
      I_{g}(\phi(g))\cap J_{1}\neq\emptyset,\ \ \ \text{ and }\ \ \ |I_{g}(\phi(g))|\geq  \frac{1}{6}(2D_{2})^{-1}|J_{1}|.
    \end{equation} 
    Then there exists an interval $J_{2}\subset[0, e^{t}]$ such that 
    \[|J_{2}|=      e|J_{1}|, \ \ \ 2J_{2}\supset 2J_{1}\]
    and that
for any  
 \[J_{2}\subset  I_{g}(\psi(g))\setminus I_{g}(\phi(g)).\] 
 \end{lem}
 \begin{proof}

 Now since by (\ref{effective2024.07.60}), we have
  \[  I_{g}( \phi(g))\cap J_{1}\neq\emptyset.\]
 Also, by (\ref{effective2024.07.61}) and (\ref{effective2024.07.65}), we have 
 \[|I_{g}(\phi(g))|\geq \frac{1}{6}(2D_{2})^{-1}|J_{1}| \geq \frac{1}{6}(2D_{2})^{-1} \frac{D_{1}}{24R^{2}}|I_{f}( \phi(f))|>|I_{f}( \phi(f))|.\]
 If $I_{g}(\phi(g))\cap J\neq\emptyset$, then it violates the maximal choice of $f$ (cf. the definition after Claim \ref{effective2024.07.57}). Thus, we conclude that $I_{g}(\phi(g))\cap J=\emptyset$. This, together with $2J_{1}\supset 2J$ (by (\ref{effective2024.07.54})) and $J_{1}   \subset I_{g}(\psi(g))\setminus I_{g}(\phi(g))$ (by (\ref{effective2024.07.62})), imply that there is an endpoint $a$ of $2J_{1}$ such that $a\not\in I_{g}(\phi(g))$. 
 
Next,  we apply Lemma \ref{effective2024.07.27}(2), setting $A,B,D,E,I,I^{\prime},J$ in the statement of the lemma equal to $ \frac{1}{6}(2D_{2})^{-1}|J_{1}|, 2\cdot|J_{1}|, 6\cdot 24D_{2}e,e, I_{g}(\phi(g)),  I_{g}(\psi(g)), 2J_{1}$, respectively. 
 To check the requirements, since $D_{2}<D_{1}$, by (\ref{effective2024.07.66}), we have
 \begin{equation}\label{effective2024.07.40}
 2D_{2} |I_{g}(\phi(g))|= 2D_{2} |I|\leq    |I|=|I_{g}(\psi(g))|
 \end{equation}   
This establishes (\ref{effective2024.07.28}). Also,   (\ref{effective2024.07.47}) follows from (\ref{effective2024.07.60}).

Therefore, by Lemma \ref{effective2024.07.27}(2),  there exists $J_{2}\subset \mathbb{R}$ such that 
\[|J_{2}|  =e|J_{1}|,\ \ \ 2J_{2}\supset 4J_{1}\supset 2J_{1},\ \ \ J_{2}  \subset I_{g}(\psi(g))\setminus I_{g}(\phi(g)). \]   Finally, note that   by (\ref{effective2024.07.67})(\ref{effective2024.07.65}), we have
 \[|J_{2}|  =e|J_{1}|\leq e\cdot 6(2D_{2})|I_{g}(\phi(g))|<\frac{e D_{2}}{D_{1}}<\frac{1}{6}e^{t},\]
  $2J_{2} \supset 2J_{1} \supset J$, and $J \subset[\frac{1}{3}e^{t},\frac{2}{3}e^{t}]$, we conclude that $J_{2}\subset
 [0,e^{t}]$. 
 \end{proof}
 
 Note that  (\ref{effective2024.07.60}) cannot always occur. For instance, it does not hold for $J_{1}=[0,e^{t}]$. Thus, we obtain:
\begin{cor}\label{effective2024.07.68} Let the notation and assumptions be as in Lemma \ref{effective2024.07.33}. There exists 
  an interval $J_{1}^{\star}\subset[0,e^{t}]$, and a function $f^{(1)}\in F$ such that 
      \[2J_{1}^{\star}\supset 2J,\ \ \  J_{1}^{\star}\subset I_{f^{(1)}}(\psi(f^{(1)}))\setminus I_{f^{(1)}}(\phi(f^{(1)})).\]  
    Moreover, any  $g\in F\setminus\{f^{(1)}\}$ with $ I_{g}(\phi(g))\cap J_{1}^{\star}\neq\emptyset$ satisfies
    \begin{equation}\label{effective2024.07.63}
   |I_{g}(\phi(g))|< \frac{1}{6}(2D_{2})^{-1}|J_{1}^{\star}|.
    \end{equation} 
\end{cor}
 
In particular, for $J$ as in (\ref{effective2024.07.69}), we have   
\[|J_{1}^{\star}|\geq  |J|\geq    \varkappa_{\vartheta}(e^{t})^{-\xi^{2}}.\] 
Let $J_{1}^{\star}=[a,a+|J_{1}^{\star}|]$. For $i=0,1,\ldots,\lfloor \varkappa_{\vartheta}(e^{t})^{\xi^{2}}|J_{1}^{\star}|\rfloor$,
 consider  
 \begin{equation}\label{effective2024.10.11}
J^{\prime}_{i}\coloneqq[a+i  \varkappa_{\vartheta}(e^{t})^{-\xi^{2}},a+(i+1) \varkappa_{\vartheta}(e^{t})^{-\xi^{2}}).
 \end{equation} 
\begin{lem}\label{effective2024.07.116} Let $D_{3}=\frac{1}{2R^{2}}D_{2}$. Let $J^{\prime}$ be the union of $J^{\prime}_{i}$ so that 
   any  $g\in F\setminus\{f^{(1)}\}$ with $ I_{g}(\phi(g))\cap J^{\prime}_{i}\neq\emptyset$ satisfies
    \begin{equation}\label{effective2024.07.109}
   |I_{g}(\phi(g))\cap J_{i}^{\prime}|< \frac{1}{6}(2D_{3})^{-1}|J_{i}^{\prime}|.
    \end{equation} 
   Then we have $|J^{\prime}|\geq\frac{1}{2}|J_{1}^{\star}|$. 
\end{lem}
\begin{proof}
   Let $J^{\prime\prime}$ be the union of $J^{\prime}_{i}$ so that there is $g\in F\setminus\{f^{(1)}\}$ so that 
   \[|I_{g}(\phi(g))\cap J_{i}^{\prime}|\geq \frac{1}{6}(2D_{3})^{-1}|J_{i}^{\prime}|.\]
   Now suppose that $|J^{\prime\prime}|\geq\frac{1}{2}|J_{1}^{\star}|$. Then, we have
     \begin{align}
\bigcup_{g\in F\setminus\{f^{(1)}\}}|I_{g}(\phi(g))\cap J^{\prime\prime}| &\geq \frac{1}{6}(2D_{3})^{-1}|J^{\prime\prime}| \;\nonumber\\
  &\geq  \frac{D_{2}}{2R^{2}D_{3}} \cdot\frac{1}{6}R^{2}(2D_{2})^{-1}\cdot 2|J^{\prime\prime}|\geq \frac{1}{6}R^{2}(2D_{2})^{-1}|J_{1}^{\star}|. \nonumber
\end{align} 
  However, by  (\ref{effective2024.07.55}) and (\ref{effective2024.07.63}), we have that 
  \[\bigcup_{g\in F\setminus\{f^{(1)}\}} |I_{g}(\phi(g))\cap J_{1}^{\star}|< \frac{1}{6}R^{2}(2D_{2})^{-1}|J_{1}^{\star}|.\]
  This leads to a contradiction.
\end{proof}
Next, we shall pick a surface other than $x_{0}$ that is close to $\Omega_{1}W_{D}$ in the sense of absolute periods again.
  For simplicity, we assume that $J_{1}^{\star}$ sits on the right of $\Lambda(x_{0})= a_{t}u_{r_{0}}(\Lambda_{1},\Lambda_{2})$.
\begin{prop}\label{effective2024.07.123}
   There exists  an interval $J_{1}^{\star\star}\subset J^{\prime}$  with $|J_{1}^{\star\star}|= \varkappa_{\vartheta}(e^{t})^{-\xi^{2}}$ such that   letting  
\begin{equation}\label{effective2024.10.5}
J_{d}^{\prime(1)}=\left\{ r\in J_{1}^{\star\star}: \|\tilde{\Lambda}-u_{r}\Lambda(x_{0})\|\leq  \varrho,\ \tilde{\Lambda}\in (Q_{D})_{\eta^{2}}\right\},
\end{equation}   
we have  
\[\left|J_{d}^{\prime(1)}\right|\gg  |J_{1}^{\star\star}|^{1-\frac{3\hyperlink{2024.08.k9}{\kappa_{9}}\delta}{2N} }=\varkappa_{\vartheta}(e^{t})^{-\xi^{2}}\varrho^{3}.\]
\end{prop}

Let 
\[a_{1}(t)\coloneqq a_{-\xi^{2}\log \varkappa_{\vartheta}(e^{t})}.\]

First,  $a_{1}(t)^{-1}\tilde{\Lambda}_{d}\in X_{2\eta}$. By (\ref{effective2024.9.5}), we have
\begin{equation}\label{effective2024.07.108}
\|a_{1}(t)^{-1}\tilde{\Lambda}_{d}-a_{1}(t)^{-1}\Lambda(x_{0})\|<\varkappa_{\vartheta}(e^{t})^{\frac{2}{3}\xi }. 
\end{equation}  
In particular, $a_{1}(t)^{-1}\Lambda(x_{0})\in X_{\eta}$.  
Moreover, 
 by  replacing $a_{t}$ with $a_{1}(t)$, we see that  $a_{1}(t)^{-1}\Lambda(x_{0})$ is a point of second kind in Theorem \ref{effective2024.07.107}. 

The following is provided by the polynomial nature of unipotent flows:   
\begin{lem}
  For any periodic $G.\Lambda^{\prime}\subset X$, the distance 
\[P_{G\Lambda^{\prime}}(r)=d_{X}(u_{r}a_{1}(t)^{-1}\Lambda(x_{0}),G\Lambda^{\prime})\]
is a polynomial with bounded degree (when the distance is small). Here $d_{X}$ is the metric on $X$ induced by $\|\cdot\|$. 
\end{lem}

 Assume that $P_{G\Lambda_{c}}(r)$ is constant, i.e. there is $C>0$ so that 
  \[d_{X}(u_{r}a_{1}(t)^{-1}\Lambda(x_{0}),Q_{D})\equiv C < \varkappa_{\vartheta}(e^{t})^{\frac{2}{3}\xi }\]
 for all $r$.
  Then one calculates that $\Lambda(x)= (\Lambda_{1},\Lambda_{2})$ is also $\varkappa_{\vartheta}(e^{t})^{\frac{2}{3}\xi }$-close to $Q_{D}$. Recall that by our assumption, $\Lambda(x)$ is a point of the first kind in Theorem \ref{effective2024.07.107}. Thus, it is impossible   as $t$ being large. 
Thus,  $P_{Q_{D}}(r)$ is a nonconstant polynomial.

  Next, we recall the discreteness of $G$-closed orbits. It is ultimately attributed to the theory of heights and discriminants. 
The theory of heights of Lie groups have been used extensively. In particular, it measures the arithmetic complexity of Lie groups. See e.g. \cite{einsiedler2009distribution,einsiedler2009effective,einsiedler2020effective} for more details. 
\begin{lem}[Discreteness of periodic orbits, {\cite[\S2.4]{einsiedler2009distribution}}]\label{effective2024.07.110} For \hypertarget{2024.08.k13} any $\eta>0$, there exists $\kappa_{13}>0$, and $K_{0}=K_{0}(\eta^{2})>0$ such that 
  for any $\Lambda\in X_{\eta^{2}}$, any $K\geq K_{0}$, any periodic  $\mathcal{O}_{1},\mathcal{O}_{2}\subset X$ of volumes $\vol(\mathcal{O}_{1}), \vol(\mathcal{O}_{2})\leq K$, if $E_{1}\subset B(x,\eta^{2})\cap\mathcal{O}_{1}$, $E_{2}\subset B(x,\eta^{2})\cap\mathcal{O}_{2}$ are two distinct connected components of $G$-orbits, then  
    \[d_{X}(E_{1},E_{2})\geq   K^{-\hyperlink{2024.08.k13}{\kappa_{13}}}.\]
\end{lem}
   Now for Theorem \ref{effective2024.07.107}(2) (by  replacing $a_{t}$ with $a_{1}(t)$), we consider  periodic $G.\Lambda^{\prime}$, $G.\Lambda^{\prime\prime}\subset X$ with $\vol(G\Lambda^{\prime}),\vol(G\Lambda^{\prime\prime})\leq (\varkappa_{\vartheta}(e^{t})^{-\xi^{2}})^{\delta}$. Let $\delta$ be small so that 
   \begin{equation}\label{effective2024.07.120}
    \hyperlink{2024.08.k13}{\kappa_{13}}\delta<\frac{1}{3}.
   \end{equation}  Then by Lemma \ref{effective2024.07.110}, for any $\Lambda\in X_{\eta^{2}}$, if $E_{1}\subset B(x,\eta^{2})\cap G.\Lambda^{\prime}$, $E_{2}\subset B(x,\eta^{2})\cap G.\Lambda^{\prime\prime}$ are two distinct connected  components of $G$-orbits, then  
   \begin{equation}\label{effective2024.07.112}
     d_{X}(E_{1},E_{2})\geq   \varkappa_{\vartheta}(e^{t})^{\hyperlink{2024.08.k13}{\kappa_{13}}\delta\xi^{2}}  \geq \varkappa_{\vartheta}(e^{t})^{\frac{1}{3}\xi^{2}}   .
   \end{equation} 
    
   We say that  $I\subset[0,e^{t} \varkappa_{\vartheta}(e^{t})^{\xi^{2}} ]$ is a \textit{maximal interval}\index{maximal interval} with respect to the periodic $G.\Lambda^{\prime}$ if $I$ is an interval so that for any $r\in I-e^{t}\varkappa_{\vartheta}(e^{t})^{\xi^{2}}r_{0}$,
     \begin{equation}\label{effective2024.07.111}
      d_{X}(u_{r}a_{1}(t)^{-1}\Lambda(x_{0}),G\Lambda^{\prime})\leq  \varkappa_{\vartheta}(e^{t})^{\frac{1}{2}\xi^{2}}  
     \end{equation} 
     and for any interval $J\supsetneq I$, there is $r\in J$ so that (\ref{effective2024.07.111}) does not hold.
   \begin{lem}\label{effective2024.07.115}
      If $I\subset[0, \varkappa_{\vartheta}(e^{t})^{\xi^{2}}e^{t} ]$ is a maximal interval  with respect to a periodic $G.\Lambda^{\prime}\subset X$, then for any $r\in 1000I\setminus I$, $u_{r}a_{1}(t)^{-1}\Lambda(x_{0})$ is a point  of the first kind in Theorem \ref{effective2024.07.107} whenever $u_{r}a_{1}(t)^{-1}\Lambda(x_{0})\in X_{\eta^{2}}$. 
   \end{lem}
   \begin{proof}
   Since $P_{G\Lambda^{\prime}}(r)=d_{X}(u_{r}a_{1}(t)^{-1}\Lambda(x_{0}),G\Lambda^{\prime})$ is a polynomial of bounded degree, there exists $N>0$ (depending only on the degree) such that 
   \begin{equation}\label{effective2024.07.113}
   \varkappa_{\vartheta}(e^{t})^{\frac{1}{2}\xi^{2}}  \leq d_{X}(u_{r}a_{1}(t)^{-1}\Lambda(x_{0}),G\Lambda^{\prime})\leq  N \varkappa_{\vartheta}(e^{t})^{\frac{1}{2}\xi^{2}} 
   \end{equation} 
   for any $r\in 1000I\setminus I$. Fix $r\in 1000I\setminus I$. Suppose that $u_{r}a_{1}(t)^{-1}\Lambda(x_{0})\in X_{\eta^{2}}$, and there exists another periodic $G\Lambda^{\prime\prime}$ so that 
    \begin{equation}\label{effective2024.07.114}
   d_{X}(u_{r}a_{1}(t)^{-1}\Lambda(x_{0}),G\Lambda^{\prime\prime})\leq  \varkappa_{\vartheta}(e^{t})^{\frac{1}{2}\xi^{2}} .
   \end{equation}  
   Let $B=B(u_{r}a_{1}(t)^{-1}\Lambda(x_{0}),\eta^{2})$, $E_{1}=B\cap G\Lambda^{\prime}$, and $E_{2}=B\cap G\Lambda^{\prime\prime}$. Then by (\ref{effective2024.07.113})(\ref{effective2024.07.114}), we have
   \[d_{X}(E_{1},E_{2})\leq (N+1) \varkappa_{\vartheta}(e^{t})^{\frac{1}{2}\xi^{2}}.\]
   This contradicts   (\ref{effective2024.07.112}).
   \end{proof}

   \begin{proof}[{Proof of Proposition \ref{effective2024.07.123}}]  
   Now for any interval $I=[a,b]$, $r>0$, we write $r\times I=[ra,rb]$. Write $ \varkappa_{\vartheta}(e^{t})^{\xi^{2}}\times J^{\star}_{1}=[a^{\star},b^{\star}]$.    Consider the interval 
   \[I^{\star}\coloneqq[0,b^{\star}- \varkappa_{\vartheta}(e^{t})^{\xi^{2}}e^{t}r_{0}]\subset (  \varkappa_{\vartheta}(e^{t})^{\xi^{2}}\times 2J^{\star}_{1})- \varkappa_{\vartheta}(e^{t})^{\xi^{2}}e^{t}r_{0}.\] 
   First, since $a_{1}(t)^{-1}\Lambda(x_{0})\in X_{\eta}$, by the sparse cover (Proposition \ref{effective2024.07.37}), we have
   \begin{equation}\label{effective2024.07.117}
    |\{r\in I^{\star}:u_{r}a_{1}(t)^{-1}\Lambda(x_{0})\not\in X_{\eta^{2}}\}| <\frac{1}{1000}|I^{\star}|.
   \end{equation}  
   On the other hand, by Lemma \ref{effective2024.07.115}, we have
    \begin{equation}\label{effective2024.07.118}
|\{r\in I^{\star}:r\text{ satisfies (\ref{effective2024.07.111}) for periodic }G\Lambda^{\prime}\neq Q_{D}\}| <\frac{1}{1000}|I^{\star}|.
   \end{equation}   
   Now for    $i=0,1,\ldots,\lfloor  \varkappa_{\vartheta}(e^{t})^{\xi^{2}}|J_{1}^{\star}|\rfloor$, let
   \[J^{\prime\prime}_{i}\coloneqq[a^{\star}+i  +\frac{45}{100},a^{\star}+i   +\frac{55}{100}]\subset  \varkappa_{\vartheta}(e^{t})^{-\xi^{2}}\times \frac{1}{3}J^{\prime}_{i}.\]
   Then one directly calculates $|J^{\prime\prime}_{i}|=\frac{1}{5}|J^{\prime}_{i}|$. Then by Lemma \ref{effective2024.07.116}, we have   
\begin{align}
 \sum_{J^{\prime}_{i}\subset J^{\prime}}|J^{\prime\prime}_{i}|  &  =\frac{1}{10} \varkappa_{\vartheta}(e^{t})^{\xi^{2}}|J^{\prime}|\;\nonumber\\
 & \geq\frac{1}{20} \varkappa_{\vartheta}(e^{t})^{\xi^{2}}|J_{1}^{\star}|\geq\frac{1}{40} \varkappa_{\vartheta}(e^{t})^{\xi^{2}}|2J_{1}^{\star}|\geq\frac{1}{40}|I^{\star}|.
  \;  \label{effective2024.07.119}
\end{align}
   Finally, combining (\ref{effective2024.07.117})(\ref{effective2024.07.118})(\ref{effective2024.07.119}), there exists $i=0,1,\ldots,\lfloor\varkappa_{\vartheta}(e^{t})^{\xi^{2}}|J_{1}^{\star}|\rfloor$, and $r^{\prime}\in J^{\prime\prime}_{i}-  \varkappa_{\vartheta}(e^{t})^{\xi^{2}}e^{t} r_{0}$ so that the point $u_{r^{\prime}}a_{1}(t)^{-1}\Lambda(x_{0})$ is  of the first kind in Theorem \ref{effective2024.07.107} (by replacing $a_{t}$ with $a_{1}(t)$). 
   
   We apply Theorem \ref{effective2024.07.107} to
   $u_{r^{\prime}}a_{1}(t)^{-1}\Lambda(x_{0})$:
   \begin{itemize}     
   \item Let  $J_{1}^{\star\star}=J^{\prime}_{i}\subset J_{1}^{\star}$ (see (\ref{effective2024.10.11})). 
  \item   Let $\varphi_{\varrho,d}$ be a function so that
  \[\chi_{\mathsf{B}_{G\times G}(\frac{9}{10}\varrho).(Q_{D})_{\eta}}\leq \varphi_{\varrho,d}\leq \chi_{\mathsf{B}_{G\times G}(\varrho).(Q_{D})_{\eta}}\]  
  with $\mathcal{S}(\varphi_{\varrho,d})\leq \varrho^{-N}\leq \varkappa_{\vartheta}(e^{t})^{\frac{\hyperlink{2024.08.k9}{\kappa_{9}}\delta}{2}\xi^{2}}$, where $N$ is absolute. In particular, we have 
   \[\int \varphi_{\varrho,d}dm_{X}\asymp    \varrho^{3}.\]
\end{itemize} 
Then by the equidistribution, we see that
     \[\left|\int_{0}^{1}\varphi_{\varrho,d} (a_{1}(t)u_{r^{\prime\prime}}\cdot u_{r^{\prime}}a_{1}(t)^{-1}\Lambda(x_{0}))dr^{\prime\prime}-\varrho^{3}\right|\leq (\varkappa_{\vartheta}(e^{t}))^{\frac{\hyperlink{2024.08.k9}{\kappa_{9}}\delta}{2}\xi^{2}} \asymp\varrho^{N}.\]
  For large enough $\varrho$,  we have   
\[ \varrho^{3}-   \varkappa_{\vartheta}(e^{t})^{\frac{\hyperlink{2024.08.k9}{\kappa_{9}}\delta}{2}\xi^{2}}\ll \int_{0}^{1}\varphi_{\varrho,d} (a_{1}(t)u_{r^{\prime\prime}}\cdot u_{r^{\prime}}a_{1}(t)^{-1}\Lambda(x_{0}))dr^{\prime\prime}\ll \varrho^{3}+ \varkappa_{\vartheta}(e^{t})^{\frac{\hyperlink{2024.08.k9}{\kappa_{9}}\delta}{2}\xi^{2}}.\] 
  Thus, there is a subset $J_{d}^{\prime(1)}\subset [0,e^{t}]$ with $|J_{d}^{\prime(1)}|\asymp  \varkappa_{\vartheta}(e^{t})^{-\xi^{2}}   \varrho^{3}$  such that  
  \[u_{\varkappa_{\vartheta}(e^{t})^{-\xi^{2}}  (r^{\prime\prime}+r^{\prime})}  \Lambda(x_{0})\in\supp(\varphi_{\varrho,d})\] for some $\varkappa_{\vartheta}(e^{t})^{-\xi^{2}}  (r^{\prime\prime}+r^{\prime})\in J_{d}^{\prime(1)}$. In other words, for 
  \[J_{d}^{\prime(1)}=\left\{ r\in J_{1}^{\star\star}: \|\tilde{\Lambda}-u_{r}\Lambda(x_{0})\|\leq  \varrho,\ \tilde{\Lambda}\in (Q_{D})_{\eta^{2}}\right\},\] 
  we have 
\begin{equation}\label{effective2024.10.7}
|J_{d}^{\prime(1)}|\asymp  \varrho^{3} \varkappa_{\vartheta}(e^{t})^{-\xi^{2}}.
\end{equation}  
 This establishes the proposition.
   \end{proof}
   
   \begin{proof}[Proof of Proposition \ref{effective2024.07.127}]
    We now   apply the above argument inductively. By Corollary \ref{effective2024.07.68}, Lemma \ref{effective2024.07.116} and Proposition \ref{effective2024.07.123}, we conclude that 
   \begin{itemize}
   \item There is a large $D_{3}=D_{3}(D_{1})>0$ so that $D_{3}\nearrow\infty$ as   $D_{1}\nearrow\infty$.
     \item  There exists  an interval $J_{1}^{\star\star}\subset   I_{f^{(1)}}(\psi(f^{(1)}))\setminus I_{f^{(1)}}(\phi(f^{(1)}))$  with
          \[|J_{1}^{\star\star}|=\varkappa_{\vartheta}(e^{t})^{-\xi^{2}}\] such that   for any  $g\in F\setminus\{f^{(1)}\}$ with $ I_{g}(\phi(g))\cap J_{1}^{\star\star}\neq\emptyset$, we have
         \begin{equation}\label{effective2024.07.124}
            |I_{g}(\phi(g))\cap J_{1}^{\star\star}|< \frac{1}{6}(2D_{3})^{-1}|J_{1}^{\star\star}|.
         \end{equation} 
 \item We have
     \begin{equation}\label{effective2024.07.125}
     \left|J_{d}^{\prime(1)}\right|\gg     \varrho^{3}|J_{1}^{\star\star}|.
     \end{equation}
     where $x_{0}^{\prime}\in u_{J_{1}^{\star\star}}x_{0}$ and
     \[J_{d}^{\prime(1)}=\left\{ r\in J_{1}^{\star\star}: \|\tilde{\Lambda}-u_{r}\Lambda(x_{0})\|\leq   \varrho,\ \tilde{\Lambda}\in (Q_{D})_{\eta^{2}}\right\}.\] 
   \end{itemize}

 Now let $\xi=\frac{1}{\hyperlink{2024.08.k7}{\kappa_{7}}+\vartheta}$.  Suppose that   
 \begin{equation}\label{effective2024.10.9}
\ell(x_{1})>\varkappa_{\vartheta}(|J_{1}^{\star\star}|)^{\xi} 
 \end{equation}   
  for any  $R_{1}\in   J_{d}^{\prime(1)}  -e^{t}r_{0}$. 
If  for any  $R_{1}\in   J_{d}^{\prime(1)}  -e^{t}r_{0}$, and $\tilde{\Lambda}_{d}\in  (Q_{D})_{\eta^{2}}$, we have 
\[   \ell(u_{R_{1}}\tilde{x}_{0})^{\vartheta} \|\tilde{\Lambda}_{d}-u_{R_{1}}\tilde{x}_{0}\|_{u_{R_{1}}\tilde{x}_{0}}\geq \varkappa_{\vartheta}(|J_{1}^{\star\star}|) ,  \] 
then we obtain (i) of Proposition \ref{effective2024.07.127}. 

  Now suppose that  there is a   $R_{1}\in   J_{d}^{\prime(1)}  -e^{t}r_{0}$, and $\tilde{\Lambda}_{d}\in  (Q_{D})_{\eta^{2}}$, such that  
  for $x_{1}= u_{R_{1}}x_{0}$, we have 
\begin{equation}\label{effective2024.10.8}
 \ell(\tilde{x}_{1})^{\vartheta}\|\tilde{x}_{1}-\tilde{\Lambda}_{d}\|_{\tilde{x}_{1}}  <\varkappa_{\vartheta}(|J_{1}^{\star\star}|)=\varkappa_{\vartheta}(\varkappa_{\vartheta}(e^{t})^{-\xi^{2}}).
\end{equation}

 Then, we can follow the same idea as in (\ref{effective2024.07.105}). One may find a surface $y_{d}^{(1)}\in (\Omega_{1}W_{D})_{\eta^{2}}$ with the absolute periods $\tilde{\Lambda}_{d}$, so that   
    \[\|y_{d}^{(1)}-x_{1}\|_{x_{1}}<\ell(\tilde{x}_{1})^{-\vartheta} \varkappa_{\vartheta}(|J_{1}^{\star\star}|) <\ell(x_{1})^{\hyperlink{2024.08.k7}{\kappa_{7}}}. \] 
    We obtain (ii) of Proposition \ref{effective2024.07.127}.

 Thus,  assume that there exists $r_{1} \in  e^{-t} J_{d}^{\prime(1)}  -r_{0}$ such that  
  for $x_{1}\coloneqq a_{t}u_{r_{0}+r_{1}}x$,
 \begin{equation}\label{effective2024.07.122}
  \ell(x_{1})<\varkappa_{\vartheta}(|J_{1}^{\star\star}|)^{ \xi}=\varkappa_{\vartheta}(\varkappa_{\vartheta}(e^{t})^{-\xi^{2}})^{ \xi}.
 \end{equation} 
   Then by an  argument similar to (\ref{effective2024.07.69}), one shows that 
   \[|J_{C^{2}}(e^{t}r_{0}+R_{1})|\geq   \varkappa_{\vartheta}(|J_{1}^{\star\star}|)^{-\xi^{2}}= \varkappa_{\vartheta}(\varkappa_{\vartheta}(e^{t})^{-\xi^{2}})^{-\xi^{2}}.\]
 By   (\ref{effective2024.07.124}), one can further deduce that $J_{C^{2}}(e^{t}(r_{0}+r_{1}))\subset\frac{1}{3}J^{\star\star}_{1}$.
   
   Now we set $[0,e^{t}],x_{0},D_{1},\eta,\Lambda_{D},F, e^{t}r_{0}$ in the statement of this section to $J^{\star\star}_{1},x_{1}$, $D_{3}$, $\eta^{2}$, $\Lambda_{D}^{(1)}$, $F\setminus\{f^{(1)}\}$, $e^{t}r_{0}+R_{1}$ respectively. This completes an inductive step.

Note that the induction will stop in finite steps. More precisely, suppose that the inequality (\ref{effective2024.07.122}) keep occurring. The process will stop after at most $R^{2}$ times. In the end, we have 
  \[J_{R^{2}}^{\star\star}\subset \bigcap_{i=1}^{R^{2}}I_{f^{(i)}}(\psi(f^{(i)}))\setminus \bigcup_{i=1}^{R^{2}}I_{f^{(i)}}(\phi(f^{(i)})).\]
  In particular, by Theorem \ref{effective2024.07.11}, we have $J_{R^{2}}^{\star\star}\cap \{r\in J_{R^{2}}^{\star\star}:\ell(u_{r}y)<C^{2}\}=\emptyset$. Then (\ref{effective2024.07.122}) cannot occur anymore. To make sure the induction can proceed, we require that sufficiently large $D_{1}$  is sufficiently large (and so sufficiently small $C$), so that Lemma \ref{effective2024.07.27} can always be used.  
  Thus, by (\ref{effective2024.07.125})(\ref{effective2024.10.9}), we conclude that  there exists some $k=0,1,\ldots,R^{2}$, 
times $r_{1},\ldots,r_{k-1}$, a surface  $x_{k-1}=a_{t}u_{r_{0}+r_{1}+\cdots+r_{k-1}}x\in \mathcal{H}_{1}(2)$  such that   
 \begin{equation}\label{effective2024.10.13}
\ell(u_{R_{k}}x_{k-1})\geq \varkappa_{\vartheta}(|J_{k}^{\star\star}|)^{\xi},
 \end{equation}    
 for any $R_{k}\in  J_{d}^{\prime(k)}  -e^{t}(r_{0}+r_{1}+\cdots+r_{k-1})$ and 
     \begin{equation}\label{effective2024.10.10}
     \left|J_{d}^{\prime(k)}\right|\gg   \varrho^{3} |J_{k}^{\star\star}| ,
     \end{equation}
     where
     \[J_{d}^{\prime(k)}=\left\{ r\in J_{k}^{\star\star}: \|\tilde{\Lambda}-u_{r}\Lambda(x_{k-1})\|\leq  \varrho,\ \tilde{\Lambda}\in (Q_{D})_{\eta^{k+1}}\right\}.\]

  Now let  $J^{\star\star}=J_{k}^{\star\star}$.   If  for any  $R_{k}\in   J_{d}^{\prime(k)}  -e^{t}(r_{0}+r_{1}+\cdots+r_{k-1})$, and $\tilde{\Lambda}_{d}\in  (Q_{D})_{\eta^{2}}$, we have 
\[   \ell(u_{R_{k}}\tilde{x}_{k-1})^{\vartheta} \|\tilde{\Lambda}_{d}-u_{R_{k}}\tilde{x}_{k-1}\|_{u_{R_{k}}\tilde{x}_{k-1}}\geq \varkappa_{\vartheta}(|J^{\star\star}|) ,  \] 
then we obtain (i) of Proposition \ref{effective2024.07.127}. 

  Now suppose that  there is a   $r_{k}\in   e^{-t}J_{d}^{\prime(k)}  - (r_{0}+r_{1}+\cdots+r_{k-1})$, and $\tilde{\Lambda}_{d}\in  (Q_{D})_{\eta^{2}}$, such  that for  $x_{k}\coloneqq a_{t}u_{r_{0}+r_{1}+\cdots+r_{k}}x$, we have
\[\ell(\tilde{x}_{k})^{\vartheta}\|\tilde{x}_{k}-\tilde{\Lambda}_{d}\|_{\tilde{x}_{k}}  <\varkappa_{\vartheta}(|J^{\star\star}|). \] 
    This guarantees the existence of a surface $y_{d}^{(k)}\in (\Omega_{1}W_{D})_{\eta^{k}}$ with the absolute periods $\tilde{\Lambda}_{d}^{(k)}$  (by (\ref{effective2024.07.146})) such that 
  \[\ell(\tilde{x}_{k})^{-\vartheta}\|\tilde{x}_{k}-y_{d}^{(k)}\|_{\tilde{x}_{k}}< \varkappa_{\vartheta}(|J_{k}^{\star\star}|)\leq \ell(x_{k})^{\hyperlink{2024.08.k7}{\kappa_{7}}}.\] 
  where $ \varkappa_{\vartheta}^{\star}$ is given by
  \[\varkappa_{\vartheta}^{\star}:r\mapsto\underbrace{\varkappa_{\vartheta}( \cdots \varkappa_{\vartheta}(\varkappa_{\vartheta}(r)^{-\xi^{2}})^{-\xi^{2}}\cdots ) ^{\xi^{2}}  }_{R^{2} \text{ copies}}.\]
 Now Proposition \ref{effective2024.07.127} follows from letting $r^{\star}=r_{0}+r_{1}+\cdots+r_{k-1}$, $x^{\star}=x_{k}$, and $y_{d}^{\star}=y_{d}^{(k)}$. 
   \end{proof}

\section{A criterion for Teichm\"{u}ller curves}\label{effective2024.08.11}   
  \subsection{A criterion for the Teichm\"{u}ller curves in $\mathcal{H}(2)$}
  In this section, we generate a class of Teichm\"{u}ller curves via $G$-closed absolute periods and additional rationality.
  \begin{prop}\label{effective2023.9.9}
     Let   $x^{\prime\prime}\in\mathcal{H}(2)$ with an algebraic sum  $x^{\prime\prime}=\Lambda_{1}^{\prime\prime}+\Lambda_{2}^{\prime\prime}$  that satisfies $\overline{G.(\Lambda_{1}^{\prime\prime},\Lambda_{2}^{\prime\prime})}=G.(\Lambda_{1}^{\prime\prime},\Lambda_{2}^{\prime\prime})$. Then there exists a square-free integer $k\in\mathbb{N}$ depending only on $(\Lambda_{1}^{\prime\prime},\Lambda_{2}^{\prime\prime})$ (regardless the areas of tori), such that if the areas of tori satisfy
 \begin{equation}\label{effective2023.9.10}
   \sqrt{\frac{\Area(\Lambda_{1}^{\prime\prime})}{\Area(\Lambda_{2}^{\prime\prime})}}\cdot \sqrt{k}\in\mathbb{Q},
 \end{equation}
 then $x^{\prime\prime}$ generates a Teichm\"{u}ller curve. 
  \end{prop}

  We start from the following lemma, which is a conclusion from \textit{Ratner theorem}\index{Ratner theorem}:

\begin{lem}\label{closing2023.07.7} Let $(\Lambda_{1},\Lambda_{2})\in G/\Gamma\times G/\Gamma$ be a point with $\overline{G.(\Lambda_{1},\Lambda_{2})}=G.(\Lambda_{1},\Lambda_{2})$.
Suppose that   $v\in\mathbb{C}^{\ast}$ is a nonzero vector such that 
\[ [0,v]\cap\Lambda_{2}=\{0,v\}.\]
Then there exists a nonzero vector $w\in\mathbb{C}^{\ast}$ such that 
\[ [0,w]\cap\Lambda_{1}=\{0,w\}\]
and $w=rv$ for some $r^{2}\in \frac{\Area(\Lambda_{1})}{\Area(\Lambda_{2})}\mathbb{Q}$.
\end{lem}
\begin{proof} By replacing $\Lambda_{2}$ with a rescaling  $\sqrt{\frac{\Area(\Lambda_{1})}{\Area(\Lambda_{2})}}\cdot \Lambda_{2}$ if necessary, we may assume that $\Lambda_{1},\Lambda_{2}\subset\mathbb{C}$ have the same area.
Let $U_{v}<G$ be the   unipotent subgroup stabilizing   $v$. Let $\SL(\Lambda_{2})\coloneqq\{g\in G:g\Lambda_{2}=\Lambda_{2}\}$ be the stabilizer of $\Lambda_{2}$.
    Let $U_{v}^{\mathbb{Z}}\coloneqq U_{v}\cap \SL(\Lambda_{2})\cong\mathbb{Z}$. It follows that 
     \[U_{v}^{\mathbb{Z}}.(\Lambda_{1},\Lambda_{2})=(U_{v}^{\mathbb{Z}}\Lambda_{1},\Lambda_{2}).\]
     Let $Z_{1}\coloneqq\overline{U_{v}^{\mathbb{Z}}\Lambda_{1}}\subset G/\Gamma$ be the  orbit closure of the first coordinate. Then by Ratner theorem, we   have $Z_{1}=H\Lambda_{1}$ where $H=U_{v}^{\mathbb{Z}}$, $U_{v}$ or $G$.
    
    Assume first that   no vectors $w\in\Lambda_{1}$ parallel to $v$. It means that $H=G$, and so 
    \[(G\Lambda_{1},\Lambda_{2})=\overline{U_{v}^{\mathbb{Z}}.(\Lambda_{1},\Lambda_{2})}\subset G.(\Lambda_{1},\Lambda_{2}).\]
    This leads to a contradiction since $G.(\Lambda_{1},\Lambda_{2})$ is given by diagonal $G$-action. 
    
    Thus, we see that there exist vectors of $\Lambda_{1}$ parallel to $v$.     
     Let $w\in\mathbb{C}^{\ast}$ be one of them with shortest length. It follows that $w=rv$ for some $r\in\mathbb{R}$. Assume that $r^{2}$ is irrational. Then we have $H=U_{v}$ and so 
         \[(U_{v}\Lambda_{1},\Lambda_{2})=\overline{U_{v}^{\mathbb{Z}}.(\Lambda_{1},\Lambda_{2})}\subset G.(\Lambda_{1},\Lambda_{2}).\]
         It is again a contradiction. Therefore, we conclude that $r^{2}\in\mathbb{Q}$.
\end{proof}

Now assume that we have a surface  $x^{\prime\prime}\in\mathcal{H}(2)$ with an algebraic sum $x^{\prime\prime}=\Lambda_{1}^{\prime\prime}+\Lambda_{2}^{\prime\prime}$  that satisfies
\[ \overline{G.(\Lambda_{1}^{\prime\prime},\Lambda_{2}^{\prime\prime})}=G.(\Lambda_{1}^{\prime\prime},\Lambda_{2}^{\prime\prime}).\]
Now let $v^{\star}\in \Lambda_{2}^{\prime\prime}$ be a primitive vector, i.e. 
\[[0,v^{\star}]\cap\Lambda_{2}^{\prime\prime}=\{0,v^{\star}\}.\]
Denote $h\coloneqq\sqrt{\frac{\Area(\Lambda_{1}^{\prime\prime})}{\Area(\Lambda_{2}^{\prime\prime})}}$. By Lemma \ref{closing2023.07.7}, we have $r=hr^{\prime}>0$ with $r^{\prime 2}\in\mathbb{Q}$ so that 
\[[0,rv^{\star}]\cap\Lambda_{1}^{\prime\prime}=\{0,rv^{\star}\}.\]
Then we can write 
\begin{equation}\label{effective2023.9.5}
  \Lambda_{1}^{\prime\prime}=\mathbb{Z}rv^{\star}\oplus\mathbb{Z}v, \ \ \ \Lambda_{2}^{\prime\prime}=\mathbb{Z}v^{\star}\oplus\mathbb{Z}w.
\end{equation} 
Now applying Lemma \ref{closing2023.07.7} to along $v$ and $w$ directions, we get that 
\begin{equation}\label{effective2023.9.1}
  tw=mrv^{\star}+nv,\ \ \  sv=iv^{\star}+jw,
\end{equation} 
 where $i,j,m,n\in\mathbb{Z}$, $t=ht^{\prime}$, $s=h^{-1}s^{\prime}$   with $t^{\prime 2},s^{\prime 2}\in\mathbb{Q}$. After counting the areas of $\Lambda_{1}^{\prime\prime}$ and $\Lambda_{2}^{\prime\prime}$, we  obtain from (\ref{effective2023.9.5}) and (\ref{effective2023.9.1}) that 
 \begin{equation}\label{effective2023.9.7}
   s^{\prime}=jr^{\prime},\ \ \ n=r^{\prime}t^{\prime}.
 \end{equation} 
 
 Now assume that the areas of tori satisfy 
 \begin{equation}\label{effective2023.9.6}
   r=hr^{\prime}\in\mathbb{Q}.
 \end{equation} 
 Then via (\ref{effective2023.9.7}), we get that 
 \begin{equation}\label{effective2023.9.8}
   t,s\in \mathbb{Q}.
 \end{equation} 
 
In what follows, we shall show that  $x^{\prime\prime}$ generates a Teichm\"{u}ller curve under the assumption (\ref{effective2023.9.6}). By replacing $x^{\prime\prime}$  with $gx^{\prime\prime}$ if necessary, we may further assume that $v$ and $v^{\star}$ are orthogonal.

Next, let us recall that  an \textit{isogeny}\index{isogeny} between a pair of elliptic curves is a surjective holomorphic map $p:E_{1}\rightarrow E_{2}$. A pair of $1$-forms $(E_{i},\omega_{i})\in\Omega\mathcal{M}_{1}$ are said to be \textit{isogenous}\index{isogenous} if there is a $\tau>0$ and an isogenous $p:E_{1}\rightarrow E_{2}$ such that $p^{\ast}(\omega_{2})=\tau\omega_{1}$. This is equivalent to the condition that $\tau\Lambda_{1}\subset \Lambda_{2}$, where $\Lambda_{i}\subset\mathbb{C}$ is the period lattice of $(E_{i},\omega_{i})$.  
In order to classify the $\SL_{2}(\mathbb{R})$-orbit  closures of $\mathcal{H}_{1}(2)$, McMullen \cite{mcmullen2007dynamics} showed the following:
  \begin{thm}\label{closing2023.07.22} Let $x\in \mathcal{H}(2)$ that can be presented, in more than one way, as an algebraic sum
  \[x\cong(E_{1},\omega_{1})+(E_{2},\omega_{2})\]
  of isogenous forms of genus $1$. Then $x$ generates a Teichm\"{u}ller curve.
  \end{thm}
  \begin{proof}
    This is just a restatement  of  \cite[Theorem 5.10 \& 6.1]{mcmullen2007dynamics}.  
  \end{proof}
  Thus, we want to find two different algebraic sums of $x^{\prime\prime}$ to satisfy the requirements. Now   by  (\ref{effective2023.9.1}),  for any $\tau>0$, we have 
  \[\tau \cdot (rv^{\star},v)=(v^{\star},w)\left[
            \begin{array}{cccc}
   \tau r & \tau i/s  \\
   0 & \tau j/s    \\
            \end{array}
          \right].\]
 Thanks to the fact that $r,s,i,j\in\mathbb{Q}$, we can choose certain $\tau>0$ so that the $2\times 2$ matrix on the right-hand side becomes an integer matrix. Then we conclude that $\tau\Lambda_{1}^{\prime\prime}\subset\Lambda_{2}^{\prime\prime}$. 
   Thus, the algebraic sum
   \[x\cong(\mathbb{C}/\Lambda_{1}^{\prime\prime},dz)+(\mathbb{C}/\Lambda_{2}^{\prime\prime},dz)\]
 is a pair of isogenous forms that meets our needs.  
   
   To find a different splitting,   we consider a factor mix:
   \begin{lem} Let the notation and assumptions be as above.    Then  
   \[ T_{1}=  \mathbb{Z}(rv^{\star}+w)\oplus \mathbb{Z}v,\ \ \ T_{2}=  \mathbb{Z}(v^{\star}+v)\oplus \mathbb{Z}w\]    
also define an algebraic sum $x^{\prime\prime}=T_{1}+T_{2}$.
   \end{lem}
\begin{proof} Let $\{\alpha_{1},\beta_{1},\alpha_{2},\beta_{2}\}$ span $\mathbb{Z}^{4}$. Then the map
\[(\alpha_{1},\beta_{1},\alpha_{2},\beta_{2})\mapsto(\alpha_{1}+\beta_{2},\beta_{1},\alpha_{2}+\beta_{1},\beta_{2})\]
  is an element in $\Sp(4,\mathbb{Z})$ (in fact, it is a Burkhardt generator of $\Sp(4,\mathbb{Z})$). Thus, the new algebraic sum is just a shift of the old one by a mapping class.
\end{proof}
\begin{proof}[Proof of Proposition \ref{effective2023.9.9}]
We want that $T_{1}$ and $T_{2}$ are isogenous. First, one calculates 
    \begin{align}
 & \tau \cdot (rv^{\star}+w,v)\;\nonumber\\
=&  (v^{\star}+v,w)\left[
            \begin{array}{cccc}
   \tau rs/(s+i) & \tau mr/(mr-n)  \\
   \tau(1-rj/(s+i)) & \tau t/(n-mr)    \\
            \end{array}
          \right]\;  \nonumber
\end{align}
 for any $\tau>0$. (Note that since $w\not\in\mathbb{Z}(v^{\star}+v)$, for otherwise $T_{2}$ is not well-defined, we know from (\ref{effective2023.9.1}) that the denominators in the entries of the $2\times 2$ matrix are nonzero.) Again, thanks to the fact that $r,s,t,i,j,m,n\in\mathbb{Q}$, we can choose certain $\tau>0$ so that the $2\times 2$ matrix on the right-hand side becomes an integer matrix. 
   Thus, $\tau T_{1}\subset T_{2}$ and the algebraic sum
   \[x^{\prime\prime}\cong(\mathbb{C}/T_{1},dz)+(\mathbb{C}/T_{2},dz)\]
is a pair of isogenous forms that meets our needs.  Then, by Theorem \ref{closing2023.07.22}, we conclude that    $x^{\prime\prime}$ generates a Teichm\"{u}ller curve. In other words, $\SL(x^{\prime\prime})$ is a lattice and 
  \[G.x^{\prime\prime}\cong G/\SL(x^{\prime\prime}).\]
  Finally, since $r^{\prime 2}\in\mathbb{Q}$, one can write 
  \[r^{\prime}=\frac{p}{q}\sqrt{k}\]
  for some $p,q,k\in\mathbb{N}$ with $k$ square-free. Thus, (\ref{effective2023.9.6}) reduces to (\ref{effective2023.9.10}).
  Therefore, we finish the proof of Proposition \ref{effective2023.9.9}.
\end{proof}

  \subsection{Effective estimates of  Teichm\"{u}ller curves}
Now    let   $x^{\prime\prime}\in\mathcal{H}(2)$ with an algebraic sum $x^{\prime\prime}=\Lambda_{1}^{\prime\prime}+\Lambda_{2}^{\prime\prime}$ that satisfies 
\[\overline{G.(\Lambda_{1}^{\prime\prime},\Lambda_{2}^{\prime\prime})}=G.(\Lambda_{1}^{\prime\prime},\Lambda_{2}^{\prime\prime}) \ \ \ \text{ and }\ \ \ \sqrt{\frac{\Area(\Lambda_{1}^{\prime\prime})}{\Area(\Lambda_{2}^{\prime\prime})}}\cdot \sqrt{k}\in\mathbb{Q},\] 
where $k$ is the square-free integer provided by Proposition \ref{effective2023.9.9}. It follows that $x^{\prime\prime}$ generates a Teichm\"{u}ller curve. In other words,  $\omega$ is an eigenform for real multiplication  by $\mathcal{O}_{D}$ (cf. \cite[Theorem 5.10 \& 4.1]{mcmullen2007dynamics}). 
 Thus, by Theorem \ref{effective2023.9.12}, it indicates that $(\mathbb{C}/\Lambda_{1}^{\prime\prime}\times\mathbb{C}/\Lambda_{2}^{\prime\prime},(dz)_{1}+(dz)_{2})\in \Omega Q_{D}$ is an eigenform with discriminant $D$. Now suppose that $(\mathbb{C}/\Lambda_{1}^{\prime\prime}\times\mathbb{C}/\Lambda_{2}^{\prime\prime},(dz)_{1}+(dz)_{2})\in \Omega Q_{D}$ is equivalent to the prototypical example of type $(e,\ell,m)$. A direct calculation shows the following:
\begin{lem}\label{effective2023.9.15}
   Let the notation and assumptions be as in Proposition \ref{effective2023.9.9}. Write 
   \[\sqrt{\frac{\Area(\Lambda_{1}^{\prime\prime})}{\Area(\Lambda_{2}^{\prime\prime})}}=\frac{p}{q}\sqrt{k}.\]
   Then prototype $(e,\ell,m)$ that attaches to $x^{\prime\prime}$ satisfies either 
   \[e=0,\ \ \ \ell=1,\ \ \ D=m/4\]
   or $D$ is a square and
\begin{equation}\label{effective2023.9.26}
  k|m,\ \ \ \ell|qp,\ \ \ e^{2}| m(q^{2}-p^{2}k)^{2}.
\end{equation} 
\end{lem}
\begin{proof} First, by (\ref{effective2023.9.13}), we get that
\[\frac{p}{q}\sqrt{k}=\sqrt{\frac{\Area(\Lambda_{1})}{\Area(\Lambda_{2})}}=\sqrt{\frac{\ell^{2}m}{\lambda^{2}}}=\sqrt{\frac{D-e^{2}}{(\sqrt{D}+e)^{2}}}=\sqrt{\frac{\sqrt{D}-e}{\sqrt{D}+e}}.\]
If $\frac{p}{q}\sqrt{k}=1$, then $e=0$ and so $\ell=1$ and $D=m/4$. If  $\frac{p}{q}\sqrt{k}\neq1$, one then solves from the above equation that
\[D=e^{2}\frac{(q^{2}+p^{2}k)^{2}}{(q^{2}-p^{2}k)^{2}}.\]
In particular, $D$ is a square.
Because $D=e^{2}+4\ell^{2}m$, we obtain
\[ \ell^{2}m\cdot(q^{2}-p^{2}k)^{2}=e^{2}\cdot q^{2}p^{2}k.\] 
In particular, $\ell^{2}| e^{2} q^{2}p^{2}$, $e^{2}|\ell^{2}m(q^{2}-p^{2}k)^{2}$ and 
\[k|m\]
 as $k$ is square-free.   On the other hand, since by (\ref{effective2023.9.14}), $\gcd(e,\ell)=1$, we get that 
 \[\ell|qp\ \ \ \text{ and }\ \ \ e^{2}| m(q^{2}-p^{2}k)^{2}.\]
 The consequence follows.
\end{proof}

Given $x^{\prime\prime}\in \mathcal{H}_{1}(2)$ with an algebraic sum $x^{\prime\prime}=\Lambda_{1}^{\prime\prime}+\Lambda_{2}^{\prime\prime}$. Then for   $\epsilon>0$, we change the area of $x^{\prime\prime}$ slightly and consider the surfaces with algebraic sums:
\begin{equation}\label{effective2024.07.160}
 x^{\prime\prime}(\epsilon)\coloneqq(1+\epsilon)^{\frac{1}{2}}\Lambda_{1}^{\prime\prime}+ (1+\epsilon)^{-\frac{1}{2}}\Lambda_{2}^{\prime\prime}
\end{equation} 
    (cf. (\ref{effective2024.07.136})).   Note that $x^{\prime\prime}(\epsilon)$ is well defined if $\epsilon$ is sufficiently small by considering the period coordinates.
\begin{lem}\label{effective2024.07.05}  \hypertarget{2024.08.C13} 
  Let $T>0$.    Let $x^{\prime\prime}\in \mathcal{H}_{1}(2)$ with an algebraic sum $x^{\prime\prime}=\Lambda_{1}^{\prime\prime}+\Lambda_{2}^{\prime\prime}$. Suppose that 
   \begin{equation}\label{effective2023.9.17}
  \Area(\Lambda_{1}^{\prime\prime})=\Area(\Lambda_{2}^{\prime\prime}),  \ \ \  \overline{G.(\Lambda_{1}^{\prime\prime},\Lambda_{2}^{\prime\prime})}=G.(\Lambda_{1}^{\prime\prime},\Lambda_{2}^{\prime\prime}),\ \ \ \vol(G.(\Lambda_{1}^{\prime\prime},\Lambda_{2}^{\prime\prime}))<T.
   \end{equation} 
 Then    \hypertarget{2024.08.k14}  for any $\eta_{0}>0$, \hypertarget{2024.08.k15}  there    exists  $\epsilon\in[0,\eta_{0}]$, and $C_{13},\kappa_{14},\kappa_{15}>0$ such that if the surface $x^{\prime\prime}(\epsilon)\in\mathcal{H}_{1}(2)$ as in (\ref{effective2023.9.17})  is well defined, then $x^{\prime\prime}(\epsilon)$ generates a Teichm\"{u}ller curve with discriminant $D$, and that 
 \begin{itemize}
           \item  either $D< \hyperlink{2024.08.C13}{C_{13}}T^{\hyperlink{2024.08.k15}{\kappa_{15}}}$,
           \item or $D$ is a square and $D< \hyperlink{2024.08.C13}{C_{13}}\eta_{0}^{-\hyperlink{2024.08.k14}{\kappa_{14}}}T^{\hyperlink{2024.08.k15}{\kappa_{15}}}$.
         \end{itemize}
\end{lem}
\begin{proof} The idea is to change $\Area(\Lambda_{1})/\Area(\Lambda_{2})$ slightly to satisfy the rationality.

   Let $x^{\prime\prime}=\Lambda_{1}^{\prime\prime}+\Lambda_{2}^{\prime\prime}\in\mathcal{H}_{1}(2)$ satisfy (\ref{effective2023.9.17}). 
   Then by Proposition \ref{effective2023.9.9}, there exists a square-free integer $k\in\mathbb{N}$ corresponding to $(\Lambda_{1}^{\prime\prime},\Lambda_{2}^{\prime\prime})\in G/\Gamma\times G/\Gamma$.

   Next, consider $x^{\prime\prime}(\epsilon)$ as in (\ref{effective2024.07.160}). Choose    $0<\epsilon< \eta_{0}$, $p,q\in\mathbb{N}$ so that 
   \[ \sqrt{\frac{\Area((1+\epsilon)^{\frac{1}{2}}\Lambda_{1}^{\prime\prime})}{\Area((1+\epsilon)^{-\frac{1}{2}}\Lambda_{2}^{\prime\prime})}}=\sqrt{\frac{(1+\epsilon)^{\frac{1}{2}}}{(1+\epsilon)^{-\frac{1}{2}}}} =1+\epsilon =\frac{p}{q}\sqrt{k}.\]
  We want to pick a small $q$. Thus,  by considering the rational numbers in $[\frac{1}{\sqrt{k}},\frac{1+\eta_{0}}{\sqrt{k}}]$, we can pick 
   \[\frac{1}{q}\geq\frac{1}{2}\cdot\left(\frac{1+\eta_{0}}{\sqrt{k}}-\frac{1}{\sqrt{k}}\right)= \frac{\eta_{0}}{2\sqrt{k}}\] 
    and so,
    \begin{equation}\label{effective2023.9.27}
      q<2\sqrt{k}  \eta_{0}^{-1}.
    \end{equation} 
    With the choice of $q$, one can estimate
     \begin{equation}\label{effective2023.9.28}
     p\leq (1+\epsilon) \frac{q}{\sqrt{k}}<4\eta_{0}^{-1}.
    \end{equation}  
    
    Under the assumption,  we  conclude from Proposition \ref{effective2023.9.9} that $x^{\prime\prime}(\epsilon)$  generates a Teichm\"{u}ller curve,  with some discriminant $D>0$. Then by Theorem \ref{effective2023.9.24}, we have 
    \[G.(\Lambda_{1}^{\prime\prime}(\delta),\Lambda_{2}^{\prime\prime}(\delta))=\Omega Q_{D}(e,\ell,m)/\mathbb{R}^{+}\cong G/\Gamma_{0}(m)\]
    for some $e,\ell,m$. Then the volume control (\ref{effective2023.9.17}) indicates that there exist  $c_{1},c_{2}>0$ such that 
    \begin{equation}\label{effective2023.9.25}
      m<c_{1} T^{c_{2}}.
    \end{equation} 
   If $D$ is not a square, by Lemma \ref{effective2023.9.15}, $D=m/4 \ll T^{c_{2}}$. Suppose $D$ is   a square. Then by (\ref{effective2023.9.26}), we immediately obtain $k<c_{1} T^{c_{2}}$.
    Moreover, combining (\ref{effective2023.9.26}) (\ref{effective2023.9.27}) (\ref{effective2023.9.28}) (\ref{effective2023.9.25}),  
    we can control  $e$ and $\ell$ by a polynomial of $\eta_{0}^{-1}$ and $T$. Thus, there exists $\hyperlink{2024.08.C13}{C_{13}},\hyperlink{2024.08.k14}{\kappa_{14}},\hyperlink{2024.08.k14}{\kappa_{15}}>0$ such that 
    \[D=e^{2}+4\ell^{2} m<\hyperlink{2024.08.C13}{C_{13}}\eta_{0}^{-\hyperlink{2024.08.k14}{\kappa_{14}}}T^{\hyperlink{2024.08.k14}{\kappa_{15}}}\]
    as what we needed. 
\end{proof}

Now we are in the position to prove Theorem \ref{effective2024.08.1}. 
 \begin{proof}[Proof of Theorem \ref{effective2024.08.1}] 
 Suppose that Theorem \ref{effective2024.07.159}(2) holds. Then by the proof of Theorem \ref{effective2024.07.159}. There exist $x^{\prime},x^{\prime\prime}\in\mathcal{H}_{1}(2)$ such that 
       \begin{itemize}
         \item  $d(x,x^{\prime})<e^{-100t}$, $d(x^{\prime},x^{\prime\prime})<e^{-\frac{1}{2}t}$, $\ell(x^{\prime})\geq\frac{1}{2}\ell(x)$,
         \item $x^{\prime}=\Lambda^{\prime}_{1}+\Lambda^{\prime}_{2}$ belongs to a Teichm\"{u}ller curve,
         \item  $x^{\prime\prime}=\Lambda^{\prime\prime}_{1}+\Lambda^{\prime\prime}_{2}$ satisfies $\Area(\Lambda^{\prime\prime}_{1})=\Area(\Lambda^{\prime\prime}_{2})$, and $G.(\Lambda^{\prime\prime}_{1},\Lambda^{\prime\prime}_{2})$ is periodic with $\vol(G.(\Lambda^{\prime\prime}_{1},\Lambda^{\prime\prime}_{2}))\leq e^{\delta t}$.
       \end{itemize} 
Note in particular that $\ell(\Lambda^{\prime}_{1},\Lambda^{\prime}_{2})>\frac{1}{2}\ell(x)$, and that by considering the volumes, there are at most $O(\ell(x)^{-3})$ different surfaces $y$ can be presented as  $\Lambda^{\prime}_{1}+\Lambda^{\prime}_{2}$. One may deduce that $x^{\prime}$ has a splitting $x^{\prime} =\Lambda_{1}^{\prime}\stackrel[I]{}{\#} \Lambda_{2}^{\prime}$  so that the lengths of the sides of parallelograms are controlled by some power of $\ell(x)$. Then by Lemma \ref{effective2024.07.134} (and Remark \ref{effective2024.08.2}), we conclude that there exists $\kappa_{16}>0$ such that 
\[\|x^{\prime}(\epsilon)-x^{\prime}\|_{x}\leq  \ell(x)^{-\kappa_{16}}|\epsilon|\]
for any $\epsilon\in[0,\frac{1}{2}\ell(x)]$. Now since $x_{1}^{\prime\prime}=gx_{1}^{\prime}$ for some $g\in B_{G}(e^{-\frac{1}{2}t})$, by Lemma \ref{effective2024.07.145}, we have
\begin{equation}\label{effective2024.08.3}
  \|x^{\prime\prime}(\epsilon)-x^{\prime\prime}\|_{x}\leq  \ell(x)^{-\kappa_{16}}|\epsilon|
\end{equation} 
for any $\epsilon\in[0,\frac{1}{2}\ell(x)]$. In particular, (\ref{effective2024.08.3}) implies that $x^{\prime\prime}(\epsilon)$ is well defined for $\epsilon\in[0,\frac{1}{2}\ell(x)^{\kappa_{16}+\hyperlink{2024.08.k7}{\kappa_{7}}}]$.

Finally, we apply Lemma \ref{effective2024.07.05}, setting $x^{\prime\prime}$, $T$, $\eta_{0}$ in the statement of the lemma equal to $x^{\prime\prime}$, $e^{\delta t}$, $\ell(x)^{\kappa_{16}}e^{-t}$ respectively. Then $x^{\prime\prime}(\epsilon)$ generates a Teichm\"{u}ller curve with discriminant $D$, and that 
 \begin{itemize}
           \item  either $D< \hyperlink{2024.08.C13}{C_{13}}e^{\hyperlink{2024.08.k15}{\kappa_{15}}\delta t}\leq  e^{2\hyperlink{2024.08.k15}{\kappa_{15}}\delta t}$,
           \item or $D$ is a square and $D< \hyperlink{2024.08.C13}{C_{13}}\ell(x)^{-\hyperlink{2024.08.k14}{\kappa_{14}}\kappa_{16}}e^{\hyperlink{2024.08.k14}{\kappa_{14}}t}e^{\hyperlink{2024.08.k15}{\kappa_{15}}\delta t}\leq  e^{2(\hyperlink{2024.08.k14}{\kappa_{14}}+\hyperlink{2024.08.k15}{\kappa_{15}})t}$.
         \end{itemize}
         (Here we use $e^{t}$ for sufficiently large $t$ to absorb the constants.) Theorem \ref{effective2024.08.1} now follows from letting  $x^{\prime\prime\prime}=x^{\prime\prime}(\epsilon)$, and $\hyperlink{2024.08.k2}{\kappa_{2}}=2(\hyperlink{2024.08.k14}{\kappa_{14}}+\hyperlink{2024.08.k15}{\kappa_{15}})$.
 \end{proof}

\bibliographystyle{alpha}
 \bibliography{text}
\end{document}